\documentclass{article}
\usepackage[letterpaper,margin=1.0in]{geometry}
\PassOptionsToPackage{numbers,compress,sort}{natbib}
\usepackage{amsmath}\allowdisplaybreaks
\usepackage{amsfonts,bm}
\usepackage{amssymb}

\def\eqref#1{equation~(\ref{#1})}
\def\1{\bf{1}}

\newcommand{\Norm}[1]{\left\| #1 \right\|}
\newcommand{\norm}[1]{\left\| #1 \right\|}
\def\inner#1#2{\langle #1, #2 \rangle}

\def\fC{{\mathcal{C}}}

\def\fF{{\mathcal{F}}}
\def\ffF{{\mathscr{F}}}

\def\fI{{\mathcal{I}}}

\def\fO{{\mathcal{O}}}
\def\fP{{\mathcal{P}}}

\def\fS{{\mathcal{S}}}

\def\fX{{\mathcal{X}}}
\def\fY{{\mathcal{Y}}}
\def\fZ{{\mathcal{Z}}}

\def\BE{{\mathbb{E}}}
\def\BN{{\mathbb{N}}}

\def\BR{{\mathbb{R}}}

\newcommand{\E}{\mathbb{E}}

\DeclareMathOperator*{\argmin}{arg\,min}

\usepackage{amsthm}
\theoremstyle{plain}
\newtheorem{thm}{Theorem}%[section]
\newtheorem{dfn}{Definition}

\newtheorem{lem}{Lemma}
\newtheorem{asm}{Assumption}

\newtheorem{cor}{Corollary}
\newtheorem{prop}{Proposition}

\makeatletter
\def\Ddots{\mathinner{\mkern1mu\raise\p@
\vbox{\kern7\p@\hbox{.}}\mkern2mu
\raise4\p@\hbox{.}\mkern2mu\raise7\p@\hbox{.}\mkern1mu}}
\makeatother

\makeatletter
\newcommand*{\rom}[1]{\expandafter\@slowromancap\romannumeral #1@}
\makeatother

\usepackage{thmtools}
\usepackage{thm-restate}

\usepackage[utf8]{inputenc} % allow utf-8 input
\usepackage[T1]{fontenc}    % use 8-bit T1 fonts
\usepackage{hyperref}       % hyperlinks
\usepackage{url}            % simple URL typesetting
\usepackage{booktabs}       % professional-quality tables
\usepackage{amsfonts}       % blackboard math symbols
\usepackage{nicefrac}       % compact symbols for 1/2, etc.
\usepackage{microtype}      % microtypography
\usepackage{xcolor}         % colors
\usepackage{multirow}
\usepackage{mathrsfs}
\usepackage{mathtools}
\usepackage{enumitem}
\usepackage[utf8]{inputenc}

\usepackage{algorithm}
\usepackage{algorithmic}
\usepackage{array}
\usepackage{graphicx}
\usepackage{hyperref}
\usepackage{url}
\usepackage{natbib} 
\usepackage{booktabs}

\usepackage{changepage}

\title{Near-Optimal Distributed Minimax Optimization under the Second-Order Similarity}
\author{Qihao Zhou\thanks{School of Data Science, Fudan University; 20307130233@fudan.edu.cn}
\qquad
Haishan Ye\thanks{School of Management, Xi’an Jiaotong University; hsye\_cs@outlook.com}
\qquad
Luo Luo\thanks{School of Data Science, Fudan University; luoluo@fudan.edu.cn}
}
\date{}

\begin{document}
\maketitle
\begin{adjustwidth}{0.28in}{0.28in}
\begin{abstract}
This paper considers the distributed convex-concave minimax optimization under the second-order similarity.
We propose stochastic variance-reduced optimistic gradient sliding (SVOGS) method, which takes the advantage of the finite-sum structure in the objective by involving the mini-batch client sampling and variance reduction.
We prove SVOGS can achieve the $\varepsilon$-duality gap within communication rounds of 
${\mathcal O}(\delta D^2/\varepsilon)$, 
communication complexity of ${\mathcal O}(n+\sqrt{n}\delta D^2/\varepsilon)$,
and local gradient calls of 
$\tilde{\mathcal O}(n+(\sqrt{n}\delta+L)D^2/\varepsilon\log(1/\varepsilon))$, 
where $n$ is the number of nodes, $\delta$ is the degree of the second-order similarity, $L$ is the smoothness parameter and $D$ is the diameter of the constraint set.
We can verify that all of above complexity (nearly) matches the corresponding lower bounds.
For the specific $\mu$-strongly-convex-$\mu$-strongly-convex case, 
our algorithm has the upper bounds on communication rounds,  communication complexity, and local gradient calls of $\mathcal O(\delta/\mu\log(1/\varepsilon))$, ${\mathcal O}((n+\sqrt{n}\delta/\mu)\log(1/\varepsilon))$, and $\tilde{\mathcal O}(n+(\sqrt{n}\delta+L)/\mu)\log(1/\varepsilon))$ respectively, which are also nearly tight.
Furthermore, we conduct the numerical experiments to show the empirical advantages of proposed method.
\end{abstract}
\end{adjustwidth}

\section{Introduction}\label{sec:intro}

We study the distributed minimax optimization problem
\begin{align}\label{prob:main}
\min_{x \in \mathcal{X}} \max_{y \in \mathcal{Y}} f(x, y) := \frac{1}{n} \sum_{i=1}^n f_i(x, y),    
\end{align}
where $f_i$ is the differentiable local function associated with the $i$-th node, and~$\fX\subseteq\BR^{d_x}$ and $\fY\subseteq\BR^{d_y}$ are the constraint sets.
We are interested in the centralized setting, where there are one server node and $n-1$ client nodes that collaboratively solve the minimax problem. 
Without loss of generality, we assume the function $f_1$ is located on the server node and the functions~$f_2,\dots,f_{n}$ are located on the client nodes.
This formulation is a cornerstone in the study of game theory, aiming to achieve the Nash equilibrium \cite{6578120,Razaviyayn_2020}. It covers a lot of applications such as signal processing \cite{4434853}, optimal control \cite{8472154}, adversarial learning \cite{Razaviyayn_2020}, robust regression \cite{marteau2019globally,hendrikx2020statistically} and portfolio management \cite{XIDONAS2017299}. 

We focus on the first-order optimization methods for solving convex-concave minimax problem.
The classical full-batch approaches including extra-gradient (EG) method \cite{korpelevich1976extragradient}, forward-backward-forward (FBF) \cite{tseng2000modified}, optimistic gradient descent ascent (OGDA)~\cite{popov1980modification}, dual extrapolation \cite{nesterov2007dual} and so forth \cite{Nemirovski2004ProxMethodWR,malitsky2015projected,malitsky2020forward} achieve the optimal first-order oracle complexity under the assumption of Lipschitz continuous gradient~\cite{ouyang2021lower,zhang2022lower,ibrahim2020linear}.
For the objective with finite-sum structure, the stochastic variance reduced methods~\cite{alacaoglu2022stochastic,chavdarova2019reducing,luo2019stochastic,han2024lower,xie2020lower} can reduce the cost of per iteration by using the inexact gradient and lead to the better overall computational cost than full-batch methods.
It is natural to design the parallel iteration schemes by directly using above ideas to reduce the computational time in distributed setting.

The communication complexity is a primary bottleneck in distributed optimization.
The local functions in machine learning applications typically exhibit homogeneity~\cite{arjevani2015communication,kairouz2021advances,hendrikx2020statistically}, which is helpful to improve the communication efficiency.
One common measure used to describe relationships among local functions is the second-order similarity, e.g., the Hessian of each local function differs by a finite quantity from the Hessian of global objective.
Based on such characterization, several communication efficient distributed optimization methods have been established
\cite{shamir2014communication,pmlr-v37-zhangb15,sun2022distributed,tian2022acceleration,karimireddy2020scaffold,khaled2022faster,lin2023stochastic,kovalev2022optimalgradient,beznosikov2021distributed,beznosikov2022compression,beznosikov2024similarity}.
The highlight of these methods is their communication complexity bounds mainly depend on the degree of second-order similarity, which is potentially much tighter than the results depend on the smoothness parameter~\cite{beznosikov2022distributed,beznosikov2022decentralized,hou2021efficient,deng2021local,kovalev2022optimalgradient,kovalev2022optimal,stich2018local,li2019convergence,khaled2019first,gorbunov2021local,khaled2020tighter,mitra2021linear,mishchenko2022proxskip,malinovsky2022variance}.

Recently, \citet{khaled2022faster,lin2023stochastic} showed the iteration with partial participation can further reduce the communication complexity, improving the dependence on the number of nodes. 
They proposed stochastic variance reduced proximal point methods for convex optimization, which allow only one of clients to participate into the communication in the most of rounds.
Additionally, \citet{beznosikov2024similarity} combined partial participation with forward-backward-forward based method, reducing volume of communication complexity for minimax optimization.
However, these methods \cite{khaled2022faster,lin2023stochastic,beznosikov2024similarity} increase the communication rounds, 
which results more expensive time cost in communication than the full participation strategies \cite{kovalev2022optimalgradient,beznosikov2021distributed}.
In other words, the partial participation methods \cite{khaled2022faster,lin2023stochastic,beznosikov2024similarity} only reduce the overall volume of information exchanged among the nodes, while the advantage of parallel communication enjoyed in full participation methods is damaged.

In this paper, we propose a novel distributed minimax optimization method, called stochastic variance-reduced optimistic
gradient sliding (SVOGS), which uses the mini-batch client sampling to balance communication rounds, communication complexity, and computational complexity.
We prove SVOGS simultaneously achieves the (near) optimal communication complexity, communication rounds, and local gradient calls for convex-concave minimax problem under the assumption of second-order similarity.
We also conduct numerical experiments to show the superiority of  SVOGS.

\section{Preliminaries}\label{sec:preliminaries}

We focus on the distributed optimization in client-sever framework for solving minimax problem (\ref{prob:main}). 
The notation $f_i$ presents the local function on the $i$-th node. We assume the function $f_1$ is located on the server 
and the other individuals are located on clients.
We stack variables $x\in\BR^{d_x}$ and $y\in\BR^{d_y}$ as the vector $z = [x;y]\in\BR^d$, where $d=d_x+d_y$.
We let $\fZ := \fX \times \fY\subseteq\BR^d$ and define the projection operator 
$\fP_{\mathcal{Z}}(v):=\argmin_{z\in\fZ}\norm{z-v}$
for given $v\in\BR^d$.
We also denote the vector functions~$F_i:\BR^d \rightarrow \BR^d$ and $F:\BR^d \rightarrow \BR^d$ as
\begin{align*}
    F_i(z) := \begin{bmatrix} \nabla_x f_i(x, y) \\ -\nabla_y f_i(x, y) \end{bmatrix}\qquad  \text{and}\qquad F(z) := \frac{1}{n} \sum_{i=1}^{n} F_i(z). 
\end{align*} 
We consider the following common assumptions for our minimax problem.
\begin{asm}\label{asm:set}
We suppose the constraint set $\fZ\subseteq\BR^d$ is a non-empty, closed, and convex.
\end{asm}
\begin{asm}\label{asm:bounded}
We suppose the constraint set $\fZ\subseteq\BR^d$ is bounded by  diameter $D>0$, i.e., we have~$\norm{z_1-z_2}\leq D$ for all $z_1,z_2\in\fZ$.  
\end{asm}
\begin{asm}\label{asm:smooth}
We suppose each local function $f_i:\BR^{d_x}\times\BR^{d_y}\to\BR$ is smooth, i.e., there exists~$L>0$ such that $\norm{F_i(z_1)-F_i(z_2)} \leq L\norm{z_1-z_2}$ for all $i\in [n]$ and~$z_1, z_2\in \BR^d$ .
\end{asm}
\begin{asm}\label{asm:cc}
We suppose each differentiable local function $f_i:\BR^{d_x}\times\BR^{d_y}\to\BR$ is convex-concave, 
i.e., we have
$f_i(x,y) \geq f_i(x',y) + \inner{\nabla_x f_i(x',y)}{x-x'}$ and
$f_i(y) \leq f_i(y') + \inner{\nabla_y f_i(y')}{y-y'}$
for all $i\in[n]$, $x,x'\in\BR^{d_x}$ and $y,y'\in\BR^{d_y}$.
\end{asm}
\begin{asm}\label{asm:scsc}
We suppose the global objective $f:\BR^{d_x}\times\BR^{d_y}\to\BR$ is strongly-convex-strongly-concave, 
i.e., there exists $\mu>0$ such that the function $f(x,y)-\frac{\mu}{2}\norm{x}^2+\frac{\mu}{2}\norm{y}^2$ is convex-concave.
\end{asm}
Besides, we introduce the assumption of second-order similarity  to measure the homogeneity in local functions \cite{beznosikov2021distributed,beznosikov2024similarity,khaled2022faster,karimireddy2020scaffold}.
% \begin{asm}[Strongly-SS]\label{asm:strongss}
% The functions $f_1,\dots,f_n:\BR^{d_x}\times\BR^{d_y}\to\BR$ satisfy the strongly second-order similar, 
% i.e., there exists $\delta>0$ such that
% \begin{align} \begin{split}
%     \norm{\nabla^2 f_i(x,y) - \nabla^2 f_j(x,y)} \leq \delta
% \end{split} \end{align}
% for all $i\in[n]$, $j\in[n]$, $x\in\BR^{d_x}$ and $y\in\BR^{d_y}$.
% \end{asm}
\begin{asm}\label{asm:ss}
The local functions $f_1,\dots,f_n:\BR^{d_x}\times\BR^{d_y}\to\BR$ are twice differentiable and hold the $\delta$-second-order similarity, 
i.e., there exists $\delta>0$ such that
\begin{align*}
    \norm{\nabla^2 f_i(x,y) - \nabla^2 f(x,y)} \leq \delta
\end{align*}
for all $i\in[n]$, $x\in\BR^{d_x}$ and $y\in\BR^{d_y}$.
\end{asm}

We measure the sub-optimality of the approximate solution $z=(x,y)\in\fZ$ by duality gap, that is
\begin{align*} \begin{split}
     \operatorname{Gap}(x,y):=\max_{y'\in\fY}f(x,y')-\min_{x'\in\fX}f(x',y).
\end{split} 
\end{align*}
We also consider the criterion of the gradient mapping for given $z=(x,y)\in\fZ$~\cite{nesterov2013introductory,yang2020catalyst,luo2021near}, that is,
\begin{align*}
    \ffF_\tau(z)=\frac{z-\fP_\fZ (z-\tau F(z))}{\tau},
\end{align*}
where $\tau>0$. 
The gradient mapping $\ffF_\tau(z)$ is a natural extension of gradient operator $F(z)$.
Noticing that we have $\ffF_\tau(z)=F(z)$ if the problem is unconstrained (i.e., $\fZ=\BR^d$), and the condition~$\ffF_\tau(z)=0$ is equivalent to the point $z$ is a solution of the problem.
Compared with the duality gap, the norm of gradient mapping is a more popular measure in empirical studied since it is easy to  achieve in practice.

For the specific strongly-convex-strongly-concave case, we can also measure the sub-optimality by the square of Euclidean distance to the unique solution $z^*=(x,y)\in\fZ$, that is
\begin{align*}
    \norm{z-z^*}^2 = \norm{x-x^*}^2 + \norm{y-y^*}^2.
\end{align*}

%We also present the duality gap by $\operatorname{Gap}(z)$ for simplification.

Moreover, we use notations $\fO(\cdot)$, $\Theta(\cdot)$ and $\Omega(\cdot)$ to hide
constants which do not depend on parameters of the problem, and notations $\tilde\fO(\cdot)$, $\tilde\Theta(\cdot)$, and $\tilde\Omega(\cdot)$ to additionally hide the logarithmic factors of $n$, $L$, $\mu$ and~$\delta$.

\section{Related Work}

For convex-concave minimax optimization,
the full batch first-order methods \cite{korpelevich1976extragradient,popov1980modification,nesterov2007dual,Nemirovski2004ProxMethodWR,tseng2000modified,malitsky2015projected,malitsky2020forward}
 can achieve $\varepsilon$-duality gap within at most~$\fO(LD^2/\varepsilon)$ iterations.
Applying these idea to distributed setting naturally leads to the communication rounds of $\fO(LD^2/\varepsilon)$ and each round requires all of the~$n$ nodes to compute and communicate their local gradient.  

In a seminar work, \citet{beznosikov2021distributed} proposed Star Min-Max Data Similarity (SMMDS) algorithm, 
which additionally consider the second-orders similarity (Assumption \ref{asm:ss}) by involving gradient sliding technique \cite{rogozin2021decentralized,lan2016gradient}. 
The SMMDS requires communication rounds of $\fO(\delta D^2/\varepsilon)$, which benefits from the homogeneity in local functions. 
Each round of this method needs to communicate/compute the local gradient of all $n$ nodes, and perform the local updates on the server within $\tilde\fO(L/\delta\log(1/\varepsilon))$ local iterations, which results the overall communication complexity of $\fO(n\delta D^2/\varepsilon)$ and local gradient complexity of $\tilde\fO((n\delta+L)D^2/\varepsilon\log(1/\varepsilon))$.
Later, \citet{kovalev2022optimalgradient} introduced extra-gradient sliding (EGS), which further improves the local gradient complexity to~$\fO((n\delta+L)D^2/\varepsilon)$.
It is worth pointing out that the communication rounds of~$\fO(\delta D^2/\varepsilon)$ achieved by SMMDS and EGS matches the lower complexity bound under the second-order similarity assumption~\cite{beznosikov2021distributed}.
However, these methods enforce all nodes to participate into  communication in every round, which does not sufficiently take the advantage of finite-sum structure in the objective.

Recently, \citet{beznosikov2024similarity} proposed Three Pillars Algorithm with Partial Participation (TPAPP), which uses the variance-reduced forward-backward-forward method \cite{alacaoglu2022stochastic,chavdarova2019reducing} to encourage only one of clients participate into the communication in most of the rounds.
%By appropriate setting of local iteration numbers, 
The TPAPP can achieve point $z\in\BR^d$ such that $\BE[\|F(z)\|^2]\leq{\varepsilon}$ for unconstrained case within the communication rounds of $\fO(n\delta^2D^2/\varepsilon)$, communication complexity of $\fO(n\delta^2D^2/\varepsilon)$, and local gradient complexity of~$\fO(n^2\delta^4L^2D^6\varepsilon^{-3})$.\footnote{Although the complexity of TPAPP for achieving $\BE[\|F(z)\|]^2\leq{\varepsilon}$ is established for the unconstrained case, its analysis additionally assume that the sequences generated by the algorithm are bounded by $D>0$ \cite[Theorem 5.12]{beznosikov2024similarity}.}
The theoretical analysis of TPAPP for the constrained problem requires the objective being strongly-convex-strongly-concave.

In addition, we can also reduce the communication complexity by using the permutation compressors \cite{szlendak2021permutation} for high-dimensional problem~\cite{beznosikov2022compression,beznosikov2024similarity}, which achieves the similar complexity to existing partial participation methods~\cite{beznosikov2024similarity}.

We present the complexity of existing methods
and compare them with our results in both general convex-concave case and strongly-convex-strongly-concave case in Table \ref{table:cc}-\ref{table:ccg}.

\begin{table}[t]
\centering
\caption{The complexity of achieving $\BE[\operatorname{Gap}(x,y)]\leq\varepsilon$ in convex-concave case.} \label{table:cc}\vskip0.2cm
\small
\begin{tabular}{cccc}
\toprule
Methods & Communication Rounds & \!Communication Complexity\! & Local Gradient Complexity  \\
\midrule
  EG \cite{korpelevich1976extragradient} & $\fO\big(\frac{L D^2 }{\varepsilon}\big)$ & $\fO\big(\frac{nL D^2 }{\varepsilon}\big)$ & $\fO\big(\frac{nL D^2 }{\varepsilon}\big)$ \\\addlinespace
  SMMDS \cite{beznosikov2021distributed} & $\fO\big(\frac{\delta D^2 }{\varepsilon}\big)$ & $\fO\big(\frac{n\delta D^2 }{\varepsilon}\big)$ & $\tilde{\fO}\big(\frac{(n\delta+L) D^2 }{\varepsilon}\log\frac{1}{\varepsilon}\big)$  \\\addlinespace
 EGS \cite{kovalev2022optimalgradient}  & $\fO\big(\frac{\delta D^2 }{\varepsilon}\big)$ & $\fO\big(\frac{n\delta D^2 }{\varepsilon}\big)$ & $\fO\big(\frac{(n\delta+L) D^2 }{\varepsilon}\big)$  \\\addlinespace
%  SVEGS, Alg. \ref{alg:SVEGS}   & $\fO\big(\frac{\sqrt{n}\delta D^2 }{\varepsilon}\big)$ & $\fO\big(n+\frac{\sqrt{n}\delta D^2 }{\varepsilon}\big)$ & $\tilde{\fO}\big(n+\frac{\sqrt{n}L D^2 }{\varepsilon}\big)$  & Strongly-SS \\
  \begin{tabular}{c}
  SVOGS \\ (Algorithm \ref{alg:SVOGS}) 
  \end{tabular}
  & $\fO\big(\frac{\delta D^2 }{\varepsilon}\big)$ & $\fO\big(n+\frac{\sqrt{n}\delta D^2 }{\varepsilon}\big)$ & $\tilde{\fO}\big(n+\frac{(\sqrt{n}\delta+L) D^2 }{\varepsilon}\log\frac{1}{\varepsilon}\big)$  \\
\midrule
\begin{tabular}{c}
Lower Bounds \\ (Theorem \ref{thm:lowerboundKcc},\ref{thm:lowerboundVG1cc},\ref{thm:lowerboundGcc}) 
\end{tabular}
& $\Omega\big(\frac{\delta D^2 }{\varepsilon}\big)$ 
& $\Omega\big(n+\frac{\sqrt{n}\delta D^2 }{\varepsilon}\big)$ 
& $\Omega\big(n+\frac{(\sqrt{n}\delta+L) D^2 }{\varepsilon}\big)$ \\
\bottomrule
\end{tabular}\vskip0.05cm
\end{table}

\begin{table}[t]
\centering
\caption{The complexity of achieving $\BE[\norm{x-x^*}^2+\norm{y-y^*}^2]\leq\varepsilon$ in the strongly-convex-strongly-concave case.  
%TPA and TPAPP requires the average second-order similarity (ASS, Assumption \ref{asm:meanss}), and the other methods requires the component second-order similarity (ASS, Assumption \ref{asm:ss}).
$^\dagger$These methods use permutation compressors \cite{szlendak2021permutation}, which require the assumption of $d>n$. $^\sharp$The complexity of TPAPP depends on local iterations number $H$, where ``TPAPP (a)'' and ``TPAPP (b)'' correspond to $H\!\!=\!\!\lceil L/(\sqrt{n}\delta) \rceil$ and $H\!\!=\!\!\lceil 8\log(40nL/\mu)\rceil$ respectively.}\label{table:scsc} \vskip0.2cm
\small
\begin{tabular}{cccc}
\toprule
Methods & Communication Rounds & Communication Complexity & Local Gradient Complexity \\%& \!\!Similarity\!\!  \\ 
\midrule  
  EG \cite{korpelevich1976extragradient} & $\fO\big(\frac{L}{\mu}\log\frac{1}{\varepsilon}\big)$ & $\fO\big(\frac{nL}{\mu}\log\frac{1}{\varepsilon}\big)$ & $\fO\big({\frac{nL}{\mu}}\log\frac{1}{\varepsilon}\big)$   \\\addlinespace
  SMMDS \cite{beznosikov2021distributed} & $\fO\big(\frac{\delta}{\mu}\log\frac{1}{\varepsilon}\big)$ & $\fO\big(\frac{n\delta}{\mu}\log\frac{1}{\varepsilon}\big)$ & $\tilde{\fO}\big({\frac{n\delta+L}{\mu}}\log\frac{1}{\varepsilon}\big)$  \\\addlinespace
 EGS \cite{kovalev2022optimalgradient} & $\fO\big(\frac{\delta}{\mu}\log\frac{1}{\varepsilon}\big)$ & $\fO\big({\frac{n\delta}{\mu}}\log\frac{1}{\varepsilon}\big)$ & $\fO\big({\frac{n\delta+L}{\mu}}\log\frac{1}{\varepsilon}\big)$   \\\addlinespace
 OMASHA \cite{beznosikov2022compression}$^\dagger$
 & $\fO\big(\frac{L}{\mu}\log\frac{1}{\varepsilon}\big)$&$\fO\big(\big(n+\frac{\sqrt{n}\delta+L}{\mu}\big)\log\frac{1}{\varepsilon}\big)$&$\fO\big(\frac{nL}{\mu}\log\frac{1}{\varepsilon}\big)$ \\\addlinespace
 TPA \cite{beznosikov2024similarity}$^\dagger$ & $\fO\big(\big({n}+\frac{\sqrt{n}\delta}{\mu}\big)\log\frac{1}{\varepsilon}\big)$ & $\fO\big(\big(n+\frac{\sqrt{n}\delta}{\mu}\big)\log\frac{1}{\varepsilon}\big)$ & $\fO\big(\big(n+\frac{\sqrt{n}L}{\delta}+\frac{L}{\mu}\big)\log\frac{1}{\varepsilon}\big)$   \\\addlinespace
 TPAPP (a) \cite{beznosikov2024similarity}$^\sharp$ & $\fO\big(\big({n}+\frac{\sqrt{n}\delta}{\mu}\big)\log\frac{1}{\varepsilon}\big)$ & $\fO\big(\big(n+\frac{\sqrt{n}\delta}{\mu}\big)\log\frac{1}{\varepsilon}\big)$ & $\fO\big(\big(n+\frac{\sqrt{n}L}{\delta}+\frac{L}{\mu}\big)\log\frac{1}{\varepsilon}\big)$   \\\addlinespace
 TPAPP (b) \cite{beznosikov2024similarity}$^\sharp$ & \!\!$\fO\big(\big(n+\frac{\sqrt{n}\delta+L}{\mu}\big)\log\frac{1}{\varepsilon}\big)$\!\! & $\fO\big(\big(n+\frac{\sqrt{n}\delta+L}{\mu}\big)\log\frac{1}{\varepsilon}\big)$ & $\tilde{\fO}\big(\big(n+\frac{\sqrt{n}\delta+L}{\mu}\big)\log\frac{1}{\varepsilon}\big)$   \\\addlinespace
%  SVEGS, Alg. \ref{alg:SVEGS}  & $\fO\big(\big({n}+\frac{\sqrt{n}\hat{\delta}}{\mu}\big)\log\frac{1}{\varepsilon}\big)$  & $\fO\big(\big(n+\frac{\sqrt{n}\hat{\delta}}{\mu}\big)\log\frac{1}{\varepsilon}\big)$ & $\tilde{\fO}\big(\big(n+\frac{\sqrt{n}\hat{\delta}+L}{\mu}\big)\log\frac{1}{\varepsilon}\big)$  & Strongly-SS  \\
  \begin{tabular}{c}
    SVOGS \\ (Algorithm \ref{alg:SVOGS})
  \end{tabular}
   & $\fO\big(\frac{\delta}{\mu}\log\frac{1}{\varepsilon}\big)$ & $\fO\big(\big(n+\frac{\sqrt{n}\delta}{\mu}\big)\log\frac{1}{\varepsilon}\big)$ & $\tilde{\fO}\big(\big(n+\frac{\sqrt{n}\delta+L}{\mu}\big)\log\frac{1}{\varepsilon}\big)$    \\
\midrule
\!\!\!\!\!
\begin{tabular}{c}
Lower Bounds \\ (\cite{beznosikov2021distributed,beznosikov2024similarity}, Theorem \ref{thm:lowerboundGG2scsc}) 
\end{tabular} \!\!\!\!\!
 & $\Omega\big(\frac{\delta}{\mu}\log\frac{1}{\varepsilon}\big)$ & $\Omega\big(\big(n+\frac{\sqrt{n}\delta}{\mu}\big)\log\frac{1}{\varepsilon}\big)$ & $\Omega\big(\big(n+\frac{\sqrt{n}\delta+L}{\mu}\big)\log\frac{1}{\varepsilon}\big)$\\ % & \!\!ASS/CSS \!\! \\
% Lower Bounds \cite{beznosikov2021distributed} & $\Omega\big(\frac{\delta}{\mu}\log\frac{1}{\varepsilon}\big)$ & -& - & CSS\\
% Lower \cite{beznosikov2024similarity} &-& $\Omega\big(\big(n+\frac{\sqrt{n}\delta}{\mu}\big)\log\frac{1}{\varepsilon}\big)$&-& ASS/CSS\\
% Theorem \ref{thm:lowerboundGscsc} &-& -&$\Omega\big(\big(n+\frac{\sqrt{n}\delta+L}{\mu}\big)\log\frac{1}{\varepsilon}\big)$ & ASS/CSS\\
\bottomrule
\end{tabular} %\vskip-0.2cm
%\footnotesize{$^{2}$Complexity form the server to clients is treated as free, which is not so reasonable.}
% \footnotesize{$^{3}$Choose $H=\lceil L/(\sqrt{n}\delta) \rceil$.}
% \footnotesize{$^{4}$Choose $H=\fO(1)\geq 8\log(40nL/\mu)$.}
\end{table}

\begin{table}[t]
\centering
\caption{The complexity of achieving $\BE[\norm{\ffF_\tau(x,y)}^2]\leq\varepsilon$ in convex-concave case.
$^\S$The  TPAPP additionally assumes $\fZ=\BR^d$ and the sequences 
 generated by the algorithm is bounded by $D>0$.
} \label{table:ccg}\vskip0.2cm
\small
\begin{tabular}{cccc}
\toprule
Methods & Communication Rounds & Communication Complexity & Local Gradient Complexity  \\
\midrule
 TPAPP \cite{beznosikov2024similarity}$^\S$  & $\fO\big(\frac{n\delta^2 D^2 }{\varepsilon}\big)$ & $\fO\big(\frac{n\delta^2 D^2 }{\varepsilon}\big)$ & $\fO\big(\frac{n^2\delta^4L^2D^6}{\varepsilon^3}\big)$  \\\addlinespace
  \!\begin{tabular}{c}
  SVOGS \\ (Algorithm \ref{alg:SVOGS}) 
  \end{tabular}\!
  & $\tilde\fO\big(\frac{\delta D }{\sqrt{\varepsilon}}\log\big(\frac{1}{\varepsilon}\big)\big)$ & $\tilde\fO\big(\big(n+\frac{\sqrt{n}\delta D }{\sqrt{\varepsilon}}\big)\log(\frac{1}{\varepsilon}\big))$ & $\tilde\fO\big(\big(n+\frac{\sqrt{n}L D }{\sqrt{\varepsilon}}\big)\log(\frac{1}{\varepsilon}\big)\big)$  \\
\bottomrule
\end{tabular}\vskip0.2cm
\end{table}

\begin{algorithm}[t]
\caption{Stochastic Variance-Reduced Optimistic Gradient Sliding (SVOGS)}
\label{alg:SVOGS}
\begin{algorithmic}[1]
\STATE \textbf{Input:} initial point $z^0 = (x^0, y^0)\in\fZ$, step size $\eta$, accuracy $\{\varepsilon_k\}_{k=1}^K$, communication rounds~$K$, mini-batch size $b$, probability $p \in (0, 1]$, weights $\alpha,\gamma \in (0, 1)$;\\[0.15cm]
\STATE \textbf{Initialization:} $w^{-1}=z^{-1}=w^0=z^0 = (x^0, y^0) \in \mathcal{Z}$, $z_{i}^0 = z^0$ for all $i\in [n]$;\\[0.15cm]
\FOR{$k = 0, 1, 2, \ldots$, $K-1$} \vskip0.15cm
    \STATE  $\bar{z}^k = (1-\gamma) z^k +\gamma w^k$;\\[0.15cm]
    \STATE Sample $\fS^k = \{j_{1}^k, \ldots, j_{b}^k \}$ uniformly and independently from $[n]$;\\[0.15cm]
    \STATE $\displaystyle{\delta^{k} =F(w^{k-1}) - F_1(w^{k-1}) + \frac{1 }{b}\sum_{j\in \fS^k}\big(F_{j}(z^k) - F_1(z^k)- F_{j}(w^{k-1}) + F_1(w^{k-1})\big)}$ \\
     $\displaystyle{\quad\quad~ + \frac{\alpha }{b}\sum_{j\in \fS^k}\big(F_{j}(z^k) - F_1(z^k)- F_{j}(z^{k-1}) + F_1(z^{k-1})\big)}$; \\[0.15cm]    
%    \STATE $\delta^{k} =F(w^{k-1}) - F_1(w^{k-1}) + \frac{1 }{b}\sum_{j\in \fS^k}\big(F_{j}(z^k) - F_1(z^k)- F_{j}(w^{k-1}) + F_1(w^{k-1})\big)+ \frac{\alpha }{b}\sum_{j\in \fS^k}\big(F_{j}(z^k) - F_1(z^k)- F_{j}(z^{k-1}) + F_1(z^{k-1})\big);$ \\[0.1cm]        
    % \STATE $\delta^{k} =F(w^{k-1}) - F_1(w^{k-1}) + \frac{1 }{b}\sum_{j\in \fS^k}[(F_{j}(z^k) - F_1(z^k)- F_{j}(w^{k-1}) + F_1(w^{k-1}))]+ \frac{\alpha }{b}\sum_{j\in \fS^k}[(F_{j}(z^k) - F_1(z^k)- F_{j}(z^{k-1}) + F_1(z^{k-1}))];$ \\[0.1cm]
    \STATE\label{line:updatevk} $ v^{k} =\bar{z}^k - \eta \delta^k;$\\[0.15cm]
    \STATE\label{line:sub-prob} Find $u^k\in\BR^d$ such that $\|u^k - \hat{u}^k\|^2 \leq \varepsilon_k$, where $\hat{u}^k$ is the solution of the problem
    \begin{align}\label{prob:sub-svogs}
        \min_{\hat x \in \mathcal{X}} \max_{\hat y \in \mathcal{Y}} \left\{ f_1(\hat x, \hat y) + \frac{1}{2\eta}\|\hat x - v^k_x\|^2 - \frac{1}{2\eta}\|\hat y - v^k_y\|^2\right\};
    \end{align}\\[0.15cm]
     \STATE $z^{k+1}=u^k;$\\[0.15cm]
     \STATE\label{line:update}$w^{k+1}= \begin{cases} 
    z^{k+1} & \text{with probability } p ,\\
    w^k & \text{with probability } 1-p .
    \end{cases}$ \\[0.15cm]
\ENDFOR
\end{algorithmic}
\end{algorithm}

\section{Stochastic Variance-Reduced Optimistic Gradient Sliding}\label{sec:algo}

We propose stochastic variance-reduced optimistic gradient sliding (SVOGS) method in Algorithm~\ref{alg:SVOGS}.

The design of our algorithm starts from reformulating problem (\ref{prob:main}) as follows
\begin{align}\label{prob:composite}
    \min_{x \in \mathcal{X}} \max_{y \in \mathcal{Y}}f(x,y):=\underbrace{\frac{1}{n}\sum_{i=1}^n (f_i(x,y)-f_1(x,y))}_{g(x,y):=f(x,y)-f_1(x,y)} + f_1(x,y).
\end{align}
The idea of gradient sliding \cite{lan2016gradient} on minimax optimization can be viewed as iteratively solving the surrogate of problem (\ref{prob:composite}) within the quadratic approximation of $g(x,y)$~\cite{beznosikov2021distributed,kovalev2022optimalgradient,beznosikov2024similarity}.
Recall that the optimistic gradient descent ascent (OGDA) method \cite{popov1980modification,malitsky2020forward} iterates with
\begin{align}\label{eq:FoRB}
    z^{k+1}=\fP_\fZ \big(z^k-\eta(\underbrace{F(z^k)+F(z^k)-F(z^{k-1})}_{\text{optimistic gradient}})\big),
\end{align}
where $\eta>0$ is the step size. 
It is well-known that OGDA achieves optimal convergence rate  under the first-order smoothness assumption~\cite{zhang2022lower,ouyang2021lower}, which motivated us construct the quadratic approximation of~$g(x,y)$ by using the optimistic gradient of $g$ at $(x^k,y^k)$ in the linear terms, that is
\begin{align}\label{approx:ogs}
\begin{split}    
   & g(x,y) \approx \hat g(x,y)  \\
 =& g(x^k,y^k) + \inner{\underbrace{\nabla_x g(x^k,y^k)+\nabla_x g(x^k,y^k)-\nabla_x g(x^{k-1},y^{k-1})}_{\text{optimistic gradient with respect to $x$}}}{x-x^k} + \frac{1}{2\eta}\norm{x-x^k}^2  \\
& + \inner{\underbrace{\nabla_y g(x^k,y^k)+\nabla_y g(x^k,y^k)-\nabla_y g(x^{k-1},y^{k-1})}_{\text{optimistic gradient with respect to $y$}}}{y-y^k} -\frac{1}{2\eta}\norm{y-y^k}^2.
\end{split}
\end{align}

Applying approximation (\ref{approx:ogs}) to formulation (\ref{prob:composite}), we obtain the optimistic gradient sliding (OGS), which iteratively solve the sub-problem
\begin{align}\label{prob:ogs}
    (x^{k+1},y^{k+1}) \approx \arg\min_{\hat x\in\fX}\max_{\hat y\in\fY}~\hat g(\hat x,\hat y) + f_1(\hat x,\hat y).
\end{align}
We can verify function $g(x,y)$ is $\delta$-smooth under Assumption \ref{asm:ss}, which indicates taking $\eta=\Theta(1/\delta)$ and solving the sub-problem sufficiently accurate can find an $\varepsilon$-suboptimal solution within the iteration numbers of $\fO(\delta D^2/\varepsilon)$ and $\fO(\delta/\mu\log(1/\varepsilon))$ for the convex-concave case and the strongly-convex-strongly-concave case respectively (see Section \ref{sec:analysis}).
The dependence on $\delta$ implies OGS benefits from the second-order similarity in local functions, while each of its iteration requires the communication and the computation of the exact gradient of $f(x,y)$ within  the complexity of $\fO(n)$.

The key idea to improve the cost in each iteration is involving the mini-batch client sampling and variance reduction with momentum~\cite{kovalev2022optimal,allen2018katyusha}. 
Specifically, we estimate the optimistic gradient in formulation (\ref{approx:ogs}) as follows
{\begin{align}\label{eq:vr-og}
\begin{split}
 \!\!\!G(z^k)+G(z^k)-G(z^{k-1}) 
\approx \frac{1}{|\fS^k|}\!\sum_{j\in\fS^k}\!\big(G(w^k)+G_j(z^k)-G_j(w^{k-1})+\underbrace{\alpha(G_j(z^k)-G_j(z^{k-1}))}_{\text{momentum term}}\big),
\end{split}
\end{align}}
where $G(z):=F(z)-F_1(z)$, $\fS^k\subseteq[n]$ is the random index set, $w^k$ is the snapshot point which is updated infrequently in iterations, and $\alpha\in(0,1)$ is the momentum parameter.
Applying the optimistic gradient estimation (\ref{eq:vr-og}) to formulations (\ref{approx:ogs})-(\ref{prob:ogs}), 
we achieve our stochastic variance-reduced optimistic gradient sliding (SVOGS) method (Algorithm~\ref{alg:SVOGS}).

The proposed SVOGS enjoys the mini-batch partial participation in the steps communication and computation in most of rounds, which is the main difference between SVOGS and existing methods \cite{beznosikov2021distributed,kovalev2022optimal,beznosikov2024similarity}.
Concretely, taking the mini-batch size $|\fS^k|=\Theta(\sqrt{n}\,)$ for SVOGS can simultaneously balance communication rounds, communication complexity and local gradient complexity.
The SVOGS keeps both the benefit of parallel communication like full participation methods (i.e., SMMDS~\cite{beznosikov2021distributed} and EGS~\cite{kovalev2022optimalgradient}) and the low communication cost like existing participation methods (i.e., TPAPP~\cite{beznosikov2024similarity}).
Additionally, the communication advantage of SVOGS also makes the algorithm achieves better local gradient complexity than state-of-the-arts \cite{beznosikov2021distributed,kovalev2022optimal,beznosikov2024similarity}.

\section{The Complexity Analysis}\label{sec:analysis}

In this section, we provide the complexity analysis of proposed SVOGS (Algorithm \ref{alg:SVOGS}) to show its superiority. 
In particular, we let $\mu=0$ for the convex-concave case to the ease of presentation.

We analyze the convergence of SVOGS (Algorithm~\ref{alg:SVOGS}) by establishing the Lyapunov function 
\begin{align}\label{eq:Lyapunov} \begin{split}
    \Phi^k :=& \left(\frac{1}{\eta}+\mu\right)\| z ^k- z^* \|^2+2\langle F( z ^{k-1})-F_1( z ^{k-1})-F( z ^k)+F_1( z ^k), z ^{k}- z^* \rangle\\
    & +\frac{1}{64\eta}\| z ^k- z ^{k-1}\|^2+\frac{\gamma}{4\eta}\|w^{k-1}- z ^k\|^2+\frac{(2\gamma+\eta\mu)}{2p\eta}\|w^k- z^* \|^2.
\end{split} \end{align}
where we take the step size such that $\eta\leq 1/(32\delta)$ which always guarantees $\Phi^k\geq 0$ by using Young's inequality and the similarity assumption (see detailed proof in Appendix \ref{appx:nonneg-Lya}).

We show that the decrease of Lyapunov function in expectation as follows.

\begin{lem}\label{lem:convergenceLya}
    Suppose Assumptions \ref{asm:set}, \ref{asm:smooth}, \ref{asm:cc}, and \ref{asm:ss} hold with $0\leq \mu\leq \delta\leq L$, running SVOGS (Algorithm~\ref{alg:SVOGS}) with  $\gamma\leq 1/8$, 
    $\alpha=\max\left\{1-\eta\mu/6,1-p\eta\mu/(2\gamma+\eta\mu)\right\}$, 
    $\eta\leq\min\{1/\mu,1/(32\delta)\}$, ${256\eta^2\delta^2\alpha^2(b+1)}/{b}\leq \alpha$, ${4\eta\delta^2}/{b}\leq {\alpha\gamma}/{(4\eta)}$, and $\varepsilon_k\leq c^{-1}\min\left\{\|\hat{u}^k- z ^k\|,\|\hat{u}^k- z ^k\|^2\right\}$ for some $c={\rm poly}(\mu,\delta)$, then we have
    \begin{align} \label{eq:convergenceLya}\begin{split}
    \E [\Phi^{k+1}]\leq \max\left\{1-\frac{\eta\mu}{6},1-\frac{p\eta\mu}{2\gamma+\eta\mu}\right\}\E [\Phi^{k}]-\frac{1}{16\eta}\E \left[\| z ^k-\hat{u}^k\|^2\right]-\frac{\gamma}{2\eta}\E \left[\|w^k-\hat{u}^k\|^2\right].
\end{split} \end{align}
\end{lem}

\subsection{The Convex-Concave Case}
For the convex-concave case, we use Jensen's inequality and the convexity (concavity) to bound the duality gap  at $u_{{\rm avg}}^K=\frac{1}{K}\sum_{k=0}^{K-1}u^k$ as follows
\begin{align}\label{ieq:gap-inner}
    \operatorname{Gap}(u_{{\rm avg}}^K)\leq \max_{(x',y')\in\fZ}\frac{1}{K}\sum_{k=0}^{K-1}\left(f(u_x^k,y')-f(x',u_y^k)\right)\leq \max_{z\in\fZ}\frac{1}{K}\sum_{k=0}^{K-1}\langle F(u^k),u^k-z\rangle.
\end{align}

Applying Lemma \ref{lem:convergenceLya} by summing over inequality (\ref{eq:convergenceLya}), we can bound the right-hand side of (\ref{ieq:gap-inner}) via the terms of $\sum_{k=0}^{K-1}\E \left[\| z ^k-\hat{u}^k\|^2\right]$ and $\sum_{k=0}^{K-1}\E \left[\|w^k-\hat{u}^k\|^2\right]$, and achieve the following theorem.

\begin{thm}\label{thm:convergenceSVOGScc}
    Suppose Assumptions \ref{asm:set}, \ref{asm:bounded}, \ref{asm:smooth}, \ref{asm:cc} and \ref{asm:ss} hold with $0<\delta\leq L$ and $D>0$, we run \mbox{Algorithm~\ref{alg:SVOGS}} with 
    $b=\left\lceil\sqrt{n}\,\right\rceil$, 
    $\gamma=p=1/{(\sqrt{n}+8)}$,  $\eta=\min\left\{{\sqrt{\gamma b}}/(4\delta),1/(32\delta)\right\}$, $\alpha=1$, and \begin{align*}
        \varepsilon_k=\min\left\{\zeta,\hat{c}^{-1}\min\left\{\|\hat{u}^k- z ^k\|,\|\hat{u}^k- z ^k\|^2\right\}\right\}
    \end{align*} for some $\zeta={\rm poly}(L,\delta,n,D,\varepsilon)$ and \mbox{$\hat{c}={\rm poly}(\delta)$}. Then we have
    % \begin{align} \begin{split}
    % \E[\operatorname{Gap}(u_{{\rm avg}}^K)]\leq \frac{67 D^2 }{\eta K} +\frac{9\eta L D +3\eta\max_{i\in [n]}\|F_i(z^0)\|  + D}{\eta}\sqrt{\zeta}+8\eta\delta^2\zeta,
    % \end{split} \end{align}
    \begin{align*} \begin{split}
    \E\left[\max_{z\in\fZ}\frac{1}{K}\sum_{k=0}^{K-1}\langle F(u^k),u^k-z\rangle\right]\leq \frac{10 D^2 }{\eta K} + \frac{\varepsilon}{2}, \quad \text{where}~~u_{{\rm avg}}^K=\frac{1}{K}\sum_{k=0}^{K-1}u^k.
    \end{split} \end{align*}    
\end{thm}

Theorem \ref{thm:convergenceSVOGScc} shows we can run SVOGS with step size $\eta=\Theta(1/\delta)$ and communication rounds of~$K=\fO(\delta D/\varepsilon)$ to achieve the $\varepsilon$-sub-optimality in expectation. 
Additionally, each communication round contains the expected communication complexity of $b(1-p)+np=\fO(\sqrt{n}\,)$, leading to the overall communication complexity of $\fO(n+\sqrt{n}\delta D^2 /\varepsilon)$.

The sub-problem (\ref{prob:sub-svogs}) in SVOGS (line \ref{line:sub-prob} of Algorithm \ref{alg:SVOGS}) is a minimax problem with $(L+1/\eta)$-smooth and $(1/\eta)$-strongly-convex-$(1/\eta)$-strongly-concave objective. 
Therefore, the setting of $\varepsilon_k$ and $\eta$ in the theorem indicates the condition $\|u^k - \hat{u}^k\|^2 \leq \varepsilon_k$ can be achieved by the local iterations number of~$\fO((L+\delta)/\delta\log(\varepsilon_k))={\tilde \fO}(L/\delta\log(1/\varepsilon))$ on the server (e.g., use EG \cite{korpelevich1976extragradient}).
Additionally, each round of SVOGS contains the expected local gradient complexity of $b(1-p)+np=\fO(\sqrt{n}\,)$ to achieve the (mini-batch) optimistic gradient $\delta^k$.
Hence, the overall local gradient complexity of SVOGS is~$\tilde\fO(K(\sqrt{n}+L/\delta\log(1/\varepsilon)))=\tilde{\fO}(n+(\sqrt{n}\delta+L) D^2 /\varepsilon\log(1/\varepsilon))$.
We formally present the upper complexity bounds of SVOGS for the convex-concave case as follows.

\begin{cor}\label{cor:complexitySVOGScc}
Following the setting of Theorem \ref{thm:convergenceSVOGScc}, 
we can achieve $ \E[\operatorname{Gap}(u_{{\rm avg}}^K)]\leq \varepsilon$ within communication rounds of $\fO(\delta D^2/\varepsilon)$, communication complexity of $\fO(n+\sqrt{n}\delta D^2 /\varepsilon)$,  and local gradient complexity of $\tilde{\fO}(n+(\sqrt{n}\delta+L) D^2 /\varepsilon\log(1/\varepsilon))$, where $u_{{\rm avg}}^K=\frac{1}{K}\sum_{k=0}^{K-1}u^k$.
\end{cor}

\subsection{The Strongly-Convex-Strongly-Concave Case}\label{sec:scsc}

By appropriate settings of SVOGS, 
Lemma \ref{lem:convergenceLya} leads to the following linear convergence of our Lyapunov function in the strongly-convex-strongly-concave case.

\begin{thm}\label{thm:convergenceSVOGSscsc}
    Suppose Assumptions \ref{asm:set}, \ref{asm:smooth}, \ref{asm:cc}, \ref{asm:scsc} and \ref{asm:ss} hold with $0< \mu\leq \delta\leq L$, we run Algorithm~\ref{alg:SVOGS} with   \mbox{$b=\left\lceil\min\left\{\sqrt{n},\delta/\mu\right\}\right\rceil$}, 
    \mbox{$\gamma=p=1/(\min\left\{\sqrt{n},\delta/\mu\right\}+8)$}, 
    $\eta =\min\left\{\sqrt{\alpha\gamma b}/(4\delta),1/(32\delta)\right\}$, \begin{align*}
        \alpha=\max\left\{1-\frac{\eta\mu}{6},1-\frac{p\eta\mu}{2\gamma+\eta\mu}\right\},\quad \text{and}\quad \varepsilon_k=c^{-1}\min\left\{\|\hat{u}^k- z ^k\|,\|\hat{u}^k- z ^k\|^2\right\}
    \end{align*}
     for some $c={\rm poly}(\mu,\delta)$. Then we have
    \begin{align*} \begin{split}
       \E[ \Phi^K]\leq \max\left\{1-\frac{\eta\mu}{6},1-\frac{p\eta\mu}{2\gamma+\eta\mu}\right\}^K \Phi^{0}.
    \end{split} \end{align*}
\end{thm}

We then apply Theorem \ref{thm:convergenceSVOGSscsc} with $K=\fO(\delta /\mu\log(1/\varepsilon))$ and analyze the complexity like the discussion after Theorem \ref{thm:convergenceSVOGScc}, which results the upper complexity bounds as follows.

\begin{cor}\label{thm:complexitySVOGSscsc}
    Following the setting of Theorem \ref{thm:convergenceSVOGSscsc}, we can achieve $\E\left[\|z^{K}-z^*\|^2\right]\leq \varepsilon$ within communication rounds of $\fO(\delta/\mu\log(1/\varepsilon))$, communication complexity of $\fO((n+\sqrt{n}\delta/\mu)\log(1/\varepsilon))$, and local gradient complexity of $\tilde{\fO}((n+(\sqrt{n}\delta +L)/\mu)\log(1/\varepsilon))$.
\end{cor}

\subsection{Making the Gradient Mapping Small}

For the convex-concave case (under the assumptions of Theorem \ref{thm:convergenceSVOGScc}), we can achieve the points with small gradient mapping by solving the regularized problem
\begin{align}\label{prob:regularized}
    \min_{x\in\fX} \max_{y\in\fY} \hat f(x,y) := f(x,y) + \frac{\lambda}{2}\norm{x-x^0}^2 -  \frac{\lambda}{2}\norm{y-y^0}^2
\end{align}
for some $\lambda>0$.
Noticing that the function $\hat f(x,y)$ is $(L+\lambda)$-smooth, $\lambda$-strongly-convex-$\lambda$-strongly-concave and $\delta$-similarity.
Then Corollary \ref{thm:complexitySVOGSscsc} implies running SVOGS (Algorithm \ref{alg:SVOGS}) by iterations number $K=\fO(\delta D/\sqrt{\varepsilon}\log(L/\varepsilon))$ to solve problem (\ref{prob:regularized}) with $\lambda = \fO(\sqrt{\varepsilon}/D)$ can achieve~$\BE[\|\ffF_\tau(z^K)\|^2]\leq\varepsilon$,
which results the complexity shown in Table \ref{table:ccg}. 
For the strongly-convex-strongly-concave case, the complexity of achieving $\BE[\|\ffF_\tau(z^K)\|^2\big]\leq\varepsilon$ nearly matches the complexity of achieving $\BE\big[\|z^K-z^*\|^2]\leq\varepsilon$.
We defer the detailed derivation for these results of making the gradient mapping small to Appendix \ref{appx:gds}.

\section{The Optimality of SVOGS}\label{sec:lower}
In this section, we provide the lower complexity bounds for solving our minimax problems by using distributed first-order oracle (DFO) methods.
The class of algorithms considered in our analysis follows the definition of \citet{beznosikov2024similarity}, 
which is formally described in Appendix \ref{appx:algorithm-class}.
Compared with existing lower bound analysis for second-order similarity only focusing on communication \cite{beznosikov2024similarity,beznosikov2021distributed}, we additionally study the computation complexity by considering the  local gradient calls.
The results in this section imply the complexity of proposed SVOGS (nearly) matches the lower bounds on the communication rounds, the communication complexity and the local gradient calls simultaneously.

% of different complexity measures. Note that in our discussion the algorithms belong to a class defined as Definition \ref{dfn:oracles}.
\subsection{The Lower Bounds for Convex-Concave Case}

We first provide the following lower bounds for convex-concave case.

\begin{thm}\label{thm:lowerboundKcc}
    For any $0<\delta\leq L$, $n\geq 3$, $D>0$ and $\varepsilon\leq \delta D^2/(12\sqrt{2}\,)$, there exist $L$-smooth and convex-concave functions $f_1,\dots,f_n:\BR^{d_x}\times\BR^{d_y}$ with $\delta$-second-order similarity, and closed convex set $\fZ=\fX\times\fY$ with diameter $D$. 
    In order to find an approximate solution $z=(x,y)$ of problem (\ref{prob:main}) such that $\E[{\rm Gap}(z)]\leq\varepsilon$, any DFO algorithm needs at least $\Omega(\delta D^2/\varepsilon)$ communication rounds.
\end{thm}

\begin{thm}\label{thm:lowerboundVG1cc}
    For any $0<\delta\leq L$, $n\geq 2$, $D>0$ and $\varepsilon\leq \delta D^2/(16\sqrt{2n}\,)$, there exist $L$-smooth and convex-concave functions $f_1,\dots,f_n:\BR^{d_x}\times\BR^{d_y}$ with $\delta$-second-order similarity, and closed convex set $\fZ=\fX\times\fY$ with diameter $D$.
    In order to find an approximate solution $z=(x,y)$ of problem~(\ref{prob:main}) such that $\E[{\rm Gap}(z)]\leq\varepsilon$, any DFO algorithm needs at least $\Omega(n+\sqrt{n}\delta D^2/\varepsilon)$ communication complexity and $\Omega(n+\sqrt{n}\delta D^2/\varepsilon)$ local gradient calls.
\end{thm}

The lower bounds on communication round and communication complexity shown in Theorem~\ref{thm:lowerboundKcc}~and~\ref{thm:lowerboundVG1cc} match the corresponding upper bounds of SVOGS shown in Corollary \ref{cor:complexitySVOGScc}.
However, the lower bound on local gradient complexity shown in Theorem \ref{thm:lowerboundVG1cc} only nearly matches the result of Corollary \ref{cor:complexitySVOGScc} in the case of $\sqrt{n}\delta \geq \Omega(L)$. 
Therefore, we also provide the following lower bound on local gradient complexity to show the tightness of dependence on the smoothness parameter $L$.

% incommunication round and communication complexity of SVOGS match the lower bounds, while the lower bound of local gradient complexity in Theorem \ref{thm:lowerboundVG1cc} only 
% while the tightness  

% Moreover, we can construct another bad function to give another lower bound for gradient complexity with $L$.

\begin{lem}\label{lem:lowerboundG2cc}
    For any $L > 0$, $n\in\BN$, $D>0$ and $\varepsilon\leq \delta D^2/(4\sqrt{2})$, there exist $L$-smooth and convex-concave functions $f_1,\dots,f_n:\BR^{d_x}\times\BR^{d_y}$ with $\delta$-second-order similarity, and closed convex set~$\fZ=\fX\times\fY$ with diameter $D$. 
    In order to find an approximate solution $z=(x,y)$ of problem~  (\ref{prob:main}) such that $\E[{\rm Gap}(z)]\leq\varepsilon$, any DFO algorithm needs at least $\Omega(n+LD^2/\varepsilon)$ local gradient calls.
\end{lem}

Combining the results of Theorem \ref{thm:lowerboundVG1cc} and Lemma \ref{lem:lowerboundG2cc}, we achieve the following lower bound on local gradient complexity, which nearly matches the corresponding upper bound shown in Corollary \ref{cor:complexitySVOGScc}.
%and note that the algorithms class in Definition \ref{dfn:oracles} includes the class in Lemma \ref{lem:lowerboundG2cc}, we can directly give a better lower bound for our setting.

\begin{thm}\label{thm:lowerboundGcc}
    For any $0<\delta\leq L$, $n\geq 2$, $D>0$ and $\varepsilon\leq \delta D^2/(16\sqrt{2n})$, there exist $L$-smooth and convex-concave functions $f_1,\dots,f_n:\BR^{d_x}\times\BR^{d_y}$ with $\delta$-second-order similarity, and closed convex set $\fZ=\fX\times\fY$ with diameter $D$.
    In order to find an approximate solution $z=(x,y)$ of problem~(\ref{prob:main}) such that $\E[{\rm Gap}(z)]\leq\varepsilon$, any DFO algorithm needs at least $\Omega(n+(\sqrt{n}\delta+L)D^2/\varepsilon)$ local gradient calls.
\end{thm}

The constructions in our lower bound analysis is based on the modifications on the blinear functions provided by \citet{han2024lower}, which are originally used to analyze the minimax optimization in non-distributed setting.
We provide detailed proofs in Appendix \ref{appx:lowercc}.
In related work, \citet{beznosikov2021distributed} also provide  the lower bound of $\Omega(\delta D^2/\varepsilon)$ (matching the result of Theorem \ref{thm:lowerboundKcc}) for communication rounds by using the regularized function, which is different from our construction in the proof of Theorem \ref{thm:lowerboundKcc}. 
In addition, our lower bounds on the communication complexity and the local gradient complexity shown in Theorem \ref{thm:lowerboundVG1cc} and \ref{thm:lowerboundGcc} are new. 

% 结合两个下界拼的

\subsection{The Lower Bounds for
Strongly-Convex-Strongly-Concave Case}

The tight lower bound on communication rounds in strongly-convex-strongly-concave case has been provided by \citet[Theorem 1]{beznosikov2021distributed}. We present the result as follows.

\begin{thm}[{\cite{beznosikov2021distributed}}]\label{thm:lowerboundKscsc}
    For any $\mu, \delta, L > 0$ with $L\geq \max\{\mu,\delta\}$ and $n\geq 3$, there exist $L$-smooth and convex-concave functions $f_1,\dots,f_n:\BR^{d_x}\times\BR^{d_y}$ with $\delta$-second-order similarity such that the function $f(x,y)=\frac{1}{n}\sum_{i=1}^n f_i(x,y)$ is $\mu$-strongly-convex-$\mu$-strongly-concave. In order to find a solution of problem (\ref{prob:main}) such that $\E[\|z-z^*\|^2]\leq\varepsilon$, any DFO algorithm needs at least $\Omega(\delta/\mu\log(1/\varepsilon))$ communication rounds.
\end{thm}
% 直接引用2021beznosikov的结果，中心化特例 图直径取2
The tight lower bound on communication complexity has been provided by \citet{beznosikov2024similarity}. 
We follow their construction to establish the lower bound on local gradient complexity, nearly matching the corresponding upper bound of our SVOGS. We formally present these lower bounds as follows.
\begin{thm}\label{thm:lowerboundGG2scsc}
    For any $\mu, \delta, L > 0$ with $L\geq \max\{\mu,\delta\}$ and $n\geq 2$, there exist $L$-smooth and convex-concave functions $f_1,\dots,f_n:\BR^{d_x}\times\BR^{d_y}$ with $\delta$-second-order similarity such that the function $f(x,y)=\frac{1}{n}\sum_{i=1}^n f_i(x,y)$ is $\mu$-strongly-convex-$\mu$-strongly-concave. In order to find a solution of problem (\ref{prob:main}) such that $\E[\|z-z^*\|^2]\leq\varepsilon$, any DFO algorithm needs at least $\Omega((n+\sqrt{n}\delta/\mu)\log(1/\varepsilon))$ communication complexity and $\Omega((n+(\sqrt{n}\delta+L)/\mu)\log(1/\varepsilon))$ local gradient calls.
\end{thm}

\begin{figure}[t]
\begin{minipage}[b]{0.3\textwidth}
\centering
\includegraphics[width=\textwidth]{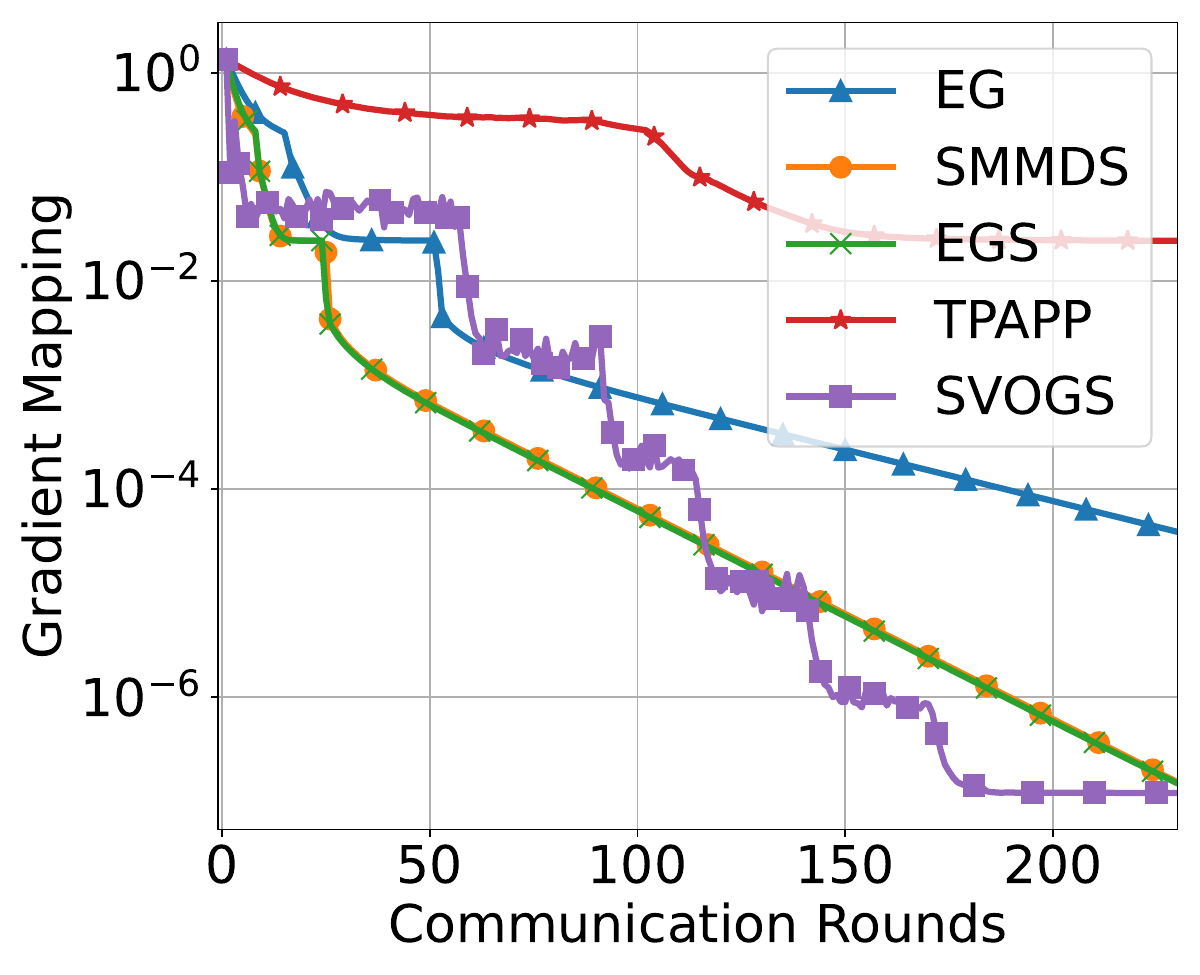}
\end{minipage}
\hfill
\begin{minipage}[b]{0.3\textwidth}
\centering
\includegraphics[width=\textwidth]{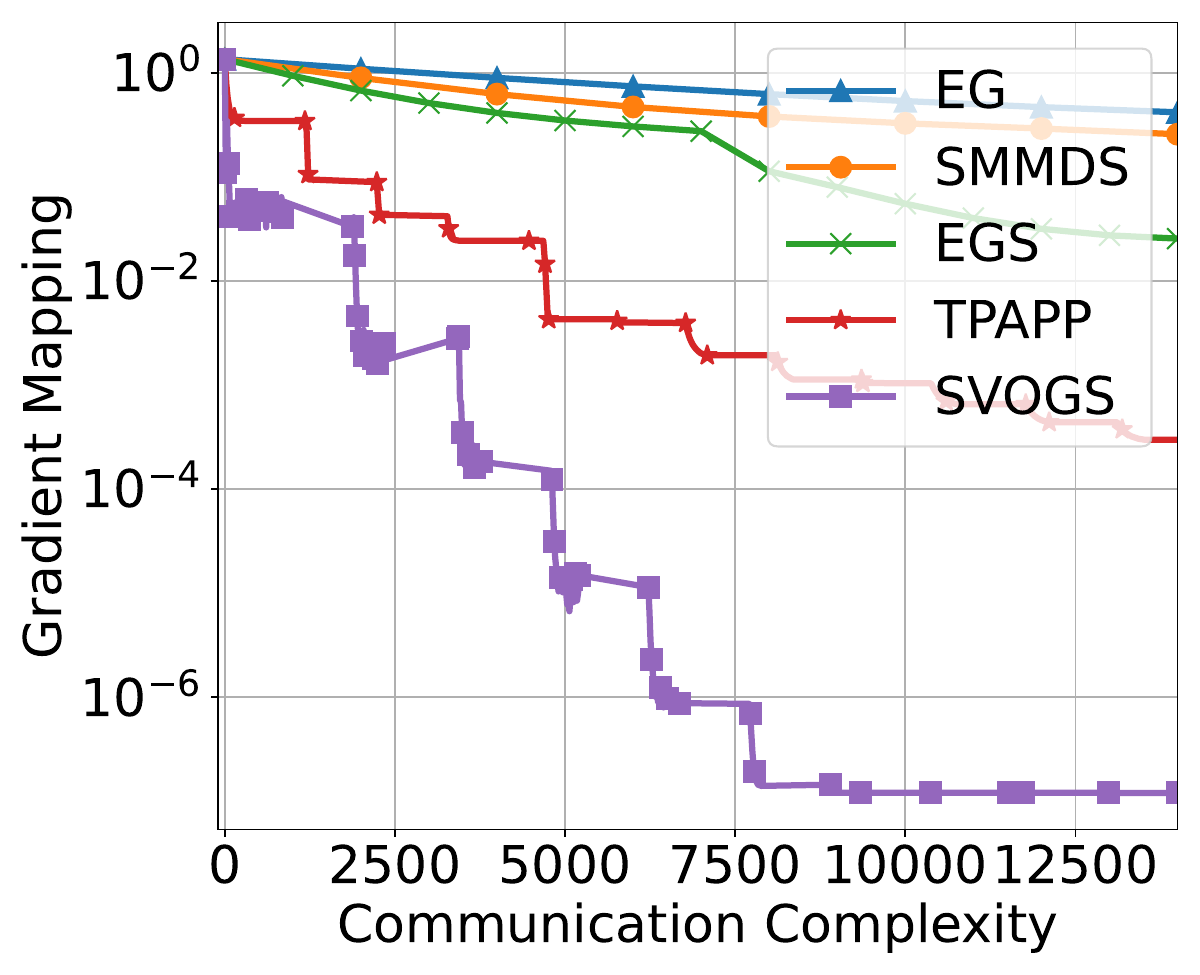}
\end{minipage}
\hfill
\begin{minipage}[b]{0.3\textwidth}
\centering
\includegraphics[width=\textwidth]{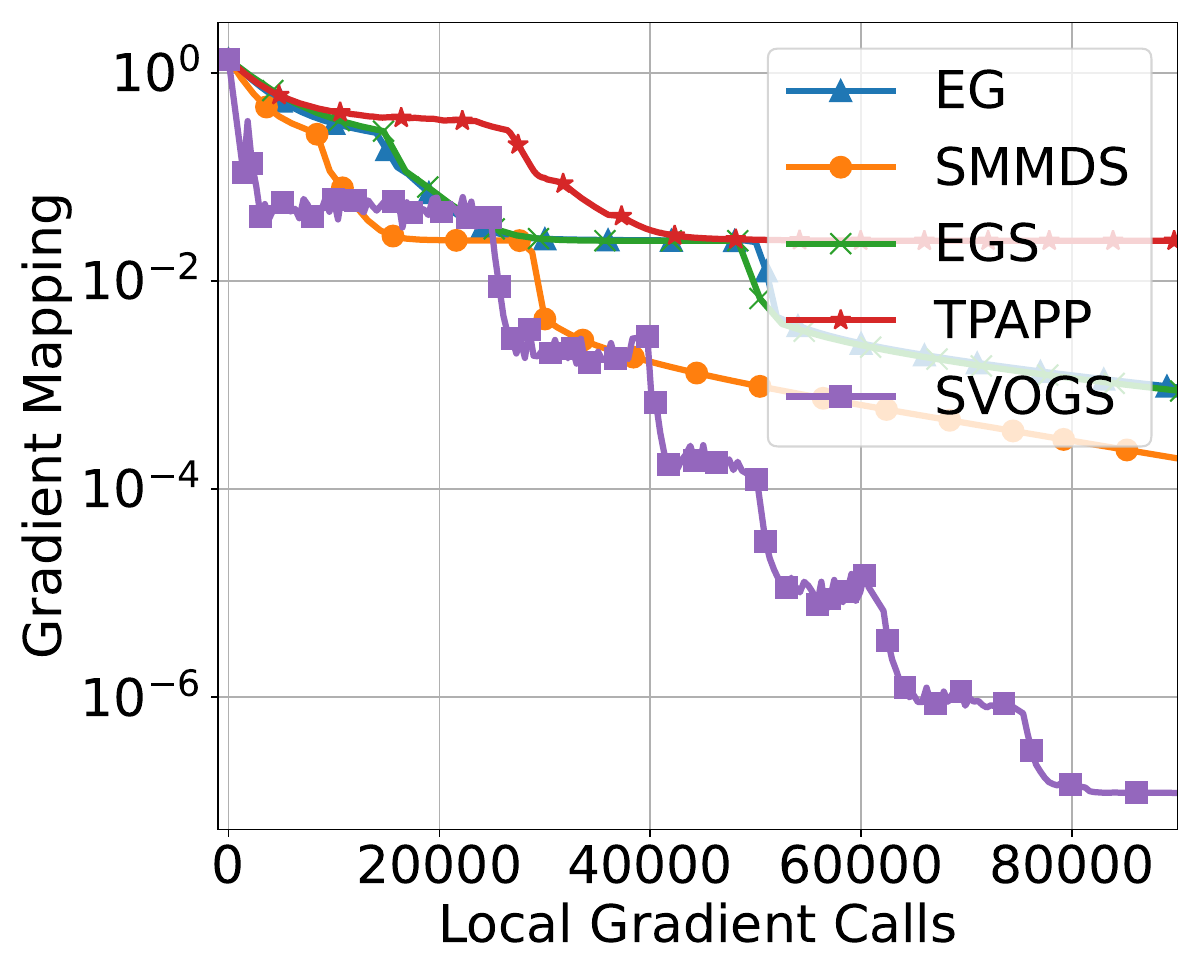}
\end{minipage}
\caption{The experimental results for convex-concave minimax problem (\ref{eq:problem1}).} \label{fig:a9acc}
\end{figure}

\begin{figure}[t]
\begin{minipage}[b]{0.3\textwidth}
\centering
\includegraphics[width=\textwidth]{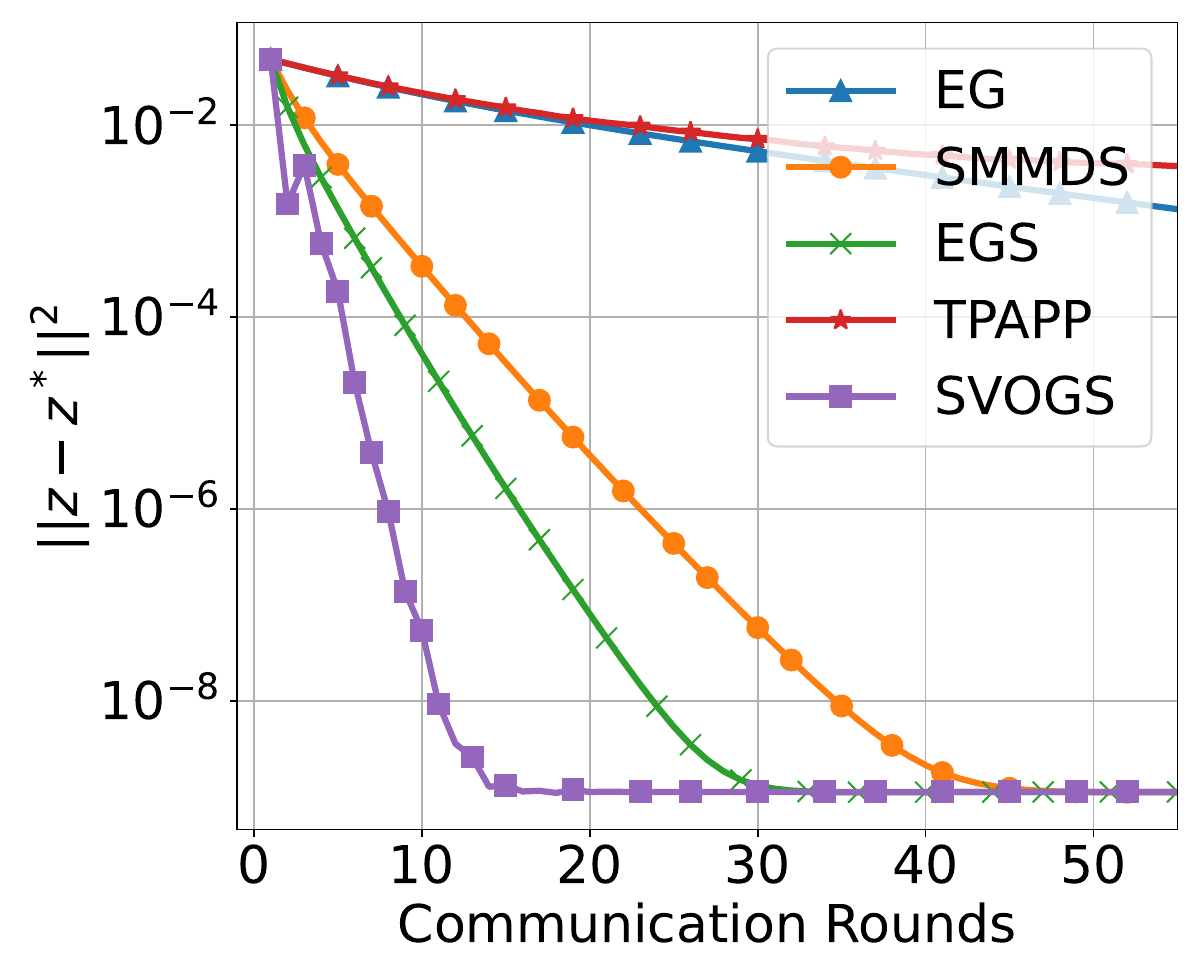}
\end{minipage}
\hfill
\begin{minipage}[b]{0.3\textwidth}
\centering
\includegraphics[width=\textwidth]{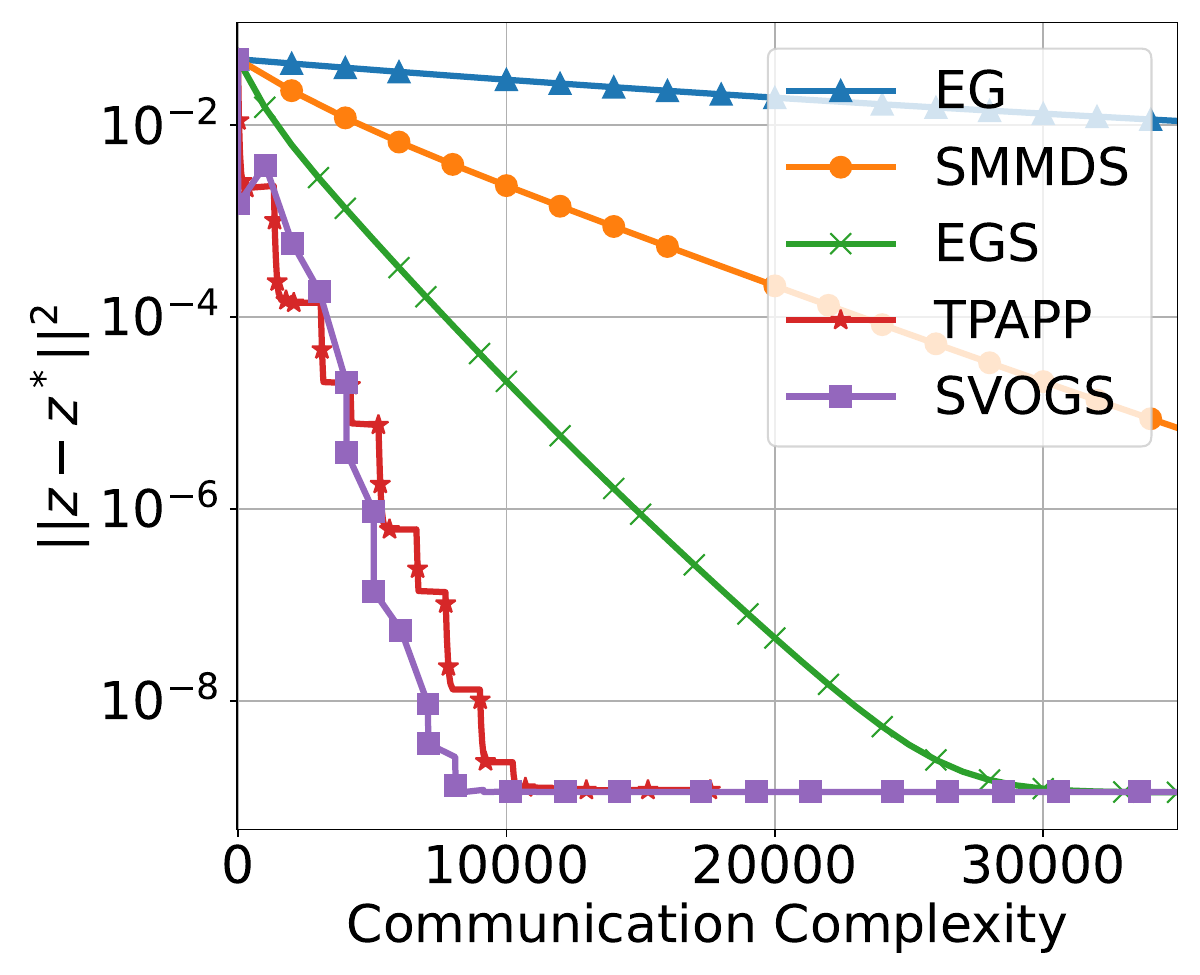}
\end{minipage}
\hfill
\begin{minipage}[b]{0.3\textwidth}
\centering
\includegraphics[width=\textwidth]{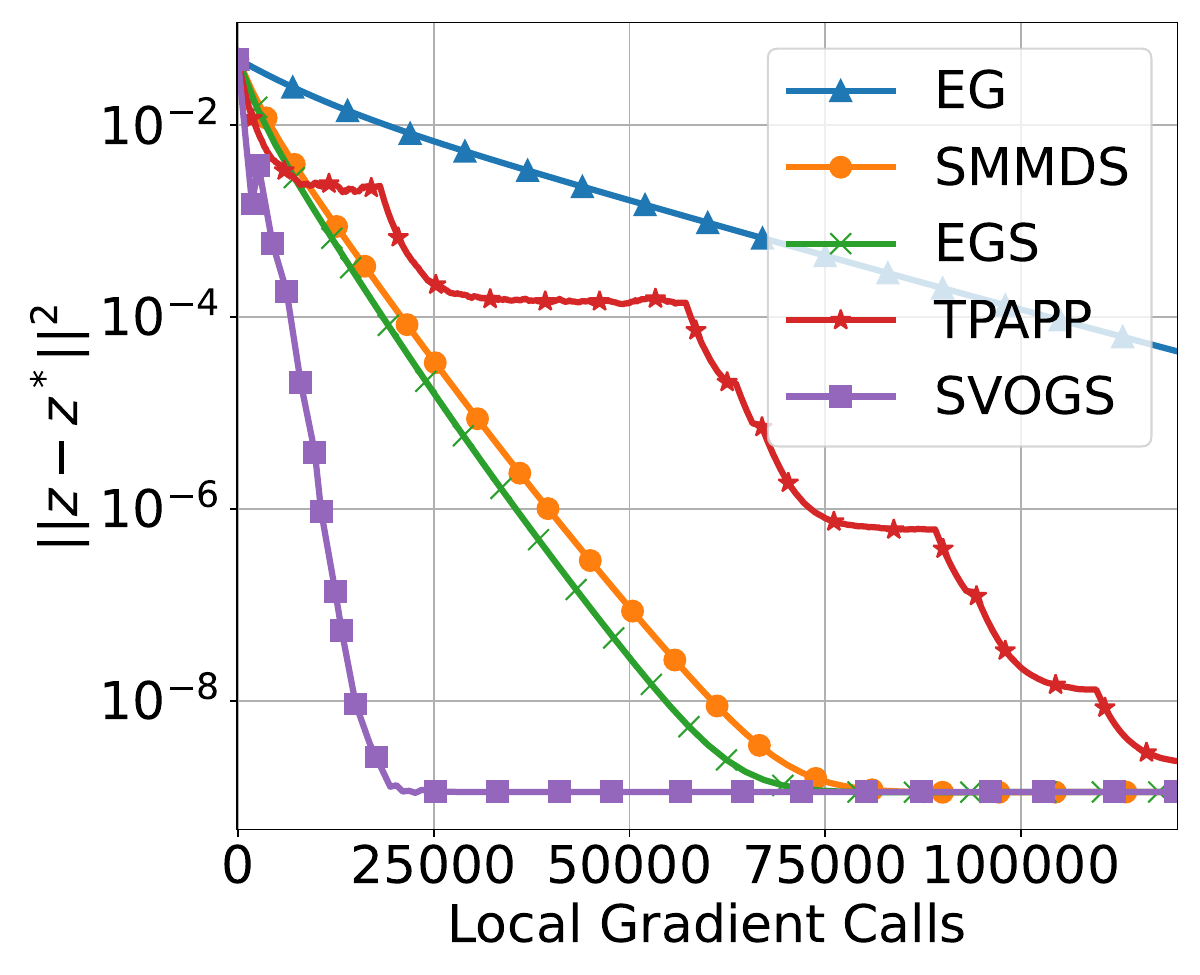}
\end{minipage}
\caption{The experimental results for strongly-convex-strongly-concave minimax problem (\ref{eq:problem2}).}\label{fig:a9a}
\end{figure}

\section{Experiments}\label{sec:experiments}

We conduct the experiment on robust linear regression~\cite{beznosikov2021distributed,hendrikx2020statistically,marteau2019globally}. 
Concretely, we consider the constrained convex-concave minimax problem 
\begin{align} \label{eq:problem1}\begin{split}
    \min_{\Norm{x}_1\leq R_x}\max_{\norm{y}\leq R_y}\frac{1}{2N}\sum_{i=1}^N\left(x^\top(a_i+y)-b_i\right)^2,
\end{split} \end{align}
and the unconstrained strongly-convex-strongly-concave minimax problem
\begin{align} \label{eq:problem2}
\begin{split}
    \min_{x\in\BR^{d'}}\max_{y\in\BR^{d'}}\frac{1}{2N}\sum_{i=1}^N\left(x^\top(a_i+y)-b_i\right)^2+\frac{\lambda}{2}\|x\|^2-\frac{\beta}{2}\|y\|^2,
\end{split} \end{align}
where $x$ contains the weights of the model, $y$ describes the noise, and $\{(a_i, b_i)\}_{i=1}^N$ is the training set.

We compare the proposed SVOGS (Algorithm \ref{alg:SVOGS}) with baselines Extra-Gradient method (EG) \cite{korpelevich1976extragradient}, Star Min-Max Data Similarity algorithm (SMMDS) \cite{beznosikov2021distributed}, Extra-Gradient Sliding (EGS) \cite{kovalev2022optimalgradient}), and Three Pillars Algorithm with Partial Participation (TPAPP) \cite{beznosikov2024similarity}.
We test the algorithms on real-world dataset ``a9a'' ($N=32,561$ and $d'=123$) from LIBSVM repository \cite{chang2011libsvm} and set the nodes number be $n=500$.
For problem (\ref{eq:problem1}), we set $R_x=2$ and $R_y=0.05$, respectively.

We present the experimental results in Figure \ref{fig:a9acc} and \ref{fig:a9a}. 
We can observe that our SVOGS outperforms all baselines in terms of the local gradient complexity.
Additionally, the SVOGS requires less communication rounds than  classical EG and existing partial participation method TPAPP, and it requires significantly less communication complexity than full participation methods EG, SMMDS and EGS.
All of these empirical results support our theoretical analysis.

\section{Conclusion}\label{sec:conclusion}

This paper presents a novel distributed optimization method named SVOGS, which use the second-order similarity in local functions and the finite-sum structure in objective to solve the convex-concave minimax problem within the near-optimal complexity.
Our theoretical results are also validated by the numerical experiments.
In future work, it is interesting to use our ideas to improve the efficiency of distributed nonconvex minimax optimization under the second-order similarity.

\bibliographystyle{plainnat}
\bibliography{reference}

\begin{thebibliography}{56}
\providecommand{\natexlab}[1]{#1}
\providecommand{\url}[1]{\texttt{#1}}
\expandafter\ifx\csname urlstyle\endcsname\relax
  \providecommand{\doi}[1]{doi: #1}\else
  \providecommand{\doi}{doi: \begingroup \urlstyle{rm}\Url}\fi

\bibitem[Alacaoglu and Malitsky(2022)]{alacaoglu2022stochastic}
Ahmet Alacaoglu and Yura Malitsky.
\newblock Stochastic variance reduction for variational inequality methods.
\newblock In \emph{Conference on Learning Theory}, 2022.

\bibitem[Allen-Zhu(2018)]{allen2018katyusha}
Zeyuan Allen-Zhu.
\newblock Katyusha: The first direct acceleration of stochastic gradient
  methods.
\newblock \emph{Journal of Machine Learning Research}, 18\penalty0
  (221):\penalty0 1--51, 2018.

\bibitem[Arjevani and Shamir(2015)]{arjevani2015communication}
Yossi Arjevani and Ohad Shamir.
\newblock Communication complexity of distributed convex learning and
  optimization.
\newblock \emph{Advances in Neural Information Processing Systems}, 2015.

\bibitem[Beznosikov and Gasnikov(2022)]{beznosikov2022compression}
Aleksandr Beznosikov and Alexander Gasnikov.
\newblock Compression and data similarity: Combination of two techniques for
  communication-efficient solving of distributed variational inequalities.
\newblock In \emph{International Conference on Optimization and Applications},
  2022.

\bibitem[Beznosikov et~al.(2021)Beznosikov, Scutari, Rogozin, and
  Gasnikov]{beznosikov2021distributed}
Aleksandr Beznosikov, Gesualdo Scutari, Alexander Rogozin, and Alexander
  Gasnikov.
\newblock Distributed saddle-point problems under data similarity.
\newblock \emph{Advances in Neural Information Processing Systems}, 2021.

\bibitem[Beznosikov et~al.(2022{\natexlab{a}})Beznosikov, Dvurechenskii,
  Koloskova, Samokhin, Stich, and Gasnikov]{beznosikov2022decentralized}
Aleksandr Beznosikov, Pavel Dvurechenskii, Anastasiia Koloskova, Valentin
  Samokhin, Sebastian~U Stich, and Alexander Gasnikov.
\newblock Decentralized local stochastic extra-gradient for variational
  inequalities.
\newblock \emph{Advances in Neural Information Processing Systems},
  2022{\natexlab{a}}.

\bibitem[Beznosikov et~al.(2022{\natexlab{b}})Beznosikov, Richt{\'a}rik,
  Diskin, Ryabinin, and Gasnikov]{beznosikov2022distributed}
Aleksandr Beznosikov, Peter Richt{\'a}rik, Michael Diskin, Max Ryabinin, and
  Alexander Gasnikov.
\newblock Distributed methods with compressed communication for solving
  variational inequalities, with theoretical guarantees.
\newblock \emph{Advances in Neural Information Processing Systems},
  2022{\natexlab{b}}.

\bibitem[Beznosikov et~al.(2023)Beznosikov, Tak{\'a}c, and
  Gasnikov]{beznosikov2024similarity}
Aleksandr Beznosikov, Martin Tak{\'a}c, and Alexander Gasnikov.
\newblock Similarity, compression and local steps: three pillars of efficient
  communications for distributed variational inequalities.
\newblock \emph{Advances in Neural Information Processing Systems}, 2023.

\bibitem[Chang and Lin(2011)]{chang2011libsvm}
Chih-Chung Chang and Chih-Jen Lin.
\newblock {LIBSVM}: a library for support vector machines.
\newblock \emph{ACM transactions on intelligent systems and technology (TIST)},
  2\penalty0 (3):\penalty0 1--27, 2011.

\bibitem[Chavdarova et~al.(2019)Chavdarova, Gidel, Fleuret, and
  Lacoste-Julien]{chavdarova2019reducing}
Tatjana Chavdarova, Gauthier Gidel, Fran{\c{c}}ois Fleuret, and Simon
  Lacoste-Julien.
\newblock Reducing noise in {GAN} training with variance reduced extragradient.
\newblock \emph{Advances in Neural Information Processing Systems}, 2019.

\bibitem[Deng and Mahdavi(2021)]{deng2021local}
Yuyang Deng and Mehrdad Mahdavi.
\newblock Local stochastic gradient descent ascent: Convergence analysis and
  communication efficiency.
\newblock In \emph{International Conference on Artificial Intelligence and
  Statistics}, 2021.

\bibitem[Gharesifard and Cortés(2014)]{6578120}
Bahman Gharesifard and Jorge Cortés.
\newblock Distributed continuous-time convex optimization on weight-balanced
  digraphs.
\newblock \emph{IEEE Transactions on Automatic Control}, 59\penalty0
  (3):\penalty0 781--786, 2014.

\bibitem[Gorbunov et~al.(2021)Gorbunov, Hanzely, and
  Richt{\'a}rik]{gorbunov2021local}
Eduard Gorbunov, Filip Hanzely, and Peter Richt{\'a}rik.
\newblock Local {SGD}: Unified theory and new efficient methods.
\newblock In \emph{International Conference on Artificial Intelligence and
  Statistics}, 2021.

\bibitem[Han et~al.(2024)Han, Xie, and Zhang]{han2024lower}
Yuze Han, Guangzeng Xie, and Zhihua Zhang.
\newblock Lower complexity bounds of finite-sum optimization problems: The
  results and construction.
\newblock \emph{Journal of Machine Learning Research}, 25\penalty0
  (2):\penalty0 1--86, 2024.

\bibitem[Hendrikx et~al.(2020)Hendrikx, Xiao, Bubeck, Bach, and
  Massoulie]{hendrikx2020statistically}
Hadrien Hendrikx, Lin Xiao, Sebastien Bubeck, Francis Bach, and Laurent
  Massoulie.
\newblock Statistically preconditioned accelerated gradient method for
  distributed optimization.
\newblock In \emph{International conference on machine learning}, 2020.

\bibitem[Hou et~al.(2021)Hou, Thekumparampil, Fanti, and Oh]{hou2021efficient}
Charlie Hou, Kiran~K. Thekumparampil, Giulia Fanti, and Sewoong Oh.
\newblock Efficient algorithms for federated saddle point optimization.
\newblock \emph{arXiv preprint:2102.06333}, 2021.

\bibitem[Ibrahim et~al.(2020)Ibrahim, Azizian, Gidel, and
  Mitliagkas]{ibrahim2020linear}
Adam Ibrahim, Wa{\i}ss Azizian, Gauthier Gidel, and Ioannis Mitliagkas.
\newblock Linear lower bounds and conditioning of differentiable games.
\newblock In \emph{International conference on machine learning}, 2020.

\bibitem[Kairouz et~al.(2021)Kairouz, McMahan, Avent, Bellet, Bennis, Bhagoji,
  Bonawitz, Charles, Cormode, Cummings, et~al.]{kairouz2021advances}
Peter Kairouz, H.~Brendan McMahan, Brendan Avent, Aur{\'e}lien Bellet, Mehdi
  Bennis, Arjun~Nitin Bhagoji, Kallista Bonawitz, Zachary Charles, Graham
  Cormode, Rachel Cummings, et~al.
\newblock Advances and open problems in federated learning.
\newblock \emph{Foundations and trends{\textregistered} in machine learning},
  14\penalty0 (1--2):\penalty0 1--210, 2021.

\bibitem[Karimireddy et~al.(2020)Karimireddy, Kale, Mohri, Reddi, Stich, and
  Suresh]{karimireddy2020scaffold}
Sai~Praneeth Karimireddy, Satyen Kale, Mehryar Mohri, Sashank Reddi, Sebastian
  Stich, and Ananda~Theertha Suresh.
\newblock {SCAFFOLD}: Stochastic controlled averaging for federated learning.
\newblock In \emph{International conference on machine learning}, 2020.

\bibitem[Khaled and Jin(2023)]{khaled2022faster}
Ahmed Khaled and Chi Jin.
\newblock Faster federated optimization under second-order similarity.
\newblock In \emph{International Conference on Learning Representations}, 2023.

\bibitem[Khaled et~al.(2019)Khaled, Mishchenko, and
  Richt{\'a}rik]{khaled2019first}
Ahmed Khaled, Konstantin Mishchenko, and Peter Richt{\'a}rik.
\newblock First analysis of local {GD} on heterogeneous data.
\newblock \emph{arXiv preprint:1909.04715}, 2019.

\bibitem[Khaled et~al.(2020)Khaled, Mishchenko, and
  Richt{\'a}rik]{khaled2020tighter}
Ahmed Khaled, Konstantin Mishchenko, and Peter Richt{\'a}rik.
\newblock Tighter theory for local {SGD} on identical and heterogeneous data.
\newblock In \emph{International Conference on Artificial Intelligence and
  Statistics}, 2020.

\bibitem[Kim and Boyd(2007)]{4434853}
Seung-Jean Kim and Stephen Boyd.
\newblock A minimax theorem with applications to machine learning, signal
  processing, and finance.
\newblock In \emph{IEEE Conference on Decision and Control}, 2007.

\bibitem[Korpelevich(1976)]{korpelevich1976extragradient}
Galina~M. Korpelevich.
\newblock The extragradient method for finding saddle points and other
  problems.
\newblock \emph{Matecon}, 12:\penalty0 747--756, 1976.

\bibitem[Kovalev et~al.(2022{\natexlab{a}})Kovalev, Beznosikov, Borodich,
  Gasnikov, and Scutari]{kovalev2022optimalgradient}
Dmitry Kovalev, Aleksandr Beznosikov, Ekaterina Borodich, Alexander Gasnikov,
  and Gesualdo Scutari.
\newblock Optimal gradient sliding and its application to optimal distributed
  optimization under similarity.
\newblock \emph{Advances in Neural Information Processing Systems},
  2022{\natexlab{a}}.

\bibitem[Kovalev et~al.(2022{\natexlab{b}})Kovalev, Beznosikov, Sadiev,
  Persiianov, Richt{\'a}rik, and Gasnikov]{kovalev2022optimal}
Dmitry Kovalev, Aleksandr Beznosikov, Abdurakhmon Sadiev, Michael Persiianov,
  Peter Richt{\'a}rik, and Alexander Gasnikov.
\newblock Optimal algorithms for decentralized stochastic variational
  inequalities.
\newblock \emph{Advances in Neural Information Processing Systems},
  2022{\natexlab{b}}.

\bibitem[Lan(2016)]{lan2016gradient}
Guanghui Lan.
\newblock Gradient sliding for composite optimization.
\newblock \emph{Mathematical Programming}, 159:\penalty0 201--235, 2016.

\bibitem[Li et~al.(2019)Li, Huang, Yang, Wang, and Zhang]{li2019convergence}
Xiang Li, Kaixuan Huang, Wenhao Yang, Shusen Wang, and Zhihua Zhang.
\newblock On the convergence of {FedAvg} on non-iid data.
\newblock \emph{International Conference on Learning Representations}, 2019.

\bibitem[Lin et~al.(2024)Lin, Han, Ye, and Zhang]{lin2023stochastic}
Dachao Lin, Yuze Han, Haishan Ye, and Zhihua Zhang.
\newblock Stochastic distributed optimization under average second-order
  similarity: Algorithms and analysis.
\newblock \emph{Advances in Neural Information Processing Systems}, 36, 2024.

\bibitem[Luo et~al.(2019)Luo, Chen, Li, Xie, and Zhang]{luo2019stochastic}
Luo Luo, Cheng Chen, Yujun Li, Guangzeng Xie, and Zhihua Zhang.
\newblock A stochastic proximal point algorithm for saddle-point problems.
\newblock \emph{arXiv preprint:1909.06946}, 2019.

\bibitem[Luo et~al.(2021)Luo, Xie, Zhang, and Zhang]{luo2021near}
Luo Luo, Guangzeng Xie, Tong Zhang, and Zhihua Zhang.
\newblock Near optimal stochastic algorithms for finite-sum unbalanced
  convex-concave minimax optimization.
\newblock \emph{arXiv preprint:2106.01761}, 2021.

\bibitem[Malinovsky et~al.(2022)Malinovsky, Yi, and
  Richt{\'a}rik]{malinovsky2022variance}
Grigory Malinovsky, Kai Yi, and Peter Richt{\'a}rik.
\newblock Variance reduced {ProxSkip}: algorithm, theory and application to
  federated learning.
\newblock \emph{Advances in Neural Information Processing Systems}, 2022.

\bibitem[Malitsky(2015)]{malitsky2015projected}
Yu~Malitsky.
\newblock Projected reflected gradient methods for monotone variational
  inequalities.
\newblock \emph{SIAM Journal on Optimization}, 25\penalty0 (1):\penalty0
  502--520, 2015.

\bibitem[Malitsky and Tam(2020)]{malitsky2020forward}
Yura Malitsky and Matthew~K. Tam.
\newblock A forward-backward splitting method for monotone inclusions without
  cocoercivity.
\newblock \emph{SIAM Journal on Optimization}, 30\penalty0 (2):\penalty0
  1451--1472, 2020.

\bibitem[Marteau-Ferey et~al.(2019)Marteau-Ferey, Bach, and
  Rudi]{marteau2019globally}
Ulysse Marteau-Ferey, Francis Bach, and Alessandro Rudi.
\newblock Globally convergent newton methods for ill-conditioned generalized
  self-concordant losses.
\newblock \emph{Advances in Neural Information Processing Systems}, 2019.

\bibitem[Mishchenko et~al.(2022)Mishchenko, Malinovsky, Stich, and
  Richt{\'a}rik]{mishchenko2022proxskip}
Konstantin Mishchenko, Grigory Malinovsky, Sebastian Stich, and Peter
  Richt{\'a}rik.
\newblock {ProxSkip}: Yes! local gradient steps provably lead to communication
  acceleration! finally!
\newblock In \emph{International Conference on Machine Learning}, pages
  15750--15769. PMLR, 2022.

\bibitem[Mitra et~al.(2021)Mitra, Jaafar, Pappas, and Hassani]{mitra2021linear}
Aritra Mitra, Rayana Jaafar, George~J. Pappas, and Hamed Hassani.
\newblock Linear convergence in federated learning: Tackling client
  heterogeneity and sparse gradients.
\newblock \emph{Advances in Neural Information Processing Systems}, 2021.

\bibitem[Nemirovski(2004)]{Nemirovski2004ProxMethodWR}
Arkadi Nemirovski.
\newblock Prox-method with rate of convergence {$O(1/t)$} for variational
  inequalities with lipschitz continuous monotone operators and smooth
  convex-concave saddle point problems.
\newblock \emph{SIAM Journal on Optimization}, 15:\penalty0 229--251, 2004.

\bibitem[Nesterov(2007)]{nesterov2007dual}
Yurii Nesterov.
\newblock Dual extrapolation and its applications to solving variational
  inequalities and related problems.
\newblock \emph{Mathematical Programming}, 109\penalty0 (2):\penalty0 319--344,
  2007.

\bibitem[Nesterov(2013)]{nesterov2013introductory}
Yurii Nesterov.
\newblock \emph{Introductory lectures on convex optimization: A basic course},
  volume~87.
\newblock Springer Science \& Business Media, 2013.

\bibitem[Notarnicola et~al.(2019)Notarnicola, Franceschelli, and
  Notarstefano]{8472154}
Ivano Notarnicola, Mauro Franceschelli, and Giuseppe Notarstefano.
\newblock A duality-based approach for distributed min–max optimization.
\newblock \emph{IEEE Transactions on Automatic Control}, 64\penalty0
  (6):\penalty0 2559--2566, 2019.

\bibitem[Ouyang and Xu(2021)]{ouyang2021lower}
Yuyuan Ouyang and Yangyang Xu.
\newblock Lower complexity bounds of first-order methods for convex-concave
  bilinear saddle-point problems.
\newblock \emph{Mathematical Programming}, 185\penalty0 (1):\penalty0 1--35,
  2021.

\bibitem[Popov(1980)]{popov1980modification}
Leonid~Denisovich Popov.
\newblock A modification of the arrow-hurwicz method for search of saddle
  points.
\newblock \emph{Mathematical notes of the Academy of Sciences of the USSR},
  28:\penalty0 845--848, 1980.

\bibitem[Razaviyayn et~al.(2020)Razaviyayn, Huang, Lu, Nouiehed, Sanjabi, and
  Hong]{Razaviyayn_2020}
Meisam Razaviyayn, Tianjian Huang, Songtao Lu, Maher Nouiehed, Maziar Sanjabi,
  and Mingyi Hong.
\newblock Nonconvex min-max optimization: Applications, challenges, and recent
  theoretical advances.
\newblock \emph{IEEE Signal Processing Magazine}, 37\penalty0 (5):\penalty0
  55–66, 2020.

\bibitem[Rogozin et~al.(2021)Rogozin, Beznosikov, Dvinskikh, Kovalev,
  Dvurechensky, and Gasnikov]{rogozin2021decentralized}
Alexander Rogozin, Aleksandr Beznosikov, Darina Dvinskikh, Dmitry Kovalev,
  Pavel Dvurechensky, and Alexander Gasnikov.
\newblock Decentralized distributed optimization for saddle point problems.
\newblock \emph{arXiv preprint:2102.07758}, 2021.

\bibitem[Shamir et~al.(2014)Shamir, Srebro, and Zhang]{shamir2014communication}
Ohad Shamir, Nati Srebro, and Tong Zhang.
\newblock Communication-efficient distributed optimization using an approximate
  newton-type method.
\newblock In \emph{International conference on machine learning}, 2014.

\bibitem[Stich(2018)]{stich2018local}
Sebastian~U. Stich.
\newblock Local {SGD} converges fast and communicates little.
\newblock \emph{arXiv preprint:1805.09767}, 2018.

\bibitem[Sun et~al.(2022)Sun, Scutari, and Daneshmand]{sun2022distributed}
Ying Sun, Gesualdo Scutari, and Amir Daneshmand.
\newblock Distributed optimization based on gradient tracking revisited:
  Enhancing convergence rate via surrogation.
\newblock \emph{SIAM Journal on Optimization}, 32\penalty0 (2):\penalty0
  354--385, 2022.

\bibitem[Szlendak et~al.(2021)Szlendak, Tyurin, and
  Richt{\'a}rik]{szlendak2021permutation}
Rafa{\l} Szlendak, Alexander Tyurin, and Peter Richt{\'a}rik.
\newblock Permutation compressors for provably faster distributed nonconvex
  optimization.
\newblock \emph{arXiv preprint:2110.03300}, 2021.

\bibitem[Tian et~al.(2022)Tian, Scutari, Cao, and
  Gasnikov]{tian2022acceleration}
Ye~Tian, Gesualdo Scutari, Tianyu Cao, and Alexander Gasnikov.
\newblock Acceleration in distributed optimization under similarity.
\newblock In \emph{International Conference on Artificial Intelligence and
  Statistics}, 2022.

\bibitem[Tseng(2000)]{tseng2000modified}
Paul Tseng.
\newblock A modified forward-backward splitting method for maximal monotone
  mappings.
\newblock \emph{SIAM Journal on Control and Optimization}, 38\penalty0
  (2):\penalty0 431--446, 2000.

\bibitem[Xidonas et~al.(2017)Xidonas, Mavrotas, Hassapis, and
  Zopounidis]{XIDONAS2017299}
Panos Xidonas, George Mavrotas, Christis Hassapis, and Constantin Zopounidis.
\newblock Robust multiobjective portfolio optimization: A minimax regret
  approach.
\newblock \emph{European Journal of Operational Research}, 262\penalty0
  (1):\penalty0 299--305, 2017.

\bibitem[Xie et~al.(2020)Xie, Luo, Lian, and Zhang]{xie2020lower}
Guangzeng Xie, Luo Luo, Yijiang Lian, and Zhihua Zhang.
\newblock Lower complexity bounds for finite-sum convex-concave minimax
  optimization problems.
\newblock In \emph{International Conference on Machine Learning}, 2020.

\bibitem[Yang et~al.(2020)Yang, Zhang, Kiyavash, and He]{yang2020catalyst}
Junchi Yang, Siqi Zhang, Negar Kiyavash, and Niao He.
\newblock A catalyst framework for minimax optimization.
\newblock \emph{Advances in Neural Information Processing Systems},
  33:\penalty0 5667--5678, 2020.

\bibitem[Zhang et~al.(2022)Zhang, Hong, and Zhang]{zhang2022lower}
Junyu Zhang, Mingyi Hong, and Shuzhong Zhang.
\newblock On lower iteration complexity bounds for the convex concave saddle
  point problems.
\newblock \emph{Mathematical Programming}, 194\penalty0 (1):\penalty0 901--935,
  2022.

\bibitem[Zhang and Lin(2015)]{pmlr-v37-zhangb15}
Yuchen Zhang and Xiao Lin.
\newblock {DiSCO}: Distributed optimization for self-concordant empirical loss.
\newblock In \emph{International Conference on Machine Learning}, 2015.

\end{thebibliography}

%%%%%%%%%%%%%%%%%%%%%%%%%%%%%%%%%%%%%%%%%%%%%%%%%%%%%%%%%%%%

%%%%%%%%%%%%%%%%%%%%%%%%%%%%%%%%%%%%%%%%%%%%%%%%%%%%%%%%%%%%

\newpage

\appendix
 {\LARGE\section*{Appendix}}
 
The appendix contains additional details supporting the main text. 
Section \ref{appx:fact} starts with some basic results.
Section \ref{appx:nonneg-Lya} shows the non-negativity of our Lyapunov function. 
Section \ref{appx:upper} provides the proof of upper bounds for the proposed method. 
Section \ref{appx:algorithm-class} formally defines the algorithm class in our lower bound analysis. 
Section \ref{appx:lowercc} and \ref{appx:lowerscsc} provide the lower complexity bound for both convex-concave and strongly-convex-strongly-concave cases. 
Section \ref{appx:gds} demonstrates the complexity of making the gradient mapping small. 
Section \ref{appx:para} shows our experimental setting.

\section{Some Basic Results}\label{appx:fact}

We introduce the following lemmas for our later analysis.

% {\color{red}
% Assumption \ref{asm:ss} leads to the Lipschitz continuous of $F(\cdot)-F_i(\cdot)$, which should be presented as a lemma.}
\begin{lem}[{\citet[Proposition B.1]{lin2023stochastic}}]\label{lem:ss}
    If the local functions $f_1,\dots,f_n:\BR^{d_x}\times\BR^{d_y}\to\BR$ hold the $\delta$-second-order similarity, then each $(F_i-F)(\cdot)$ is $\delta$-Lipschitz continuous, i.e., we have
    \begin{align*}
        \|(F_i-F)(z_1)-(F_i-F)(z_2)\|\leq \delta\|z_1-z_2\|
    \end{align*}
    for all $z_1,z_2\in\BR^d$ and $i\in[n]$.
\end{lem}

\begin{lem}[{\citet[Section 8]{alacaoglu2022stochastic}}]
\label{lem:mar}
    Let $\fF = \{\fF_k\}_{k\geq 0}$ be a filtration and $\{r^k\}$ be a stochastic process adapted to $\fF$ with $\E[r^{k+1}|\fF_k]=0$. Then for any $K \in \mathbb{N}$, $x^0 \in \fZ$, and any compact set $\fC \subset \fZ$, we have
    \begin{align*} \begin{split}
        \E\left[\max_{x\in\fC}\sum_{k=0}^{K-1}\langle  r^{k+1},x\rangle\right]\leq \max_{x\in\fC} \frac{1}{2}\|x^0-x\|^2+\frac{1}{2}\sum_{k=0}^{K-1}\E\|r^{k+1}\|^2.
    \end{split} \end{align*}
\end{lem}

In related work, \citet[Assumption 4.3]{beznosikov2024similarity} considers 
the following second-order similarity assumption that is slightly different from our Assumption \ref{asm:ss}.

\begin{asm}\label{asm:meanss}
The local functions $f_1,\dots,f_n:\BR^{d_x}\times\BR^{d_y}\to\BR$ are twice differentiable and hold the $\delta$-average-second-order similarity, 
i.e., there exists $\delta>0$ such that
\begin{align*}
    \frac{1}{n}\sum_{i=1}^n\norm{\nabla^2 f_i(x,y) - \nabla^2 f_j(x,y)}^2 \leq \delta^2
\end{align*}
for all $x\in\BR^{d_x}$, $y\in\BR^{d_y}$, and $j\in[n]$, .
\end{asm}

We present the relationship between $\delta$-average-second-order similarity (Assumption \ref{asm:meanss}) and $\delta$-second-order-similarity (Assumption \ref{asm:ss}).

\begin{prop}\label{lem:asm}
For twice differentiable local functions  $f_1,\dots,f_n:\BR^{d_x}\times\BR^{d_y}\to\BR$, we have
\begin{itemize}[leftmargin=0.5cm,topsep=-0.03cm,itemsep=-0.1cm]
\item If functions $\{f_i\}_{i=1}^n$ hold the $\delta$-average-second-order similarity, then they also hold the $\delta$-second-order similarity. 
\item If functions $\{f_i\}_{i=1}^n$ hold the $\delta$-second-order similarity, then they hold the $2\delta$-average-second-order similarity.
\end{itemize}
\end{prop}
\begin{proof}
If functions $\{f_i\}_{i=1}^n$ hold the $\delta$-average-second-order similarity, then for all $j\in [n]$, we have
\begin{align*}\begin{split}
        \|\nabla^2f_j(x,y)-\nabla^2f(x,y)\|_2^2&=\left\|\frac{1}{n}\sum_{i=1}^n\left[\nabla^2f_j(x,y)-\nabla^2f_i(x,y)\right]\right\|_2^2\\
        &\leq \frac{1}{n}\sum_{i=1}^n\|\nabla^2f_j(x,y)-\nabla^2f_i(x,y)\|_2^2\leq \delta^2,
\end{split}
\end{align*}
where we use the convexity of $\|\cdot\|_2^2$. This implies functions $\{f_i\}_{i=1}^n$ also hold  the $\delta$-second-order similarity.

If functions $\{f_i\}_{i=1}^n$ hold the $\delta$ second-order similarity, then for all $j\in [n]$, we have
    \begin{align*}\begin{split}
        &\frac{1}{n}\sum_{i=1}^n\|\nabla^2f_j(x,y)-\nabla^2f_i(x,y)\|_2^2\\
        &\leq \frac{1}{n}\sum_{i=1}^n\left(2\|\nabla^2f_j(x,y)-\nabla^2f(x,y)\|_2^2+2\|\nabla^2f(x,y)-\nabla^2f_i(x,y)\|_2^2\right) \leq (2\delta)^2,
        \end{split}
    \end{align*}
    where we use the Young's inequality for the matrix 2-norm. This implies functions $\{f_i\}_{i=1}^n$ hold the $2\delta$-average-second-order similarity.
\end{proof}

\section{The Non-Negativity of Lyapunov Function}\label{appx:nonneg-Lya}

Our convergence analysis is based on the following Lyapunov function
\begin{align*}
    \Phi^k =& \left(\frac{1}{\eta}+\mu\right)\| z ^k- z^* \|^2+2\langle F( z ^{k-1})-F_1( z ^{k-1})-F( z ^k)+F_1( z ^k), z ^{k}- z^* \rangle\\
    & +\frac{1}{64\eta}\| z ^k- z ^{k-1}\|^2+\frac{\gamma}{2\eta}\|w^{k-1}- z ^k\|^2+\frac{(2\gamma+\eta\mu)}{2p\eta}\|w^k- z^* \|^2.
\end{align*}
Noticing that we can always guarantees $\Phi^k\geq 0$ by taking $\eta\leq 1/(32\delta)$, because the Young's inequality and Lemma \ref{lem:ss} indicates
\begin{align*}
     \Phi^k&\geq \frac{1}{\eta}\| z ^k- z^* \|^2+2\langle F( z ^{k-1})-F_1( z ^{k-1})-F( z ^k)+F_1( z ^k), z ^{k}- z^* \rangle+\frac{1}{64\eta}\| z ^k- z ^{k-1}\|^2\\
     &\geq \frac{1}{\eta}\| z ^k- z^* \|^2-\frac{1}{64\eta\delta^2}\|F( z ^{k-1})-F_1( z ^{k-1})-F( z ^k)+F_1( z ^k)\|^2\\
     &\quad -64\eta\delta^2\|z ^{k}- z^*\|^2+\frac{1}{64\eta}\| z ^k- z ^{k-1}\|^2\\
     &\geq \frac{1}{\eta}\| z ^k- z^* \|^2-\frac{1}{64\eta\delta^2}\delta^2\|z^k-z^{k-1}\|^2-64\eta\delta^2\|z ^{k}- z^*\|^2+\frac{1}{64\eta}\| z ^k- z ^{k-1}\|^2\\
     &= \frac{1}{2\eta}(2-128\eta^2\delta^2)\| z ^k- z^* \|^2\geq 0.
\end{align*}

% The eq is numbered beginning at 21
\section{The Proofs for Upper Complexity Bounds}\label{appx:upper}

We provide the proofs for results in Section \ref{sec:analysis}.

\subsection{Proof of Lemma \ref{lem:convergenceLya}}

In later analysis, we denote $\E_k[\cdot]$ as the expectation with respect to the random sampled set $\fS^k$ in round $k$ and denote $\E_{k+1/2}[\cdot]$ as the expectation with respect to the random update of the snapshot point $w^k$ in round $k$.
Specifically, we take the constant
\begin{align}\label{eq:value-c}\begin{split}
c:=100+\frac{64\eta\delta^2}{\mu}+ 2048\eta^2 \delta^2 +96\eta\mu+64\sqrt{2\eta\Phi^0}
\end{split} \end{align}
for the statement of Lemma \ref{lem:convergenceLya}.

We first provide several lemmas that will be used in the proof of Lemma \ref{lem:convergenceLya}.

\begin{lem}\label{lem:lem1prooflem1}
    Under the setting of Lemma \ref{lem:convergenceLya}, we have
    \begin{align*} \begin{split}
    &-2\E \left[\langle\delta^k-\E _k[\delta^k],\hat{u}^k- z ^k\rangle\right]\\
    &\leq \frac{1}{2\eta}\E \left[\|\hat{u}^k- z ^k\|^2\right]+\frac{4\eta \delta^2}{b}\E \left[\| z ^k-w^{k-1}\|^2\right]+\frac{4\eta \delta^2\alpha^2}{b}\E \left[\| z ^k- z ^{k-1}\|^2\right],
\end{split} \end{align*}
and 
\begin{align*}
    &-2\E \left[\langle \E _k[\delta^k]+F_1(\hat{u}^k)-F( z^* ),\hat{u}^k- z^* \rangle\right]\\
    &\leq \frac{2\delta^2}{\mu}\E \left[\|\hat{u}^k-u^k\|^2  \right]+\frac{\mu}{2}\E \left[\|\hat{u}^k- z^* \|^2  \right]-2\mu\E \left[\|\hat{u}^k- z^* \|^2  \right]
    +\frac{1}{64\eta}\E \left[\| z ^k- z ^{k+1}\|^2\right]\\
    &\quad + 64\eta \delta^2 \E \left[\|\hat{u}^k-u^k\|^2  \right]-2\E \left[\langle F( z ^k)-F_1( z ^k)-F( z ^{k+1})+F_1( z ^{k+1}), z ^{k+1}- z^* \rangle\right]\\
    &\quad +\frac{1}{4\eta}\E \left[\| z ^k-\hat{u}^k\|^2\right]+4\eta\delta^2\alpha^2\E \left[\| z ^k- z ^{k-1}\|^2\right]\\
    &\quad -2\alpha\E \left[\langle F( z ^k)-F_1( z ^k)-F( z ^{k-1})+F_1( z ^{k-1}), z ^{k}- z^* \rangle\right].
\end{align*}
\end{lem}
\begin{proof}
    Firstly note that
\begin{align*}\begin{split}
    \E _k[\delta^k]=F( z ^k)-F_1( z ^k)+\alpha\left(F( z ^k) - F_1( z ^k)- F( z ^{k-1}) + F_1( z ^{k-1})\right).
\end{split} \end{align*}
According to the uniform and independent sampling and Lemma \ref{lem:ss} we have
\begin{align*} 
    \E \left[\|\delta^k-\E _k[\delta^k]\|^2\right]&\leq 2\E \left\|\frac{1}{b}\sum_{j\in \fS^k}\left(F_{j}( z ^k) - F_{j}(w^{k-1})\right)- \left(F( z ^k) -F(w^{k-1})\right)\right\|^2\\
    &\quad+2\E \left\|\frac{\alpha}{b}\sum_{j\in \fS^k}\left(F_{j}( z ^k) - F_{j}( z ^{k-1})\right)- \left(F( z ^k) -F( z ^{k-1})\right)\right\|^2\\
    &=  \frac{2}{b^2}\E\left[ \sum_{j\in \fS^k}\left\|\left(F_{j}( z ^k) - F_{j}(w^{k-1})\right)- \left(F( z ^k) -F(w^{k-1})\right)\right\|^2\right]\\
    &\quad+\frac{2\alpha^2}{b^2}\E \left[\sum_{j\in \fS^k}\left\|\left(F_{j}( z ^k) - F_{j}( z ^{k-1})\right)- \left(F( z ^k) -F( z ^{k-1})\right)\right\|^2\right]\\
    &\leq   \frac{2}{nb}\E \left[\sum_{j=1}^n\left\|\left(F_{j}( z ^k) - F_{j}(w^{k-1})\right)- \left(F( z ^k) -F(w^{k-1})\right)\right\|^2\right]\\
    &\quad+\frac{2\alpha^2}{nb}\E\left[ \sum_{j=1}^n\left\|\left(F_{j}( z ^k) - F_{j}( z ^{k-1})\right)- \left(F( z ^k) -F( z ^{k-1})\right)\right\|^2\right]\\
    &\leq \frac{2\delta^2}{b}\E \left[\| z ^k-w^{k-1}\|^2\right]+\frac{2\delta^2\alpha^2}{b}\E \left[\| z ^k- z ^{k-1}\|^2\right].
 \end{align*}

According to the above bound on $\E \left[\|\delta^k-\E _k[\delta^k]\|^2\right]$, we achieve the first result as follows
\begin{align*}\begin{split}
    &-2\E \left[\langle\delta^k-\E _k[\delta^k],\hat{u}^k- z ^k\rangle\right]\\
    &\leq \frac{1}{2\eta}\E \left[\|\hat{u}^k- z ^k\|^2\right]+2\eta\E \left[ \|\delta^k-\E _k[\delta^k]\|^2\right]\\
    &\leq \frac{1}{2\eta}\E \left[\|\hat{u}^k- z ^k\|^2\right]+\frac{4\eta \delta^2}{b}\E \left[\| z ^k-w^{k-1}\|^2\right]+\frac{4\eta \delta^2\alpha^2}{b}\E \left[\| z ^k- z ^{k-1}\|^2\right].
\end{split} \end{align*}

Again using Lemma \ref{lem:ss}, we achieve the second result as follows
\begin{align*}
    &-2\E \left[\langle \E _k[\delta^k]+F_1(\hat{u}^k)-F( z^* ),\hat{u}^k- z^* \rangle\right]\\
    &= -2\E \left[\langle F( z ^k)-F_1( z ^k)+\alpha\left(F( z ^k) - F_1( z ^k)- F( z ^{k-1}) + F_1( z ^{k-1})\right),\hat{u}^k- z^* \rangle\right]\\
    &\quad -2\E \left[\langle F_1(\hat{u}^k)-F( z^* ),\hat{u}^k- z^* \rangle\right]\\
    &= -2\E \left[\langle F(u^k)-F_1(u^k)-F(\hat{u}^k)+F_1(\hat{u}^k),\hat{u}^k- z^* \rangle\right]-2\E \left[\langle F(\hat{u}^k)-F( z^* ),\hat{u}^k- z^* \rangle\right]\\
    &\quad -2\E \left[\langle F( z ^k)-F_1( z ^k)-F( z ^{k+1})+F_1( z ^{k+1}),\hat{u}^k- z ^{k+1}\rangle\right]\\
    &\quad -2\E \left[\langle F( z ^k)-F_1( z ^k)-F( z ^{k+1})+F_1( z ^{k+1}), z ^{k+1}- z^* \rangle\right]\\
    &\quad -2\alpha\E \left[\langle F( z ^k)-F_1( z ^k)-F( z ^{k-1})+F_1( z ^{k-1}),\hat{u}^k- z ^{k}\rangle\right]\\
    &\quad -2\alpha\E \left[\langle F( z ^k)-F_1( z ^k)-F( z ^{k-1})+F_1( z ^{k-1}), z ^{k}- z^* \rangle\right]\\
    &\leq \frac{2\delta^2}{\mu}\E \left[\|\hat{u}^k-u^k\|^2  \right]+\frac{\mu}{2}\E \left[\|\hat{u}^k- z^* \|^2  \right]-2\mu\E \left[\|\hat{u}^k- z^* \|^2  \right]
    +\frac{1}{64\eta}\E \left[\| z ^k- z ^{k+1}\|^2\right]\\
    &\quad + 64\eta \delta^2 \E \left[\|\hat{u}^k-u^k\|^2  \right]-2\E \left[\langle F( z ^k)-F_1( z ^k)-F( z ^{k+1})+F_1( z ^{k+1}), z ^{k+1}- z^* \rangle\right]\\
    &\quad +\frac{1}{4\eta}\E \left[\| z ^k-\hat{u}^k\|^2\right]+4\eta\delta^2\alpha^2\E \left[\| z ^k- z ^{k-1}\|^2\right]\\
    &\quad -2\alpha\E \left[\langle F( z ^k)-F_1( z ^k)-F( z ^{k-1})+F_1( z ^{k-1}), z ^{k}- z^* \rangle\right].
 \end{align*}
\end{proof}

\begin{lem}\label{lem:lem1prooflem2}
    Under setting of Lemma \ref{lem:convergenceLya}, we have
    \begin{align*}
        &-\frac{1}{8\eta}\E \left[\|\hat{u}^k- z ^k\|^2\right]-\frac{\gamma}{\eta}\E \left[\|w^k-\hat{u}^k\|^2\right]-\frac{3\mu}{2}\E \left[\|\hat{u}^k- z^* \|^2\right]\\
        &\leq -\frac{1}{16\eta}\E \left[\|\hat{u}^k- z ^k\|^2\right]-\frac{1}{32\eta}\E \left[\| z ^{k+1}- z ^k\|^2\right]+\left(\frac{1}{8\eta}+3\mu\right)\E \left[\|\hat{u}^k-u^k\|^2\right]\\
        &\quad -\frac{\gamma}{2\eta}\E \left[\|w^k-\hat{u}^k\|^2\right]
   -\frac{\gamma}{4\eta}\E \left[\|w^k- z ^{k+1}\|^2\right] -\mu\E \left[\| z ^{k+1}- z^* \|^2\right].
    \end{align*}
\end{lem}
\begin{proof}
    From the facts $\|a+b\|^2\geq\frac{1}{2}\|a\|^2-\|b\|^2$ and $\frac{3}{2}\|a+b\|^2\geq\|a\|^2-3\|b\|^2$, we have
\begin{align*} \begin{split}
    -\frac{1}{8\eta}\E \left[\|\hat{u}^k- z ^k\|^2\right]\leq -\frac{1}{16\eta}\E \left[\|\hat{u}^k- z ^k\|^2\right]-\frac{1}{32\eta}\E \left[\| z ^{k+1}- z ^k\|^2\right]+\frac{1}{16\eta}\E \left[\|\hat{u}^k-u^k\|^2\right],
\end{split} \end{align*}
\begin{align*} \begin{split}
   -\frac{\gamma}{\eta}\E \left[\|w^k-\hat{u}^k\|^2\right]&\leq 
   -\frac{\gamma}{2\eta}\E \left[\|w^k-\hat{u}^k\|^2\right]
   -\frac{\gamma}{4\eta}\E \left[\|w^k- z ^{k+1}\|^2\right]+\frac{\gamma}{2\eta}\E \left[\|\hat{u}^k-u^k\|^2\right]\\
   &\leq 
   -\frac{\gamma}{2\eta}\E \left[\|w^k-\hat{u}^k\|^2\right]
   -\frac{\gamma}{4\eta}\E \left[\|w^k- z ^{k+1}\|^2\right]+\frac{1}{16\eta}\E \left[\|\hat{u}^k-u^k\|^2\right],
\end{split} \end{align*}
and
\begin{align*} \begin{split}
   -\frac{3\mu}{2}\E \left[\|\hat{u}^k- z^* \|^2\right]\leq -\mu\E \left[\| z ^{k+1}- z^* \|^2\right]+3\mu\E \left[\|\hat{u}^k-u^k\|^2\right],
\end{split} \end{align*}
where we use the setting $\gamma\leq 1/8$ in the second inequality.
\end{proof}

\begin{lem}\label{lem:lem1prooflem3}
Under setting of Lemma \ref{lem:convergenceLya}, we have
    \begin{align*}
        &\left(\frac{1}{\eta}+\mu\right)\E \left[\| z ^{k+1}- z^* \|^2\right]+2\E \left[\langle F( z ^k)-F_1( z ^k)-F( z ^{k+1})+F_1( z ^{k+1}), z ^{k+1}- z^* \rangle\right]\\
    &\quad +\frac{1}{64\eta}\E \left[\| z ^{k+1}- z ^k\|^2\right]+\frac{\gamma}{4\eta}\E \left[\|w^k- z ^{k+1}\|^2\right]+\frac{\gamma+\frac{1}{2}\eta\mu}{p\eta}\E \left[\|w^{k+1}- z^* \|^2\right]\\
    &\leq \left(\frac{1}{\eta}+\frac{\mu}{2}\right)\E \left[\| z ^k- z^* \|^2\right]+2\alpha\E \left[\langle F( z ^{k-1})-F_1( z ^{k-1})-F( z ^k)+F_1( z ^k), z ^{k}- z^* \rangle\right]\\
    &\quad +\alpha\frac{1}{64\eta}\E \left[\| z ^k- z ^{k-1}\|^2\right]+\frac{4\eta \delta^2}{b}\E \left[\|w^{k-1}- z ^k\|^2\right]-\frac{1}{16\eta}\E \left[\| z ^k-\hat{u}^k\|^2\right] \\
    &\quad +\left(-\frac{7}{8\eta}+\frac{2\delta^2}{\mu}+ 64\eta \delta^2 +3\mu\right)\E \left[\|\hat{u}^k-u^k\|^2  \right]+\frac{2}{\eta}\E \left[\|\hat{u}^k-u^k\|\| z ^{k+1}- z^* \|\right]\\
    &\quad +\left(1-\frac{p\eta\mu}{2\gamma+\eta\mu}\right)\frac{(\gamma+\frac{1}{2}\eta\mu)}{p\eta}\E \left[\|w^k- z^* \|^2\right]-\frac{\gamma}{2\eta}\E \left[\|w^k-\hat{u}^k\|^2\right].
    \end{align*}
\end{lem}

\begin{proof}
The optimality of $\hat{u}^k$ implies
\begin{align} \label{eq:21-main}\begin{split}
\langle \eta F_1(\hat{u}^k)+\hat{u}^k-v^k, z^* -\hat{u}^k\rangle \geq 0.
\end{split} \end{align}

% And as we defined in the algorithm,
% \begin{align*} \begin{split}
%     \delta^k&:=  F(w^{k-1}) - F_1(w^{k-1})   + \frac{1}{b}\sum_{j\in \fS^k}\left(F_{j}( z ^k) - F_1( z ^k)- F_{j}(w^{k-1}) + F_1(w^{k-1})\right)\\
%        &\quad+ \frac{\alpha}{b}\sum_{j\in \fS^k}\left(F_{j}( z ^k) - F_1( z ^k)- F_{j}( z ^{k-1}) + F_1( z ^{k-1})\right)
% \end{split} \end{align*}

Combine equation (\ref{eq:21-main}) with the update rule in Line \ref{line:updatevk} of Algorithm \ref{alg:SVOGS} and $\langle -\gamma F( z^* ), z^* -\hat{u}^k\rangle \geq 0$, we achieve
\begin{align}\label{eq:24-main} \begin{split}
    -\frac{1}{\eta}\langle \bar{ z }^k-\hat{u}^k-\eta\delta^k,\hat{u}^k- z^* \rangle \leq -\langle F_1(\hat{u}^k)-F( z^* ),\hat{u}^k- z^* \rangle.
\end{split} \end{align}

Using the result of equation (\ref{eq:24-main}), we have
\begin{align*}\begin{split}
    \frac{1}{\eta}\|\hat{u}^k- z^* \|^2&=\frac{1}{\eta}\| z ^k- z^* \|^2+\frac{2}{\eta}\langle\hat{u}^k- z ^k,\hat{u}^k- z^* \rangle-\frac{1}{\eta}\|\hat{u}^k- z ^k\|^2\\
    &=\frac{1}{\eta}\| z ^k- z^* \|^2+\frac{2\gamma}{\eta}\langle w^k- z ^k,\hat{u}^k- z^* \rangle-2\langle\delta^k,\hat{u}^k- z^* \rangle\\
    &\quad -\frac{1}{\eta}\|\hat{u}^k- z ^k\|^2-\frac{2}{\eta}\langle \bar{ z }^k-\hat{u}^k-\eta\delta^k,\hat{u}^k- z^* \rangle \\
    &\leq \frac{1}{\eta}\| z ^k- z^* \|^2+\frac{2\gamma}{\eta}\langle w^k- z ^k,\hat{u}^k- z^* \rangle-2\langle\delta^k,\hat{u}^k- z^* \rangle\\
    &\quad -\frac{1}{\eta}\|\hat{u}^k- z ^k\|^2-2\langle F_1(\hat{u}^k)-F( z^* ),\hat{u}^k- z^* \rangle\\
    % &= \frac{1}{\eta}\| z ^k- z^* \|^2+\frac{\gamma}{\eta}\|w^k- z^* \|^2-\frac{\gamma}{\eta}\|w^k-\hat{u}^k\|^2 -\frac{\gamma}{\eta}\| z ^k- z^* \|^2\\
    %  &\quad-\frac{1-\gamma}{\eta}\|\hat{u}^k- z ^k\|^2 -2\langle\delta^k,\hat{u}^k- z^* \rangle-2\langle F_1(\hat{u}^k)-F( z^* ),\hat{u}^k- z^* \rangle\\
     &= \frac{1}{\eta}\| z ^k- z^* \|^2+\frac{\gamma}{\eta}\|w^k- z^* \|^2-\frac{\gamma}{\eta}\|w^k-\hat{u}^k\|^2-\frac{\gamma}{\eta}\| z ^k- z^* \|^2\\
     &\quad -\frac{1-\gamma}{\eta}\|\hat{u}^k- z ^k\|^2-2\langle\delta^k-\E _k[\delta^k],\hat{u}^k- z ^k+ z ^k- z^* \rangle\\
     &\quad -2\langle \E _k[\delta^k]+F_1(\hat{u}^k)-F( z^* ),\hat{u}^k- z^* \rangle.
\end{split} \end{align*}
%where in the third equality we use the fact that $2(a-c)(b-d)=(a-d)^2-(a-b)^2-(c-d)^2+(c-b)^2$.
Taking the expectation on above result and using the fact 
\begin{align*}
    \E \left[\langle\delta^k-\E _k[\delta^k], z ^k- z^* \rangle\right]=\E \left[\E _k\left[\langle\delta^k-\E _k[\delta^k], z ^k- z^* \rangle\right]\right]=0,
\end{align*}
we obtain
\begin{align}\label{eq:28-main} \begin{split}
     \frac{1}{\eta}\E \left[\|\hat{u}^k- z^* \|^2\right]&\leq \E \left[\frac{1}{\eta}\| z ^k- z^* \|^2+\frac{\gamma}{\eta}\|w^k- z^* \|^2-\frac{\gamma}{\eta}\|w^k-\hat{u}^k\|^2-\frac{\gamma}{\eta}\| z ^k- z^* \|^2\right]\\
     &\quad -\E \left[\frac{1-\gamma}{\eta}\|\hat{u}^k- z ^k\|^2\right]-2\E \left[\langle\delta^k-\E _k[\delta^k],\hat{u}^k- z ^k\rangle\right]\\
     &\quad -2\E \left[\langle \E _k[\delta^k]+F_1(\hat{u}^k)-F( z^* ),\hat{u}^k- z^* \rangle\right].
\end{split} \end{align}

Applying Lemma \ref{lem:lem1prooflem1} to bound the term 
$-2\E \left[\langle\delta^k-\E _k[\delta^k],\hat{u}^k- z ^k\rangle\right]$ in equation (\ref{eq:28-main}), we obtain
\begin{align} \label{eq:33-main}
    &\frac{1}{\eta}\E \left[\|\hat{u}^k- z^* \|^2\right]\nonumber\\
    &\leq \E \left[\frac{1}{\eta}\| z ^k- z^* \|^2+\frac{\gamma}{\eta}\|w^k- z^* \|^2-\frac{\gamma}{\eta}\|w^k-\hat{u}^k\|^2-\frac{\gamma}{\eta}\| z ^k- z^* \|^2-\frac{1/4-\gamma}{\eta}\|\hat{u}^k- z ^k\|^2\right]\nonumber\\
     &\quad +\frac{4\eta \delta^2}{b}\E \left[\| z ^k-w^{k-1}\|^2\right]+\frac{4\eta \delta^2\alpha^2}{b}\E \left[\| z ^k- z ^{k-1}\|^2\right]\nonumber\\
    &\quad + \frac{2\delta^2}{\mu}\E \left[\|\hat{u}^k-u^k\|^2  \right]-\frac{3}{2}\mu\E \left[\|\hat{u}^k- z^* \|^2  \right]+\frac{1}{64\eta}\E \left[\| z ^k- z ^{k+1}\|^2\right]+ 64\eta \delta^2 \E \left[\|\hat{u}^k-u^k\|^2  \right]\nonumber\\
    &\quad -2\E \left[\langle F( z ^k)-F_1( z ^k)-F( z ^{k+1})+F_1( z ^{k+1}), z ^{k+1}- z^* \rangle\right]+4\eta\delta^2\alpha^2\E \left[\| z ^k- z ^{k-1}\|^2\right]\nonumber\\
    &\quad -2\alpha\E \left[\langle F( z ^k)-F_1( z ^k)-F( z ^{k-1})+F_1( z ^{k-1}), z ^{k}- z^* \rangle\right]\nonumber\\
    &\leq \frac{1}{\eta}\E \left[\| z ^k- z^* \|^2\right]+\frac{\gamma}{\eta}\E \left[\|w^k- z^* \|^2\right]-\frac{\gamma}{\eta}\E \left[\|w^k-\hat{u}^k\|^2\right]-\frac{\gamma}{\eta}\E \left[\| z ^k- z^* \|^2\right]\nonumber\\
     &\quad -\frac{1}{8\eta}\E \left[\|\hat{u}^k- z ^k\|^2\right]+\frac{1}{64\eta}\E \left[\| z ^k- z ^{k+1}\|^2\right]+\frac{4\eta \delta^2}{b}\E \left[\| z ^k-w^{k-1}\|^2\right]\nonumber\\
    &\quad +\left(\frac{4\eta \delta^2\alpha^2}{b}+4\eta\delta^2\alpha^2\right)\E \left[\| z ^k- z ^{k-1}\|^2\right]+\left(\frac{2\delta^2}{\mu}+ 64\eta \delta^2 \right)\E \left[\|\hat{u}^k-u^k\|^2  \right] \\
    &\quad -2\E \left[\langle F( z ^k)-F_1( z ^k)-F( z ^{k+1})+F_1( z ^{k+1}), z ^{k+1}- z^* \rangle\right]-\frac{3\mu}{2}\E \left[\|\hat{u}^k- z^* \|^2\right]\nonumber\\
    &\quad -2\alpha\E \left[\langle F( z ^k)-F_1( z ^k)-F( z ^{k-1})+F_1( z ^{k-1}), z ^{k}- z^* \rangle\right],\nonumber
 \end{align}
 where we use the setting $\gamma\leq 1/8$ in the second inequality.

Then we consider the terms related to $\hat{u}$. 
Firstly, we have
\begin{align}\label{eq:34-hatu1}\begin{split}
     \frac{1}{\eta}\E \left[\|\hat{u}^k- z^* \|^2\right]=
     \frac{1}{\eta}\E \left[\| z ^{k+1}- z^* \|^2\right]+\frac{1}{\eta}\E \left[\|\hat{u}^k-u^k\|^2\right]-\frac{2}{\eta}\E \left[\|\hat{u}^k-u^k\|\| z ^{k+1}- z^* \|\right].
\end{split} \end{align}

Applying Lemma \ref{lem:lem1prooflem2} and plugging equation (\ref{eq:34-hatu1}) into equation (\ref{eq:33-main}), we have
\begin{align} \label{eq:39-main}\begin{split}
    & \frac{1}{\eta}\E \left[\| z ^{k+1}- z^* \|^2\right]+2\E \left[\langle F( z ^k)-F_1( z ^k)-F( z ^{k+1})+F_1( z ^{k+1}), z ^{k+1}- z^* \rangle\right]\\
    &\quad +\frac{1}{64\eta}\E \left[\| z ^{k+1}- z ^k\|^2\right]+\frac{\gamma}{4\eta}\E \left[\|w^k- z ^{k+1}\|^2\right]\\
    &\leq \frac{1}{\eta}\E \left[\| z ^k- z^* \|^2\right]+2\alpha\E \left[\langle F( z ^{k-1})-F_1( z ^{k-1})-F( z ^k)+F_1( z ^k), z ^{k}- z^* \rangle\right]\\
    &\quad + \frac{\alpha}{64\eta}\E \left[\| z ^k- z ^{k-1}\|^2\right]+\frac{\gamma}{\eta}\E \left[\|w^k- z^* \|^2\right]-\frac{\gamma}{\eta}\E \left[\| z ^k- z^* \|^2\right]-\mu\E \left[\| z ^{k+1}- z^* \|^2\right]\\
     &\quad +\frac{4\eta \delta^2}{b}\E \left[\| z ^k-w^{k-1}\|^2\right]+\frac{2}{\eta}\E \left[\|\hat{u}^k-u^k\|\| z ^{k+1}- z^* \|\right]-\frac{1}{16\eta}\E \left[\|\hat{u}^k- z ^k\|^2\right]\\
    &\quad +\left(-\frac{7}{8\eta}+\frac{2\delta^2}{\mu}+ 64\eta \delta^2 +3\mu\right)\E \left[\|\hat{u}^k-u^k\|^2  \right]-\frac{\gamma}{2\eta}\E \left[\|w^k-\hat{u}^k\|^2\right],
\end{split} \end{align}
where we use the fact that ${256\eta^2 \delta^2\alpha^2}/{b}+256\eta^2\delta^2\alpha^2\leq \alpha$ to bound the coefficient before the term of $\E \left[\| z ^k- z ^{k-1}\|^2\right]$.

Then we add the term
\begin{align}\label{eq:40-update1} \begin{split}
    \mu\E \left[\| z ^{k+1}- z^* \|^2\right]+\frac{\gamma+\frac{1}{2}\eta\mu}{p\eta}\E \left[\|w^{k+1}- z^* \|^2\right]
\end{split} \end{align}
to both sides of equation (\ref{eq:39-main}) and use the update rule in Line \ref{line:update} of Algorithm \ref{alg:SVOGS} to obtain
\begin{align*}\begin{split}
     \frac{\gamma+\frac{1}{2}\eta\mu}{p\eta}\E \left[\|w^{k+1}- z^* \|^2\right]&=\frac{\gamma+\frac{1}{2}\eta\mu}{p\eta}\E \left[\E _{w^{k+1}}\left[\|w^{k+1}- z^* \|^2\right]\right]\\
     &=\frac{\gamma+\frac{1}{2}\eta\mu}{\eta}\E \left[\| z ^{k}- z^* \|^2\right]+\frac{(\gamma+\frac{1}{2}\eta\mu)(1-p)}{p\eta}\E \left[\|w^{k}- z^* \|^2\right],
\end{split} \end{align*}
and 
\begin{align*}\begin{split}
    \frac{\gamma}{\eta}+\frac{(\gamma+\frac{1}{2}\eta\mu)(1-p)}{p\eta}=\left(1-p+\frac{p\gamma}{(\gamma+\frac{1}{2}\eta\mu)}\right)\frac{(\gamma+\frac{1}{2}\eta\mu)}{p\eta}=\left(1-\frac{p\eta\mu}{2\gamma+\eta\mu}\right)\frac{(\gamma+\frac{1}{2}\eta\mu)}{p\eta}.
\end{split} \end{align*}

Combining all above results, we achieve
\begin{align*}
    &\left(\frac{1}{\eta}+\mu\right)\E \left[\| z ^{k+1}- z^* \|^2\right]+2\E \left[\langle F( z ^k)-F_1( z ^k)-F( z ^{k+1})+F_1( z ^{k+1}), z ^{k+1}- z^* \rangle\right]\\
    &\quad +\frac{1}{64\eta}\E \left[\| z ^{k+1}- z ^k\|^2\right]+\frac{\gamma}{4\eta}\E \left[\|w^k- z ^{k+1}\|^2\right]+\frac{\gamma+\frac{1}{2}\eta\mu}{p\eta}\E \left[\|w^{k+1}- z^* \|^2\right]\\
    &\leq \left(\frac{1}{\eta}+\frac{\mu}{2}\right)\E \left[\| z ^k- z^* \|^2\right]+2\alpha\E \left[\langle F( z ^{k-1})-F_1( z ^{k-1})-F( z ^k)+F_1( z ^k), z ^{k}- z^* \rangle\right] \\
    &\quad + \frac{\alpha}{64\eta}\E \left[\| z ^k- z ^{k-1}\|^2\right]+\frac{4\eta \delta^2}{b}\E \left[\|w^{k-1}- z ^k\|^2\right]-\frac{1}{16\eta}\E \left[\| z ^k-\hat{u}^k\|^2\right] \\
    &\quad +\left(-\frac{7}{8\eta}+\frac{2\delta^2}{\mu}+ 64\eta \delta^2 +3\mu\right)\E \left[\|\hat{u}^k-u^k\|^2  \right]+\frac{2}{\eta}\E \left[\|\hat{u}^k-u^k\|\| z ^{k+1}- z^* \|\right]\\
    &\quad +\left(1-\frac{p\eta\mu}{2\gamma+\eta\mu}\right)\frac{(\gamma+\frac{1}{2}\eta\mu)}{p\eta}\E \left[\|w^k- z^* \|^2\right]-\frac{\gamma}{2\eta}\E \left[\|w^k-\hat{u}^k\|^2\right].
\end{align*}
\end{proof}

%\subsection{Proof of Lemma \ref{lem:convergenceLya}}

\begin{lem}\label{lem:lem1prooflem4}
Under setting of Lemma \ref{lem:convergenceLya}, we additionally assume $\E[\Phi^k]\leq \Phi^0 $ holds, then we have
\begin{align*} 
    \E [\Phi^{k+1}]\leq \max\left\{1-\frac{\eta\mu}{6},1-\frac{p\eta\mu}{2\gamma+\eta\mu}\right\}\E [\Phi^{k}]-\frac{1}{16\eta}\E \left[\| z ^k-\hat{u}^k\|^2\right]-\frac{\gamma}{2\eta}\E \left[\|w^k-\hat{u}^k\|^2\right].
 \end{align*}
\end{lem}

\begin{proof}

Recall that the definition of our Lyapunov function is
\begin{align}\label{eq:45-lya} \begin{split}
   \Phi^k =& \left(\frac{1}{\eta}+\mu\right)\| z ^k- z^* \|^2+2\langle F( z ^{k-1})-F_1( z ^{k-1})-F( z ^k)+F_1( z ^k), z ^{k}- z^* \rangle\\
    & +\frac{1}{64\eta}\| z ^k- z ^{k-1}\|^2+\frac{\gamma}{4\eta}\|w^{k-1}- z ^k\|^2+\frac{(2\gamma+\eta\mu)}{2p\eta}\|w^k- z^* \|^2.
\end{split} \end{align}

Recall that we take constant $c$ by equation (\ref{eq:value-c}), then the condition 
\begin{align*}
\varepsilon_k\leq c^{-1}\min\left\{\|\hat{u}^k- z ^k\|,\|\hat{u}^k- z ^k\|^2\right\}    
\end{align*}
guarantees 
\begin{align*} \begin{split}
    \E \left[\|\hat{u}^k-u^k\|\right]\leq \tilde{\zeta}^k\min\left\{\E \left[\|\hat{u}^k- z ^k\|\right],\E \left[\|\hat{u}^k- z ^k\|^2\right]\right\},
\end{split} \end{align*}
where
\begin{align} \label{eq:49-tildee}\begin{split}
    \tilde{\zeta}^k&=\frac{1}{32\eta}\frac{1}{\frac{9}{8\eta}+\frac{2\delta^2}{\mu}+ 64\eta \delta^2 +3\mu+\frac{2}{\eta}+2\sqrt{\frac{2\E[\Phi^k]}{\eta}}}\\
    &=\frac{1}{100+\frac{64\eta\delta^2}{\mu}+ 2048\eta^2 \delta^2 +96\eta\mu+64\sqrt{2\eta\E[\Phi^k]}}\\
    &\geq \frac{1}{100+\frac{64\eta\delta^2}{\mu}+ 2048\eta^2 \delta^2 +96\eta\mu+64\sqrt{2\eta\Phi^0}}=\frac{1}{c}.
\end{split} \end{align}
The inequality (\ref{eq:49-tildee}) is based on the assumption $\E[\Phi^k]\leq \Phi^0 $.

Note that we have $\| z ^{k+1}- z^* \|\leq \|u^k-\hat{u}^k\|+\|\hat{u}^k- z ^k\|+\| z ^k- z^* \|$, then
\begin{align*} 
    &\left(-\frac{7}{8\eta}+\frac{2\delta^2}{\mu}+ 64\eta \delta^2 +3\mu\right)\E \left[\|\hat{u}^k-u^k\|^2  \right]+\frac{2}{\eta}\E \left[\|\hat{u}^k-u^k\|\| z ^{k+1}- z^* \|\right]\\
    &\leq \left(\frac{9}{8\eta}+\frac{2\delta^2}{\mu}+ 64\eta \delta^2 +3\mu\right)\E \left[\|\hat{u}^k-u^k\|^2  \right]+\frac{2}{\eta}\E \left[\|\hat{u}^k-u^k\|\|\hat{u}^k- z ^k\|\right]\\
    &\quad +\frac{2}{\eta}\E \left[\|\hat{u}^k-u^k\|\| z ^k- z^* \|\right]\\
    &\leq \left(\frac{9}{8\eta}+\frac{2\delta^2}{\mu}+ 64\eta \delta^2 +3\mu\right)\E \left[\|\hat{u}^k-u^k\|^2  \right]+\frac{2}{\eta}\E \left[\|\hat{u}^k-u^k\|\|\hat{u}^k- z ^k\|\right]\\
    &\quad +2\sqrt{\frac{2\E[\Phi^k]}{\eta}}\E \left[\|\hat{u}^k-u^k\|\right]\\
    &\leq \frac{1}{32\eta}\E \left[\|\hat{u}^k- z ^k\|^2\right].
 \end{align*}

According to Lemma \ref{lem:lem1prooflem3}, we have
\begin{align*}
    &\left(\frac{1}{\eta}+\mu\right)\E \left[\| z ^{k+1}- z^* \|^2\right]+2\E \left[\langle F( z ^k)-F_1( z ^k)-F( z ^{k+1})+F_1( z ^{k+1}), z ^{k+1}- z^* \rangle\right]\\
    &\quad +\frac{1}{64\eta}\E \left[\| z ^{k+1}- z ^k\|^2\right]+\frac{\gamma}{4\eta}\E \left[\|w^k- z ^{k+1}\|^2\right]+\frac{\gamma+\frac{1}{2}\eta\mu}{p\eta}\E \left[\|w^{k+1}- z^* \|^2\right]\\
    &\leq \left(\frac{1}{\eta}+\frac{\mu}{2}\right)\E \left[\| z ^k- z^* \|^2\right]+2\alpha\E \left[\langle F( z ^{k-1})-F_1( z ^{k-1})-F( z ^k)+F_1( z ^k), z ^{k}- z^* \rangle\right]\\
    &\quad +\alpha\frac{1}{64\eta}\E \left[\| z ^k- z ^{k-1}\|^2\right]+\frac{4\eta \delta^2}{b}\E \left[\|w^{k-1}- z ^k\|^2\right]+\frac{1}{32\eta}\E \left[\| z ^k-\hat{u}^k\|^2\right]  \\
    &\quad -\frac{1}{16\eta}\E \left[\| z ^k-\hat{u}^k\|^2\right]-\frac{\gamma}{2\eta}\E \left[\|w^k-\hat{u}^k\|^2\right]+\left(1-\frac{p\eta\mu}{2\gamma+\eta\mu}\right)\frac{(\gamma+\frac{1}{2}\eta\mu)}{p\eta}\E \left[\|w^k- z^* \|^2\right].
\end{align*}

From the facts ${4\eta \delta^2}/{b}\leq \alpha{\gamma}/{(4\eta)}$ and 
$\eta\mu\leq 1$, we obtain $(1/\eta+\mu/2)\leq (1/\eta+\mu)(1-\eta\mu/6)$.
Thus, it holds
\begin{align*} \begin{split}
    &\left(\frac{1}{\eta}+\mu\right)\E \left[\| z ^{k+1}- z^* \|^2\right]+2\E \left[\langle F( z ^k)-F_1( z ^k)-F( z ^{k+1})+F_1( z ^{k+1}), z ^{k+1}- z^* \rangle\right]\\
    &\quad +\frac{1}{64\eta}\E \left[\| z ^{k+1}- z ^k\|^2\right]+\frac{\gamma}{4\eta}\E \left[\|w^k- z ^{k+1}\|^2\right]+\frac{\gamma+\frac{1}{2}\eta\mu}{p\eta}\E \left[\|w^{k+1}- z^* \|^2\right]\\
    &\leq \left(1-\frac{\eta\mu}{6}\right)\left(\frac{1}{\eta}+\mu\right)\E \left[\| z ^k- z^* \|^2\right]+\left(1-\frac{p\eta\mu}{2\gamma+\eta\mu}\right)\frac{(\gamma+\frac{1}{2}\eta\mu)}{p\eta}\E \left[\|w^k- z^* \|^2\right]\\
    &\quad +2\alpha\E \left[\langle F( z ^{k-1})-F_1( z ^{k-1})-F( z ^k)+F_1( z ^k), z ^{k}- z^* \rangle\right]+\alpha\frac{1}{64\eta}\E \left[\| z ^k- z ^{k-1}\|^2\right]\\
&\quad +\alpha\frac{\gamma}{2\eta}\E \left[\|w^{k-1}- z ^k\|^2\right]-\frac{1}{16\eta}\E \left[\| z ^k-\hat{u}^k\|^2\right]-\frac{\gamma}{2\eta}\E \left[\|w^k-\hat{u}^k\|^2\right].
\end{split} \end{align*}

The definition (\ref{eq:45-lya}) and the setting $\alpha=\max\left\{1-\eta\mu/6,1-p\eta\mu/(2\gamma+\eta\mu)\right\}$ implies
\begin{align}\label{eq:lemma1res}  \begin{split}
    \E [\Phi^{k+1}]\leq \max\left\{1-\frac{\eta\mu}{6},1-\frac{p\eta\mu}{2\gamma+\eta\mu}\right\}\E [\Phi^{k}]-\frac{1}{16\eta}\E \left[\| z ^k-\hat{u}^k\|^2\right]-\frac{\gamma}{2\eta}\E \left[\|w^k-\hat{u}^k\|^2\right].
\end{split} \end{align}

\end{proof}

Then we provide the proof of Lemma \ref{lem:convergenceLya}.
\begin{proof}
    We firstly use the induction to prove 
\begin{align*}
\E[\Phi^k]\leq \Phi^0    
\end{align*}
holds for all $k\in \BN$. 

Note that it holds for $k=0$. Assume we have $\E[\Phi^k]\leq  \Phi^0$ holds, then Lemma \ref{lem:lem1prooflem4} means it holds
\begin{align*}
    \E [\Phi^{k+1}]\leq \max\left\{1-\frac{\eta\mu}{6},1-\frac{p\eta\mu}{2\gamma+\eta\mu}\right\}\E [\Phi^{k}]\leq \E [\Phi^{k}]\leq \Phi^0,
\end{align*}
which finish the induction.

The result of above induction implies the condition of Lemma \ref{lem:lem1prooflem4} always holds. 
Therefore, we can apply Lemma \ref{lem:lem1prooflem4} to achieve equation (\ref{eq:lemma1res}), which finishes the proof of Lemma \ref{lem:convergenceLya}.
\end{proof}

\subsection{Proof of Theorem 
\ref{thm:convergenceSVOGScc}}

We firstly introduce the following quantities for our analysis
\begin{align} \label{eq:errors}\begin{split}
    e_{11}(z,k)&:= \frac{2\eta}{b}\sum_{j\in \fS^k}\langle F(z^k)-F_j(z^k)-F(w^{k-1})+F_j(w^{k-1}) ,\hat{u}^k- z\rangle, \\[0.1cm]
    e_{12}(z,k)&:=\frac{2\eta\alpha}{b}\sum_{j\in \fS^k}\langle F(z^k)-F_j(z^k)-F(z^{k-1})+F_j(z^{k-1}) ,\hat{u}^k- z\rangle, \\[0.1cm]
    e_2(z,k)&:=\|w^{k+1}-z\|^2-p\|z^k-z\|^2-(1-p)\|w^k-z\|^2, \\[0.15cm]
    \Psi^k(z)&:=(1-\gamma)\| z ^{k+1}- z\|^2+\frac{\gamma}{p}\|w^{k}-z\|^2+\frac{1}{16}  \| z ^{k}- z ^{k-1}\|^2.
\end{split} \end{align}
Specifically, we take the constant
\begin{align}\label{value-zeta}
    \zeta:=\min\left\{\frac{\eta^2\varepsilon^2}{16(9\eta L D +3\eta\max_{i\in [n]}\|F_i(z^0)\|  + D )^2},\frac{\eta\varepsilon}{4(12\eta^2\delta^2+1)}\right\}
\end{align}
and 
\begin{align}\label{value-chat}
\hat{c}:=100+2048\eta^2\delta^2+64\sqrt{2\eta\Phi^0}\leq 102+16\sqrt{\Phi^0/\delta}
\end{align}
for the statement Theorem \ref{thm:convergenceSVOGScc}. We then provide several lemmas that will be used in the proof of Theorem \ref{thm:convergenceSVOGScc}.

\begin{lem}\label{lem:thm1prooflem1}
    Under the setting of Theorem \ref{thm:convergenceSVOGScc}, we have
    \begin{align*}
        & -2   \langle \E _k[\delta^k]+F_1(\hat{u}^k),\hat{u}^k- z  \rangle \\
    &\leq 4L D \|u^k-\hat{u}^k\|+2 \langle F(u^k), z -u^k\rangle+(8 L D +6 D_F)\|\hat{u}^k-u^k\|
    +\frac{1}{16\eta}   \| z ^k- z ^{k+1}\|^2   \\
    &\quad +16\eta \delta^2   \|\hat{u}^k-u^k\|^2-2   \langle F( z ^k)-F_1( z ^k)-F( z ^{k+1})+F_1( z ^{k+1}), z ^{k+1}- z  \rangle \\
    &\quad+\frac{1}{2\eta}   \| z ^k-\hat{u}^k\|^2 +2\eta\delta^2\alpha^2   \| z ^k- z ^{k-1}\|^2  \\
    &\quad -2\alpha   \langle F( z ^k)-F_1( z ^k)-F( z ^{k-1})+F_1( z ^{k-1}), z ^{k}- z\rangle,
    \end{align*}
    where $D_F:= \max_{i\in [n]}\sup_{z\in\fZ}\|F_i(z)\|$.
\end{lem}
\begin{proof}

Note that the sequence $\{\|F_i(z)\|\}_{i=1}^n$ is bounded on $z\in\fZ$, since we have
\begin{align}\label{eq:bound-DF}
    D_F= \max_{i\in [n]}\sup_{z\in\fZ}\|F_i(z)\| \leq \max_{i\in [n]}\sup_{z\in\fZ}(\|F_i(z)-F_i(z^0)\|+\|F_i(z^0)\|)\leq LD+\max_{i\in [n]}\|F_i(z^0)\|.
\end{align}

The Lipschitz continuity of $F(\cdot)$ impies
\begin{align*}\begin{split}
    \|F_1(\hat{u}^k)\|-D_F \leq \|F_1(\hat{u}^k)\|-\|F_1(z)\|\leq \|F_1(\hat{u}^k)-F_1(z)\|\leq L D,
\end{split} \end{align*} 
then we have
\begin{align*} 
    & -2   \langle \E _k[\delta^k]+F_1(\hat{u}^k),\hat{u}^k- z  \rangle \\
    &= -2   \langle F( z ^k)-F_1( z ^k)+\alpha\left(F( z ^k) - F_1( z ^k)- F( z ^{k-1}) + F_1( z ^{k-1})\right)+F_1(\hat{u}^k),\hat{u}^k- z  \rangle \\
    &= -2   \langle F(u^k)-F_1(u^k)-F(\hat{u}^k)+F_1(\hat{u}^k),\hat{u}^k- z  \rangle \\
    &\quad -2   \langle F(\hat{u}^k),\hat{u}^k- z  \rangle \\
    &\quad -2   \langle F( z ^k)-F_1( z ^k)-F( z ^{k+1})+F_1( z ^{k+1}),\hat{u}^k- z ^{k+1}\rangle \\
    &\quad -2   \langle F( z ^k)-F_1( z ^k)-F( z ^{k+1})+F_1( z ^{k+1}), z ^{k+1}- z  \rangle \\
    &\quad -2\alpha   \langle F( z ^k)-F_1( z ^k)-F( z ^{k-1})+F_1( z ^{k-1}),\hat{u}^k- z ^{k}\rangle \\
    &\quad -2\alpha   \langle F( z ^k)-F_1( z ^k)-F( z ^{k-1})+F_1( z ^{k-1}), z ^{k}- z  \rangle \\
    &\leq 4L D \|u^k-\hat{u}^k\|+2 \langle F(u^k), z -u^k\rangle+(8 L D +6 D_F)\|\hat{u}^k-u^k\|
    +\frac{1}{16\eta}   \| z ^k- z ^{k+1}\|^2   \\
    &\quad +16\eta \delta^2   \|\hat{u}^k-u^k\|^2-2   \langle F( z ^k)-F_1( z ^k)-F( z ^{k+1})+F_1( z ^{k+1}), z ^{k+1}- z  \rangle \\
    &\quad+\frac{1}{2\eta}   \| z ^k-\hat{u}^k\|^2 +2\eta\delta^2\alpha^2   \| z ^k- z ^{k-1}\|^2  \\
    &\quad -2\alpha   \langle F( z ^k)-F_1( z ^k)-F( z ^{k-1})+F_1( z ^{k-1}), z ^{k}- z\rangle.
\end{align*}
\end{proof}
\begin{lem}\label{lem:thm1prooflem2}
  Under the setting of Theorem \ref{thm:convergenceSVOGScc}, the quantities defined in equation (\ref{eq:errors}) hold
  \begin{align*}
&\max_{z\in\fZ}\Psi^0(z)+\E\left[\max_{z\in\fZ}\sum_{k=0}^{K-1}e_{11}(z,k)+e_{12}(z,k)+\frac{\gamma}{p}e_2(z,k)\right]\\
       & \leq  \max_{z\in\fZ} 4\|z^0-z\|^2 +\frac{4\eta^2\delta^2}{b} \sum_{k=0}^{K-1}\E\left[\|z^k-\hat{u}^{k}\|^2\right]+\left(2p+\frac{4\eta^2\delta^2}{b}  \right)\sum_{k=0}^{K-1}\E\left[\|w^k-\hat{u}^{k}\|^2\right]\\
    &\quad +\left(2p+\frac{8\eta^2\delta^2}{b}\right)\sum_{k=0}^{K-1}\E\left[\|\hat{u}^k-u^k\|^2\right],
  \end{align*}
\end{lem}
\begin{proof}
    
Applying Lemma \ref{lem:mar} with $x^0=z^0$, $\fF_0=\sigma(\fS^0)$, $\fF_k=\sigma(\fS^0,\ldots,\fS^{k-1},w^k)$ for $k\geq 1$, and \begin{align*}
    r^{k+1}=\frac{2\eta}{b}\sum_{j\in \fS^k}  F_j(z^k)-F(z^k)-F_j(w^{k-1})+F(w^{k-1}),
\end{align*} and then using $\E_k[r^{k+1}]=0$, we have
\begin{align*} \begin{split}
    \E\left[\max_{z\in\fZ}\sum_{k=0}^{K-1}e_{11}(z,k)\right]&=\E\left[\max_{z\in\fZ}\sum_{k=0}^{K-1}\langle r^{k+1},z\rangle \right]\leq \max_{z\in\fZ} \frac{1}{2}\|z^0-z\|^2+\frac{1}{2}\sum_{k=0}^{K-1}\E\left[\|r^{k+1}\|^2\right]\\
    &\leq \max_{z\in\fZ} \frac{1}{2}\|z^0-z\|^2+\frac{2\eta^2\delta^2}{b} \sum_{k=0}^{K-1}\E\left[\|z^k-w^{k-1}\|^2\right].
\end{split} \end{align*}

Similarly we can obtain
\begin{align*}  
\E\left[\max_{z\in\fZ}\sum_{k=0}^{K-1}e_{12}(z,k)\right]&\leq \max_{z\in\fZ} \frac{1}{2}\|z^0-z\|^2+\frac{2\eta^2\alpha^2\delta^2}{b} \sum_{k=0}^{K-1}\E\left[\|z^k-z^{k-1}\|^2\right].
\end{align*}

Applying Lemma \ref{lem:mar} with $x^0=z^0$, $\fF_0=\sigma(\fS^0)$, $\fF_k=\sigma(\fS^0,\ldots,\fS^{k-1},w^k)$ for $k\geq 1$, and \begin{align*}
    r^{k+1}=pz^{k+1}+(1-p)w^k-w^{k+1},
\end{align*} 
and then using the fact that $\E[\|w^{k+1}\|^2-p\| z^{k+1}\|^2-(1-p)\|w^{k}\|^2]=0$ and $\E_k[r^{k+1}]=0$, we have
\begin{align*}
    \E\left[\max_{z\in\fZ}\sum_{k=0}^{K-1}e_2(z,k)\right]&=2\E\left[\max_{z\in\fZ}\sum_{k=0}^{K-1}\langle r^{k+1},z\rangle \right]\leq \max_{z\in\fZ} \|z^0-z\|^2+\sum_{k=0}^{K-1}\E\left[\|r^{k+1}\|^2\right]\\
    &\leq \max_{z\in\fZ}  \|z^0-z\|^2+p(1-p) \sum_{k=0}^{K-1}\E\left[\|z^{k+1}-w^k\|^2\right],
\end{align*}
where we use
\begin{align*} \E\left[\|r^{k+1}\|^2\right]&=\E\left[\E_{k+1/2}\|\E_{k+1/2}[w^{k+1}]-w^{k+1}\|^2\right]\\
    &=\E\left[\E_{k+1/2}[\|w^{k+1}\|^2]-\|\E_{k+1/2}[w^{k+1}]\|^2\right]\\
    &=\E\left[p\|z^{k+1}\|^2+(1-p)\|w^k\|^2-\|pz^{k+1}+(1-p)w^k\|^2\right]\\
    &=p(1-p) \E\left[\|z^{k+1}-w^k\|^2\right].
 \end{align*}

Note that $z^k=u^{k-1}$, then we have
\begin{align*} 
    & \max_{z\in\fZ}\Psi^0(z)+\E\left[\max_{z\in\fZ}\sum_{k=0}^{K-1}e_{11}(z,k)+e_{12}(z,k)+\frac{\gamma}{p}e_2(z,k)\right]\\
    & \leq  4\max_{z\in\fZ} \|z^0-z\|^2+\frac{2\eta^2\delta^2}{b} \sum_{k=0}^{K-1}\E\left[\|z^k-w^{k-1}\|^2\right]+\frac{2\eta^2\alpha^2\delta^2}{b} \sum_{k=0}^{K-1}\E\left[\|z^k-z^{k-1}\|^2\right]\\
    &\quad +p(1-p) \sum_{k=0}^{K-1}\E\left[\|z^{k+1}-w^k\|^2\right]\\
    & \leq 4 \max_{z\in\fZ} \|z^0-z\|^2+\frac{4\eta^2\delta^2}{b} \sum_{k=0}^{K-1}\E\left[\|\hat{u}^{k-1}-w^{k-1}\|^2\right]+\frac{4\eta^2\alpha^2\delta^2}{b} \sum_{k=0}^{K-1}\E\left[\|\hat{u}^{k-1}-z^{k-1}\|^2\right]\\
    &\quad +2p \sum_{k=0}^{K-1}\E\left[\|\hat{u}^k-w^k\|^2\right]+\left(\frac{4\eta^2\delta^2}{b}+\frac{4\eta^2\alpha^2\delta^2}{b}+2p\right)\sum_{k=0}^{K-1}\E\left[\|\hat{u}^k-u^k\|^2\right]\\
    & \leq  4\max_{z\in\fZ} \|z^0-z\|^2 +\frac{4\eta^2\delta^2}{b} \sum_{k=0}^{K-1}\E\left[\|z^k-\hat{u}^{k}\|^2\right]+\left(2p+\frac{4\eta^2\delta^2}{b}  \right)\sum_{k=0}^{K-1}\E\left[\|w^k-\hat{u}^{k}\|^2\right]\\
    &\quad +\left(2p+\frac{8\eta^2\delta^2}{b}\right)\sum_{k=0}^{K-1}\E\left[\|\hat{u}^k-u^k\|^2\right],
\end{align*}
where we use Young's inequality and $\alpha\leq 1$.
\end{proof}

Now we provide the proof of Theorem \ref{thm:convergenceSVOGScc}.

\begin{proof}
The optimality of $\hat{u}^k$ implies for all $z\in \fZ$, we have
\begin{align} \label{eq:211-main}\begin{split}
\langle \eta F_1(\hat{u}^k)+\hat{u}^k-v^k, z -\hat{u}^k\rangle \geq 0.
\end{split} \end{align}

Combine equation (\ref{eq:211-main}) with the update rule in Line \ref{line:updatevk} of Algorithm \ref{alg:SVOGS} , we achieve
\begin{align}\label{eq:244-main} \begin{split}
    -\frac{1}{\eta}\langle \bar{ z }^k-\hat{u}^k-\eta\delta^k,\hat{u}^k- z \rangle \leq -\langle F_1(\hat{u}^k),\hat{u}^k- z \rangle.
\end{split} \end{align}

Then we have
 \begin{align*} 
     \frac{1}{\eta}\|\hat{u}^k- z \|^2&=\frac{1}{\eta}\| z ^k- z \|^2+\frac{2}{\eta}\langle\hat{u}^k- z ^k,\hat{u}^k- z \rangle-\frac{1}{\eta}\|\hat{u}^k- z ^k\|^2\\
    &=\frac{1}{\eta}\| z ^k- z \|^2+\frac{2\gamma}{\eta}\langle w^k- z ^k,\hat{u}^k- z \rangle-2\langle\delta^k,\hat{u}^k- z \rangle\\
    &\quad -\frac{1}{\eta}\|\hat{u}^k- z ^k\|^2-\frac{2}{\eta}\langle \bar{ z }^k-\hat{u}^k-\eta\delta^k,\hat{u}^k- z \rangle \\
    &\leq \frac{1}{\eta}\| z ^k- z \|^2+\frac{2\gamma}{\eta}\langle w^k- z ^k,\hat{u}^k- z \rangle-2\langle\delta^k,\hat{u}^k- z \rangle-\frac{1}{\eta}\|\hat{u}^k- z ^k\|^2\\
     &\quad -2\langle F_1(\hat{u}^k),\hat{u}^k- z \rangle\\
     &= \frac{1}{\eta}\| z ^k- z \|^2+\frac{\gamma}{\eta}\|w^k- z \|^2-\frac{\gamma}{\eta}\|w^k-\hat{u}^k\|^2-\frac{\gamma}{\eta}\| z ^k- z \|^2-\frac{1-\gamma}{\eta}\|\hat{u}^k- z ^k\|^2\\
     &\quad -2\langle\delta^k+ F_1(\hat{u}^k),\hat{u}^k- z \rangle\\
     &= \frac{1}{\eta}\| z ^k- z \|^2+\frac{\gamma}{\eta}\|w^k- z \|^2-\frac{\gamma}{\eta}\|w^k-\hat{u}^k\|^2-\frac{\gamma}{\eta}\| z ^k- z \|^2-\frac{1-\gamma}{\eta}\|\hat{u}^k- z ^k\|^2\\
     &\quad -2\langle\E\delta^k+ F_1(\hat{u}^k),\hat{u}^k- z \rangle +\frac{1}{\eta}e_{11}(z,k)+\frac{1}{\eta}e_{12}(z,k),
\end{align*}

Combining this result with Lemma \ref{lem:thm1prooflem1}, we have
 \begin{align}\label{eq:83-main} \begin{split}
     \frac{1}{\eta}\|\hat{u}^k- z \|^2&\leq  \frac{1}{\eta}\| z ^k- z \|^2+\frac{\gamma}{\eta}\|w^k- z \|^2-\frac{\gamma}{\eta}\|w^k-\hat{u}^k\|^2-\frac{\gamma}{\eta}\| z ^k- z \|^2\\
     &\quad -\frac{1-\gamma}{\eta}\|\hat{u}^k- z ^k\|^2-2\langle\E\delta^k+ F_1(\hat{u}^k),\hat{u}^k- z \rangle +\frac{1}{\eta}e_{11}(z,k)+\frac{1}{\eta}e_{12}(z,k)\\
     &\leq  \frac{1}{\eta}\| z ^k- z \|^2+\frac{\gamma}{\eta}\|w^k- z \|^2-\frac{\gamma}{\eta}\|w^k-\hat{u}^k\|^2-\frac{\gamma}{\eta}\| z ^k- z \|^2-\frac{1}{4\eta}\|\hat{u}^k- z ^k\|^2\\
     &\quad +2 \langle F(u^k), z -u^k\rangle+(12 L D +6 D_F )\|\hat{u}^k-u^k\|
    +\frac{1}{16\eta}   \| z ^k- z ^{k+1}\|^2  \\
    &\quad  + 16\eta \delta^2   \|\hat{u}^k-u^k\|^2-2  \langle F( z ^k)-F_1( z ^k)-F( z ^{k+1})+F_1( z ^{k+1}), z ^{k+1}- z  \rangle  \\
    &\quad +2\eta\delta^2\alpha^2  \| z ^k- z ^{k-1}\|^2-2\alpha  \langle F( z ^k)-F_1( z ^k)-F( z ^{k-1})+F_1( z ^{k-1}), z ^{k}- z\rangle \\
    &\quad  +\frac{1}{\eta}e_{11}(z,k)+\frac{1}{\eta}e_{12}(z,k),
\end{split} \end{align}
where we use the fact $\gamma\leq 1/4$.

From the fact that $\|a+b\|^2\geq\frac{1}{2}\|a\|^2-\|b\|^2$ and $\frac{3}{2}\|a+b\|^2\geq\|a\|^2-3\|b\|^2$  , we have
\begin{align}\label{eq:84-hat} \begin{split}
    -\frac{1}{4\eta}  \|\hat{u}^k- z ^k\|^2 &\leq  -\frac{1}{8\eta}  \| z ^{k+1}- z ^k\|^2 +\frac{1}{4\eta}  \|\hat{u}^k-u^k\|^2 \\
   -\frac{\gamma}{\eta}  \|w^k-\hat{u}^k\|^2 &\leq -\frac{\gamma}{2\eta}  \|w^k- z ^{k+1}\|^2 +\frac{\gamma}{\eta}  \|\hat{u}^k-u^k\|^2 .
\end{split} \end{align}

Plugging equation (\ref{eq:84-hat}) into equation (\ref{eq:83-main}), we achieve
 \begin{align*} \begin{split}
     \frac{1}{\eta}\|\hat{u}^k- z \|^2&\leq  \frac{1}{\eta}\| z ^k- z \|^2+\frac{\gamma}{\eta}\|w^k- z \|^2-\frac{\gamma}{\eta}\| z ^k- z \|^2 -\frac{1}{8\eta}  \| z ^{k+1}- z ^k\|^2 \\
     &\quad  +\frac{1}{4\eta}  \|\hat{u}^k-u^k\|^2+2 \langle F(u^k), z -u^k\rangle+(12 L D +6 D_F )\|\hat{u}^k-u^k\|
      \\
    &\quad +\frac{1}{16\eta}   \| z ^k- z ^{k+1}\|^2 -2  \langle F( z ^k)-F_1( z ^k)-F( z ^{k+1})+F_1( z ^{k+1}), z ^{k+1}- z  \rangle  \\
    &\quad + 16\eta \delta^2   \|\hat{u}^k-u^k\|^2-2\alpha  \langle F( z ^k)-F_1( z ^k)-F( z ^{k-1})+F_1( z ^{k-1}), z ^{k}- z\rangle  \\
     &\quad +2\eta\delta^2\alpha^2  \| z ^k- z ^{k-1}\|^2-\frac{\gamma}{2\eta}  \|w^k- z ^{k+1}\|^2 +\frac{\gamma}{\eta}\|\hat{u}^k-u^k\|^2\\
     &\quad +\frac{1}{\eta}e_{11}(z,k)+\frac{1}{\eta}e_{12}(z,k)\\
     &\leq  \frac{1}{\eta}\| z ^k- z \|^2+\frac{\gamma}{\eta}\|w^k- z \|^2-\frac{\gamma}{\eta}\| z ^k- z \|^2 -\frac{1}{16\eta}  \| z ^{k+1}- z ^k\|^2  \\
     &\quad -\frac{\gamma}{2\eta}  \|w^k- z ^{k+1}\|^2 +2 \langle F(u^k), z -u^k\rangle+(12 L D +6 D_F )\|\hat{u}^k-u^k\|\\
    &\quad -2  \langle F( z ^k)-F_1( z ^k)-F( z ^{k+1})+F_1( z ^{k+1}), z ^{k+1}- z  \rangle +2\eta\delta^2\alpha^2  \| z ^k- z ^{k-1}\|^2 \\
    &\quad -2\alpha  \langle F( z ^k)-F_1( z ^k)-F( z ^{k-1})+F_1( z ^{k-1}), z ^{k}- z\rangle  \\
    &\quad + \left(16\eta \delta^2 +\frac{1}{2\eta}\right)  \|\hat{u}^k-u^k\|^2
     +\frac{1}{\eta}e_{11}(z,k)+\frac{1}{\eta}e_{12}(z,k).
\end{split} \end{align*}

Note that the fact
\begin{align*} \begin{split}
     \frac{1}{\eta}\|\hat{u}^k- z \|^2=\frac{1}{\eta}\| z ^{k+1}- z\|^2+\frac{1}{\eta}\|\hat{u}^k-u^k\|^2-\frac{2}{\eta} \|\hat{u}^k-u^k\|\| z ^{k+1}- z \|,
\end{split} \end{align*}
then we have
 \begin{align*} 
    2 \langle F(u^k),  u^k-z \rangle&\leq  -\frac{1}{\eta}\| z ^{k+1}- z\|^2-\frac{1}{\eta}\|\hat{u}^k-u^k\|^2+\frac{2}{\eta} \|\hat{u}^k-u^k\|\| z ^{k+1}- z \|\\
    &\quad +\frac{1}{\eta}\| z ^k- z \|^2+\frac{\gamma}{\eta}\|w^k- z \|^2-\frac{\gamma}{\eta}\| z ^k- z \|^2 -\frac{1}{16\eta}  \| z ^{k+1}- z ^k\|^2  \\
     &\quad - \frac{\gamma}{2\eta}  \|w^k- z ^{k+1}\|^2 +(12 L D +6 D_F )\|\hat{u}^k-u^k\|+2\eta\delta^2\alpha^2  \| z ^k- z ^{k-1}\|^2 \\
    &\quad -2  \langle F( z ^k)-F_1( z ^k)-F( z ^{k+1})+F_1( z ^{k+1}), z ^{k+1}- z  \rangle  \\
    &\quad -2\alpha  \langle F( z ^k)-F_1( z ^k)-F( z ^{k-1})+F_1( z ^{k-1}), z ^{k}- z\rangle 
    \\
     &\quad  + \left(16\eta \delta^2 +\frac{1}{2\eta}\right)  \|\hat{u}^k-u^k\|^2 +\frac{1}{\eta}e_{11}(z,k)+\frac{1}{\eta}e_{12}(z,k)\\
     &\leq  -\frac{1}{\eta}\| z ^{k+1}- z\|^2-\frac{1}{16\eta}  \| z ^{k+1}- z ^k\|^2 -\frac{\gamma}{2\eta}  \|z ^{k+1}-w^k \|^2 \\
     &\quad -2  \langle F( z ^k)-F_1( z ^k)-F( z ^{k+1})+F_1( z ^{k+1}), z ^{k+1}- z  \rangle  \\
    &\quad +\frac{1}{\eta}\| z ^k- z \|^2+\frac{\gamma}{\eta}\|w^k- z \|^2-\frac{\gamma}{\eta}\| z ^k- z \|^2 +2\eta\delta^2\alpha^2  \| z ^k- z ^{k-1}\|^2 \\
     &\quad +2\alpha  \langle F( z ^{k-1})-F_1( z ^{k-1})-F( z ^k)+F_1( z ^k), z ^{k}- z\rangle \\
     &\quad +\left(12 L D +6 D_F +\frac{2 D }{\eta}\right)\|\hat{u}^k-u^k\|
    +  16\eta \delta^2   \|\hat{u}^k-u^k\|^2 \\
    &\quad +\frac{1}{\eta}e_{11}(z,k)+\frac{1}{\eta}e_{12}(z,k)\\
    %%%
    &\leq  -\frac{1-\gamma}{\eta}\| z ^{k+1}- z\|^2-\frac{\gamma}{\eta p}\|w^{k+1}-z\|^2-\frac{1}{16\eta }  \| z ^{k+1}- z ^k\|^2  \\
     &\quad -2  \langle F( z ^k)-F_1( z ^k)-F( z ^{k+1})+F_1( z ^{k+1}), z ^{k+1}- z  \rangle  \\
    &\quad +\frac{1-\gamma}{\eta }\| z ^k- z \|^2+\frac{\gamma}{\eta p}\|w^k- z \|^2+2\eta\delta^2\alpha^2  \| z ^k- z ^{k-1}\|^2 \\
     &\quad +2\alpha  \langle F( z ^{k-1})-F_1( z ^{k-1})-F( z ^k)+F_1( z ^k), z ^{k}- z\rangle \\
     &\quad +\left(12  L\Omega+6  D_F+\frac{2D}{\eta }\right)\|\hat{u}^k-u^k\|
    +  16\eta\delta^2   \|\hat{u}^k-u^k\|^2 \\
    &\quad +\frac{1}{\eta}\left(e_{11}(z,k)+e_{12}(z,k) +\frac{\gamma}{p}e_2(z,k)\right)-\frac{\gamma}{2\eta }  \|z ^{k+1}-w^k \|^2.
 \end{align*}

The parameters settings implies $2\eta\delta^2\alpha^2\leq 1/{(16\eta)}$, then we have
%we can let $u_{{\rm avg}}^K=\frac{1}{K}\sum_{k=0}^{K-1}u^k$, and then we will have
\begin{align} \label{eq:93-main}\begin{split}
  &2\eta K \E\left[\max_{z\in\fZ}\frac{1}{K}\sum_{k=0}^{K-1}\langle F(u^k),u^k-z\rangle\right]\\
  &\leq \max_{z\in\fZ}\Psi^0(z)+\E\left[\max_{z\in\fZ}\sum_{k=0}^{K-1}e_{11}(z,k)+e_{12}(z,k)+\frac{\gamma}{p}e_2(z,k)\right]\\
  &\quad +\sum_{k=0}^{K-1}\left(\left(12\eta L D +6\eta D_F  +2 D \right)\|\hat{u}^k-u^k\|+16\eta^2\delta^2\|u^k-\hat{u}^k\|^2\right).
\end{split} \end{align}

Recall that we take $\hat c$ by equation (\ref{value-chat}), then the setting 
\begin{align*}
\varepsilon_k=\min\left\{\zeta,\hat{c}^{-1}\min\left\{\|\hat{u}^k- z ^k\|,\|\hat{u}^k- z ^k\|^2\right\}\right\}   
\end{align*}
satisfy the condition on $\varepsilon_k$ in Lemma \ref{lem:convergenceLya}.
Then we apply Lemma \ref{lem:convergenceLya} with $\mu=0$ and $\alpha=1$ and sum over equation~(\ref{eq:convergenceLya}) with $k=0,\dots,K-1$ to obtain
\begin{align}\label{eq:99-bound2} \begin{split}
    \sum_{k=0}^{K-1}\left(\frac{1}{32}\E \left[\| z ^k-\hat{u}^k\|^2\right]+\frac{\gamma}{2}\E \left[\|w^k-\hat{u}^k\|^2\right]\right)\leq \left(1+\frac{\gamma}{p}\right)\|z^0-z^*\|^2.
\end{split} \end{align}

Note that parameter settings $\gamma=p=1/{(\sqrt{n}+8)}$, $b=\left\lceil\sqrt{n}\right\rceil$, and $\eta =\min\left\{{\sqrt{\gamma b}}/{(4\delta)}, 1/(32\delta)\right\}$ satisfy
\begin{align} \label{eq:101-req}\begin{split}
    \frac{4\eta^2\delta^2}{b}\leq \frac{\gamma}{4}\leq 8\cdot\frac{1}{32},\quad 
    2p+\frac{4\eta^2\delta^2}{b}\leq 2p+\frac{4\delta^2}{b}\frac{\gamma b}{16\delta^2}\leq 5 \cdot\frac{\gamma}{2}\quad\text{and}\quad
    1+\frac{\gamma}{p}=2.
\end{split} \end{align}

Substituting equations (\ref{eq:99-bound2}) and (\ref{eq:101-req}) into equation (\ref{eq:93-main}) and applying Lemma \ref{lem:thm1prooflem2}, we obtain
\begin{align*} \begin{split}
  &\E\left[\max_{z\in\fZ}\frac{1}{K}\sum_{k=0}^{K-1}\langle F(u^k),u^k-z\rangle\right]\\
  &\leq \frac{1}{ 2\eta K}\left(4+8\cdot 2\right)\max_{z\in\fZ}\|z^0-z\|^2 \\
  &\quad+\frac{1}{ 2\eta K}\sum_{k=0}^{K-1}\left(\left(12\eta L D +6\eta D_F  +2 D \right)\|\hat{u}^k-u^k\|+\left(16\eta^2\delta^2+\frac{8\eta^2\delta^2}{b}+2p\right)\|u^k-\hat{u}^k\|^2\right)\\
  &\leq \frac{10 D^2 }{\eta K} +\frac{6\eta L D +3\eta D_F  + D }{\eta}\sqrt{\zeta}+\frac{12\eta^2\delta^2+1}{\eta}\zeta\\
  &\leq \frac{10 D^2 }{\eta K} +\frac{9\eta L D +3\eta\max_{i\in [n]}\|F_i(z^0)\|  + D }{\eta}\sqrt{\zeta}+\frac{12\eta^2\delta^2+1}{\eta}\zeta,
\end{split} \end{align*}
where we use the equation (\ref{eq:bound-DF}) to bound $D_F$.
% \begin{align*}
%     D_F= \max_{i\in [n]}\sup_{z\in\fZ}\|F_i(z)\| \leq \max_{i\in [n]}\sup_{z\in\fZ}(\|F_i(z)-F_i(z^0)\|+\|F_i(z^0)\|)\leq LD+\max_{i\in [n]}\|F_i(z^0)\|.
% \end{align*}

Recall that we take constant $\zeta$ by equation (\ref{value-zeta}), then we get the bound
\begin{align*} 
    \E\left[\max_{z\in\fZ}\frac{1}{K}\sum_{k=0}^{K-1}\langle F(u^k),u^k-z\rangle\right]\leq \frac{10 D^2 }{\eta K} +\frac{\varepsilon}{2}.
\end{align*}

\end{proof}

\subsection{Proof of Corollary \ref{cor:complexitySVOGScc}}\label{sec:proofcomplexitySVOGScc}
\begin{proof}

Theorem \ref{thm:convergenceSVOGScc} means we can achieve $\E[\operatorname{Gap}(u_{{\rm avg}}^K)]\leq\varepsilon$ by taking the communication rounds of
\begin{align*} \begin{split}
    K =\left\lceil\frac{20D^2}{\varepsilon\eta}\right\rceil=\fO\left(\frac{ D^2  }{\varepsilon\eta}\right)=\fO\left(\frac{\delta D^2 }{\varepsilon}\right).
\end{split} \end{align*}

Consider that the expected communication complexity in each round is $\fO(b(1-p)+np)=\fO(\sqrt{n})$ and the server need to communicate with all client in initialization within the  communication complexity of $\fO(n)$, the overall communication complexity is
\begin{align*} \begin{split}    \fO(n)+K\cdot\fO(\sqrt{n})=\fO\left(n+\frac{\sqrt{n}\delta D^2 }{\varepsilon} \right) .
\end{split} 
\end{align*}

Note that the objective of the sub-problem in Line \ref{line:sub-prob} of Algorithm \ref{alg:SVOGS} is $(L+1/\eta)$-smooth-$(1/\eta)$-strongly-convex-$(1/\eta)$-strongly-concave, hence the local gradient complexity for solving the sub-problem is \mbox{$\fO((1+\eta L)\log(\max\{\zeta^{-1},\hat{c}\})$}.
Therefore, the overall local gradient complexity is
\begin{align*} \begin{split}
        & \fO(n)+K\cdot\left(\fO(\sqrt{n})+\fO\left((1+\eta L)\log(\zeta^{-1}+\hat{c})\right)\right)\\
        &=\fO(n)+\fO\left(\frac{\delta D^2 }{\varepsilon}\right)\cdot\left(\fO(\sqrt{n})+\fO\left(\left(1+\frac{L}{\delta}\right)\log\left(\frac{LD+D_F}{\varepsilon}+\sqrt{\frac{\Phi^0}{\delta}}\,\right)\right)\right)\\
        &=\tilde{\fO}\left(n+\frac{(\sqrt{n}\delta+L) D^2 }{\varepsilon}\log\frac{1}{\varepsilon}\right).
\end{split} \end{align*}
\end{proof}

\subsection{Proof of Theorem \ref{thm:convergenceSVOGSscsc}}\label{appx:proofconscsc}
\begin{proof}
We can verify that the parameter setting of Theorem \ref{thm:convergenceSVOGSscsc} satisfies the condition of  Lemma \ref{lem:convergenceLya}. 
Then we can apply Lemma \ref{lem:convergenceLya} to obtain
\begin{align}
    \E [\Phi^{K}]\leq \max\left\{1-\frac{\eta\mu}{6},1-\frac{p\eta\mu}{2\gamma+\eta\mu}\right\}^K\Phi^{0}.
\end{align}
\end{proof}

\subsection{Proof of Corollary \ref{thm:complexitySVOGSscsc}}\label{appx:proofscsc}
\begin{proof}
Recall that we set the parameters as
    \begin{align*} 
        \gamma=p=\frac{1}{\min\left\{\sqrt{n},\frac{\delta}{\mu}\right\}+8},&\quad b=\left\lceil\min\left\{\sqrt{n},\frac{\delta}{\mu}\right\}\right\rceil\\
        \eta =\min\left\{\frac{\sqrt{\gamma b}}{4\delta},\frac{1}{32\delta}\right\},&\quad \alpha=\max\left\{1-\frac{\eta\mu}{6},1-\frac{p\eta\mu}{2\gamma+\eta\mu}\right\}.
 \end{align*}    
    
We can lower bound $\alpha$ as
\begin{align} \begin{split}
    \alpha\geq 1-\frac{p\eta\mu}{2\gamma+\eta\mu}=1-\frac{p\eta\mu}{2p+\eta\mu}=1-\frac{1}{\frac{2}{\eta\mu}+\frac{1}{p}}\geq \frac{7}{8}.
\end{split} \end{align}

Then the number of communication rounds is
\begin{align*} \begin{split}
    K&=\fO\left(\left(1+\frac{1}{\eta\mu}+\frac{\gamma+\eta\mu}{p\eta\mu}\right)\log\frac{1}{\varepsilon}\right)\\
    &=\fO\left(\left(\frac{1}{p}+\frac{1}{\eta\mu}\right)\log\frac{1}{\varepsilon}\right)\\
    &=\fO\left(\left(\frac{1}{p}+\frac{1}{\mu}\left(32\delta+\frac{32\delta}{\sqrt{\alpha\gamma b}}\right)\right)\log\frac{1}{\varepsilon}\right)\\
    &=\fO\left(\left(\frac{1}{p}+\frac{\delta}{\mu}+\frac{1}{\sqrt{pb}}\frac{\delta}{\mu}\right)\log\frac{1}{\varepsilon}\right).
\end{split} \end{align*}

Note that 
\begin{align*} \begin{split}
    \frac{1}{p}+\frac{1}{\sqrt{pb}}\frac{\delta}{\mu}&\leq \min\left\{\sqrt{n},\frac{\delta}{\mu}\right\}+8+\frac{\delta}{\mu}\sqrt{\frac{\min\left\{\sqrt{n},\frac{\delta}{\mu}\right\}+8}{\min\left\{\sqrt{n},\frac{\delta}{\mu}\right\} }}\\
    &=\min\left\{\sqrt{n},\frac{\delta}{\mu}\right\}+8+\frac{\delta}{\mu}\sqrt{1+8\max\left\{\frac{1}{\sqrt{n}},\frac{\mu}{\delta}\right\}}=\fO\left(\frac{\delta}{\mu}\right),
\end{split} \end{align*}
then we have $K=\fO(\delta/\mu\log(1/\varepsilon))$. 

Consider that the expected communication complexity in each round is 
\begin{align*}
\fO(b(1-p)+np)=\fO\left(\sqrt{n}+\frac{n\mu}{\delta}\right),     
\end{align*}
and the server need to communicate with all client in initialization within the  communication complexity of $\fO(n)$, the overall communication complexity is
\begin{align*} \begin{split}    \fO(n)+K\cdot\fO\left(\sqrt{n}+\frac{n\mu}{\delta}\right)=\fO\left(\left(n+\frac{\sqrt{n}\delta}{\mu}\right)\log\frac{1}{\varepsilon}\right).
\end{split} 
\end{align*}

Note that the objective of the sub-problem in Line \ref{line:sub-prob} of Algorithm \ref{alg:SVOGS} is $(L+1/\eta)$-smooth-$(1/\eta)$-strongly-convex-$(1/\eta)$-strongly-concave, the local gradient complexity for solving the sub-problem is $\fO((1+\eta L)\log(c))$.
Therefore, the overall local gradient complexity is
\begin{align*}\begin{split}
    &\fO(n)+K\cdot\left(\fO\left(\sqrt{n}+\frac{n\mu}{\delta}\right)+\fO\left((1+\eta L)\log(c)\right)\right)\\
    &=\fO(n)+\fO\left(\frac{\delta}{\mu}\log\frac{1}{\varepsilon}\right)\cdot\left(\fO\left(\sqrt{n}+\frac{n\mu}{\delta}\right)+\fO\left(\left(1+\frac{L}{\delta}\right)\log\frac{\delta}{\mu}\right)\right)\\
    &=\fO\left(\left(n+\frac{\sqrt{n}\delta}{\mu}+\frac{L}{\mu}\log\frac{\delta}{\mu}\right)\log\frac{1}{\varepsilon}\right)\\
    &=\tilde{\fO}\left(\left(n+\frac{\sqrt{n}\delta+L}{\mu}\right)\log\frac{1}{\varepsilon}\right).
\end{split} \end{align*}
\end{proof}

\section{The Algorithm Class}\label{appx:algorithm-class}

We formally define the distributed first-order oracle (DFO) algorithm as follows.

\begin{dfn}[DFO Algorithm]\label{dfn:oracles}
    Each node $i$ has its own local memories $\mathcal{M}_i^x$ and $\mathcal{M}_i^y$ for the $x$- and $y$-variables with initialization  $\mathcal{M}_i^x=\mathcal{M}_i^y=\{0\}$ for all $i\in[n]$. 
    Specifically, the server has memories $\mathcal{M}_1^x$ and $\mathcal{M}_1^y$. 
    These memories $\{\mathcal{M}_i^x\}_{i=1}^n$ and $\{\mathcal{M}_i^y\}_{i=1}^n$ can be updated as follows:
\begin{itemize}[leftmargin=0.5cm,topsep=-0.03cm,itemsep=-0.1cm]
    \item \textbf{Communication from clients to server}: During one communication round, we sample  uniformly and independently batch $\fS$ of any size $b$ and ask client with number from $\fS$ to share some vector of their local memories with the server, i.e. can add points $x_1', y_1'$ to the local memories of the server according to the next rule:
\begin{align*}
  x_1' \in \operatorname{span}\left\{x_1, \bigcup_{i \in \fS} x_i\right\} \qquad\text{and}\qquad   y_1' \in \operatorname{span}\left\{y_1, \bigcup_{i \in \fS} y_i\right\}
\end{align*}
where $x_i \in \mathcal{M}_i^x$ and $y_i\in \mathcal{M}_i^y$. If the batch size is equal to $b$ we say that it costs $b$ communication complexity from clients to the server. Batch of the size $n$ is equal to the situation, when all clients send their memories to the server.

\item \textbf{Communication from server to clients}: 
During one communication round, we sample  uniformly and independently batch $\fS$ of any size $b$ and ask the server to share some vector of its local memories with the clients with numbers from $\fS$, i.e. can add points $x_i', y_i'$ to the corresponding local memories of client $i$ as
\begin{align*}
    x_i' \in \operatorname{span}\left\{x_1, x_i\right\} \qquad\text{and}\qquad y_i' \in \operatorname{span}\left\{y_1, y_i\right\},
\end{align*}

where $x_i \in \mathcal{M}_i^x$ and $y_i \in \mathcal{M}_i^y$, and we say that it costs $b$ communication complexity. 

\item \textbf{Local computations}: During local computations each client $i$ can make any computations using $f_i$, i.e. can add points $x_i', y_i'$ to the corresponding local memory of client $i$ as
\begin{align*}
  x_i' \in \operatorname{span}\left\{x^{\prime}, \nabla_x f_i\left(x^{\prime \prime}, y^{\prime \prime}\right)\right\} \qquad\text{and}\qquad y_i' \in \operatorname{span}\left\{y^{\prime}, \nabla_y f_i\left(x^{\prime \prime}, y^{\prime \prime}\right)\right\},  
\end{align*}

for given $x^{\prime}, x^{\prime \prime} \in \mathcal{M}_i^x$ and $y^{\prime}, y^{\prime \prime} \in \mathcal{M}_i^y$. And we use local gradient calls to count the times when $\nabla_x$ and $\nabla_y$ are applied to any one of $\{f_i\}$.
\end{itemize}
The final global output is calculated as $\hat{x} \in \mathcal{M}_1^x, \hat{y} \in \mathcal{M}_1^y$.
\end{dfn}

Our Definition \ref{dfn:oracles} follows the algorithm class of \citet[Definition C.7]{beznosikov2024similarity}, but additionally take the communication from the server to the clients into considerations.

{\small\section{Lower Bounds in Convex-Concave Case}\label{appx:lowercc}}
In this section, we provide the proofs of the lower bounds for solving the problem
\begin{align}\label{prob:cc}
    \min_{x\in\fX}\max_{y\in\fY} f(x,y)=\frac{1}{n}\sum_{i=1}^n f_i(x,y)
\end{align}
by DFO algorithms, where the diameters of closed convex sets $\fX$ and $\fY$ are $R_x$ and $R_y$ respectively.
We define the subspaces $\{\fF_k\}_{k=0}^d$ as 
\begin{align*}
    \fF_k=\begin{cases}
        {\rm span}\{e_1,\ldots,e_k\},&\quad \text{for}~1\leq k\leq d,\\
        \{0_d\},&\quad \text{for}~k=0,
    \end{cases}
\end{align*}
which is used in the following proofs of lower bounds.

\subsection{Proof of Theorem \ref{thm:lowerboundKcc}}
We first define the function set with one server ($i=1$) and $n-1$ clients ($i=2,\dots,n-1$) as follows
\begin{align}\label{def:ficc}
    f_i(x,y)=\begin{dcases}
        \frac{\delta}{4}x^\top A_1y-\frac{\delta R_y}{2\sqrt{d}}e_1^\top x,&\quad i-1\equiv 1\ ({\rm mod}\ 3),\\
        \frac{\delta}{4}x^\top A_2y,&\quad i-1\equiv 2\ ({\rm mod}\ 3),\\
        0,&\quad \text{otherwise}.
    \end{dcases}
\end{align}
Then corresponding global objective is
\begin{align}\label{def:fcc}
    f(x,y)=\frac{\delta}{6}x^\top Ay-\frac{\delta R_y}{6\sqrt{d}}e_1^\top x,
\end{align}
where
\begin{align}\label{eq:matrixA}
\begingroup
\setlength\arraycolsep{4pt}
    A_1 = \begin{pmatrix}
1 & 0 &  &  &  \\
 & 1 & -2 &  &  \\
 &  & 	\ddots & 	\ddots &  \\
 &  &  & 1 & 0 \\
 &  &  &  & 1
\end{pmatrix},~ A_2 = \begin{pmatrix}
1 & -2 &  &  &  \\
 & 1 & 0 &  &  \\
 &  & 	\ddots & 	\ddots &  \\
 &  &  & 1 & -2 \\
 &  &  &  & 1
\end{pmatrix},~ A = \begin{pmatrix}
1 & -1 &  &  &  \\
 & 1 & -1 &  &  \\
 &  & 	\ddots & 	\ddots &  \\
 &  &  & 1 & -1 \\
 &  &  &  & 1
\end{pmatrix}.
\endgroup
\end{align}

\begin{prop}\label{prop:lowerccrounds}
    For any $d\geq 3$, the functions $f_i(x,y)$ and $f(x,y)$ defined by equations (\ref{def:ficc}) and (\ref{def:fcc}) satisfy
    \begin{enumerate}[leftmargin=2em]
        \item The function $f_i$ is $L$-smooth with $L\geq \delta$ and convex-concave for all $i\in[n]$, and the function set~$\{f_i\}_{i=1}^n$ holds $\delta$-second-order similarity. Thus, the function $f$ is also convex-concave.
        \item For $1\leq k\leq d-1$, we have
        \begin{flalign}\label{eq:mingap}
            \min_{(x,y)\in\fZ\cap\fF_k^2}{\rm Gap}(x,y)=
            \min_{x\in\fX\cap\fF_k}\max_{y\in\fY}f(x,y)-\max_{y\in\fY\cap\fF_k}\min_{x\in\fX}f(x,y)\geq \frac{\delta R_xR_y}{6\sqrt{d(k+1)}}.
        \end{flalign}
    \end{enumerate}
\end{prop}
\begin{proof}
    The smoothness and the convexity (concavity) are easy to verify. The similarity holds because
\begin{align*}
    \begin{split}
    &\nabla_{xx}^2f_1(x,y)-\nabla_{xx}^2f(x,y)=\nabla_{xx}^2f_2(x,y)-\nabla_{xx}^2f(x,y)=\nabla_{xx}^2f_3(x,y)-\nabla_{xx}^2f(x,y)=0;\\
   & \nabla_{yy}^2f_1(x,y)-\nabla_{yy}^2f(x,y)=\nabla_{yy}^2f_2(x,y)-\nabla_{yy}^2f(x,y)=\nabla_{yy}^2f_3(x,y)-\nabla_{yy}^2f(x,y)=0\\
&\|\nabla_{xy}^2f_1(x,y)-\nabla_{xy}^2f(x,y)\|\leq\|\nabla_{xy}^2f_1(x,y)\|+\|\nabla_{xy}^2f(x,y)\|\leq\frac{\delta}{3}\leq\delta; \\
&\|\nabla_{xy}^2f_2(x,y)-\nabla_{xy}^2f(x,y)\|\leq\|\nabla_{xy}^2f_2(x,y)\|+\|\nabla_{xy}^2f(x,y)\|\leq\delta\left(\frac{5}{8}+\frac{1}{3}\right)\leq\delta; \\
&\|\nabla_{xy}^2f_3(x,y)-\nabla_{xy}^2f(x,y)\|\leq\|\nabla_{xy}^2f_3(x,y)\|+\|\nabla_{xy}^2f(x,y)\|\leq\delta\left(\frac{5}{8}+\frac{1}{3}\right)\leq\delta.
    \end{split}
\end{align*}

The function $f(x,y)$ defined by our equation (\ref{def:fcc}) is identical to the function $f_{{\rm CC}}(x,y)$ defined by \citet[Proposition 3.31]{han2024lower}\footnote{We follow the notation of \citet{han2024lower}'s arXiv version: \url{https://arxiv.org/pdf/2103.08280v1}} by replacing their notation $L$ with our $\delta$ and taking $n=3$.
Then Proposition 3.31 of \citet{han2024lower} directly prove the result of equation (\ref{eq:mingap}).
\end{proof}

The structure of $A_1$ and $A_2$ results the following lemma.

\begin{lem}\label{lem:zero-chain}
    For the function set (\ref{def:ficc}), all $(x,y)\in \fF_{k}\times\fF_{k}$ and $k=0,\dots,d-1$, we have 
    \begin{align*} 
        \nabla f_i(x,y)\in \begin{cases}
            \fF_{k+1}\times\fF_{k+1},&\quad (i,k)\in \fI_1\cup\fI_2,\\
            \fF_{k}\times\fF_{k},&\quad \text{otherwise,}
        \end{cases}
    \end{align*}
    where $\fI_1:=\{(i,k):i-1\equiv 1\ ({\rm mod}\ 3), k\equiv 0\ ({\rm mod}\ 2)\}$ and $\fI_2:=\{(i,k):i-1\equiv 2\ ({\rm mod}\ 3), k\equiv 1\ ({\rm mod}\ 2)\}$.
\end{lem}

Now we provide the proof of Theorem \ref{thm:lowerboundKcc}.
\begin{proof}
    Consider the minimax problem (\ref{prob:cc}) with functions (\ref{def:ficc}) and (\ref{def:fcc}), $R_x=R_y=D$, $n\geq 3$ and $d=\lfloor \delta D^2/(3\sqrt{2}\varepsilon)\rfloor-1$. 
    Then the assumption $\varepsilon\leq \delta D^2/(12\sqrt{2})$ implies $d\geq 3$. 
    Lemma \ref{lem:zero-chain} means that we need at least one communication round to increase the number of non-zero coordinate, i.e. $(x,y)\in\fZ\cap\fF_K^2$. 
    Running any DFO algorithm with communication rounds of $K=\lfloor (d-1)/2\rfloor\geq 1$, we have $d/2\leq (K+1)\leq (d+1)/2$ and Proposition \ref{prop:lowerccrounds} implies
    \begin{align*}
        \E[{\rm Gap}(x,y)]\geq \min_{(x,y)\in\fZ\cap\fF_K^2}{\rm Gap}(x,y)\geq\frac{\delta D^2}{6\sqrt{d(K+1)}}\geq\frac{\delta D^2}{6\sqrt{2}(K+1)}\geq \frac{\delta D^2}{3\sqrt{2}(d+1)}\geq \varepsilon.
    \end{align*}
    Hence, we achieve the lower bound on the communication rounds of
    \begin{align*}
        K=\left\lfloor\frac{\delta D^2}{6\sqrt{2}\varepsilon}\right\rfloor -1 = \Omega\left(\frac{\delta D^2}{\varepsilon}\right).
    \end{align*}
\end{proof}

\subsection{Proof of Theorem \ref{thm:lowerboundVG1cc}}
We first define the function set with one server ($i=1$) and $n-1$ clients ($i=2,\dots,n-1$) as follows
\begin{align}\label{def:ficc2}
    f_i(x,y)=\begin{dcases}
       - \frac{\sqrt{n}\delta R_y}{8\sqrt{d}}e_1^\top x,&\quad i=1, \\
        \frac{\sqrt{n}\delta}{8}x^\top \Bigg[\sum_{j\equiv(i-1){\rm mod}(n-1)}e_ja_j^\top\Bigg]y,&\quad i\geq 2.
    \end{dcases}
\end{align}

Then corresponding global objective is
\begin{align}\label{def:fcc2}
    f(x,y)=\frac{\delta}{8\sqrt{n}}x^\top Ay-\frac{\delta R_y}{8\sqrt{nd}}e_1^\top x,
\end{align}
where $e_j$ is the $j$-th basis column vector, $a_j^\top$ is the $j$-th row of $A$, $A_i=\sum_{j \equiv(i-1) \bmod (n-1)} e_j a_j^\top$, $A_1=0$ and $A$ is defined in equation (\ref{eq:matrixA}).

We provide the following proposition and lemmas for the proof of Theorem \ref{thm:lowerboundVG1cc}.

\begin{prop}\label{prop:lowerccrounds2}
    For any $d\geq 3$, $f_i(x,y)$ and $f(x,y)$ defined as equations  (\ref{def:ficc2}) and (\ref{def:fcc2}) satisfy
    \begin{enumerate}[leftmargin=2em]
        \item $f_i$ is $L$-smooth with $L\geq \sqrt{n}\delta/4$ and convex-concave, function set $\{f_i\}_{i=1}^n$ has $\delta$ second-order similarity. Thus, $f$ is convex-concave.
        \item For $1\leq k\leq d-1$, we have
        \begin{flalign}\label{eq:mingap2}
            \min_{(x,y)\in\fZ\cap\fF_k^2}{\rm Gap}(x,y)=
            \min_{x\in\fX\cap\fF_k}\max_{y\in\fY}f(x,y)-\max_{y\in\fY\cap\fF_k}\min_{x\in\fX}f(x,y)\geq \frac{\delta R_xR_y}{8\sqrt{nd(k+1)}}.
        \end{flalign}
    \end{enumerate}
\end{prop}
\begin{proof}
    The smoothness and convexity (concavity) are easy to check. And the similarity can be verified following the methods of \citet[Lemma C.8]{beznosikov2024similarity}. 
    The function $f(x,y)$ defined by our equation (\ref{def:fcc2}) is identical to the function $f_{{\rm CC}}(x,y)$ defined by \citet[Proposition 3.31]{han2024lower} by replacing their notation $L$ with our $\sqrt{n}\delta/4$.
Then Proposition 3.31 of \citet{han2024lower} directly prove the result of equation (\ref{eq:mingap2}).
\end{proof}

\begin{lem}\label{lem:coordinatebound}
    Consider the minimax problem (\ref{prob:cc}) with  functions (\ref{def:ficc2}) and (\ref{def:fcc2}), $R_x=R_y=D$, $d=\lfloor \delta D^2/(4\sqrt{2n}\varepsilon)\rfloor-1$ and $\varepsilon\leq \delta D^2/(16\sqrt{2n})$. 
    We let $M=\lfloor (d-1)/2\rfloor \geq \Omega(\delta D^2/(\sqrt{n}\varepsilon))$, then
    any point $(x,y) \in \fZ\cap\fF_M^2$ satisfies ${\rm Gap}(x,y)\geq \varepsilon$.
\end{lem}

\begin{proof}
    The assumptions on $\varepsilon$ and $M$ imply $d\geq 3$.  and $d/2\leq (M+1)\leq (d+1)/2$.
    Then Proposition \ref{prop:lowerccrounds2} means
    \begin{align*}
        {\rm Gap}(x,y) \geq \min_{(x,y)\in\fZ\cap\fF_M^2}{\rm Gap}(x,y)\geq\frac{\delta D^2}{8\sqrt{nd(M+1)}}\geq\frac{\delta D^2}{8\sqrt{2n}(M+1)}\geq \frac{\delta D^2}{4\sqrt{2n}(d+1)}\geq \varepsilon.
    \end{align*}
\end{proof}
 \begin{lem}\label{lem:coordinate}
    Consider the minimax problem (\ref{prob:cc}) with functions (\ref{def:ficc2}) and (\ref{def:fcc2}) and run any DFO algorithm with $V$ communication complexity and $C$ local gradient calls. 
    In expectation, only the first $M\leq \min\{2V/n,2C/n\}$ coordinates of the final output can be non-zero while the rest of the $d-M$ coordinates are strictly equal to zero.
\end{lem}
\begin{proof}

At initialization, $\mathcal{M}_i^x=\mathcal{M}_i^y=\fF_0$. Let's analyze how $\mathcal{M}_i^x$ and $\mathcal{M}_i^y$ change through local computations. For the $i$-th client, we add the following points to $\mathcal{M}_i^x$ and $\mathcal{M}_i^y$ as 
\begin{align*}
    x \in \operatorname{span}\left\{x^{\prime}, A_i y^{\prime}\right\}, \quad\text{and}\quad  y \in \operatorname{span}\left\{e_1 \cdot \mathbb{I}\{i=1\}, y^{\prime}, A_i^\top  x^{\prime}\right\},
\end{align*}
where $x^{\prime} \in \mathcal{M}_i^x$ and $y^{\prime} \in \mathcal{M}_i^y$.

It is easy to see that the server can make the first coordinate of $y$ non-zero using $e_1$, and broadcast this progress to other clients. Only updates of the type $A_i y^{\prime}$ or $A_i^\top  x^{\prime}$ will help in this regard. Since $A_i$ only contains rows from the matrix $A$ such as the $(i-1)$-th row, $(n+i-1)$-th row, etc., to make the first coordinate of $x$ in the global output non-zero, we need and can only use the $A_2$ matrix. It can be noted that by using $A_2$, we can also make the second coordinate of $y$ non-zero after making the first coordinate of $x$ non-zero. Furthermore, to make more progress, we need to use $A_3$ and so on.
We conclude that we must constantly transfer progress from the node currently needed (to make the next coordinate of $x$ non-zero; then of $y$) to the server, and then to other nodes.

By definition, one communication round involves communication with all clients or only with batches of some uniform and independent clients. When we sample without replacement, the success probability of a communication round on clients with batch size $b$ (i.e., making one coordinate non-zero) is $\frac{b}{n-1}$ (or 1 when $b=n$), which is also the expected number of non-zero coordinates that can be obtained with $b$ communication complexity and at least $b$ local gradient calls (as each use of the matrix $A$ necessarily comes from a gradient call). 
 When we sample with replacement by a batch size $b$, it is equivalent to that we sample without replacement by batch size $1$ for $b$ times. 
Assuming that we have communication rounds with batch sizes (sampling without replacement) of~$1, 2, \ldots,n$ for $s_1, s_2, \ldots, s_n$ times, then the communication complexity we spent is $V=\sum_{j=1}^{n} js_j$ 
and the minimum gradient calls we spent is $C=\sum_{j=1}^{n} js_j$. 
This implies the expected number of non-zero coordinates is $M=\sum_{j=1}^{n-1} \frac{j}{n-1} s_j+s_n$. 
Therefore, the expected total number of non-zero coordinates in the global output is at most $M$ for $y$ and $M-1$ for $x$ (or we can say $M$). By comparing expressions of $V$, $C$ and $M$, we can have $M\leq 2V/n$ and $M\leq 2C/n$, completing the proof.
\end{proof}

Then we can prove Theorem \ref{thm:lowerboundVG1cc} by combining Lemma \ref{lem:coordinatebound} and \ref{lem:coordinate}. 

\subsection{Proof of Lemma \ref{lem:lowerboundG2cc}}
We choose the function set as
\begin{align}\label{def:fcc3}
    f(x,y)=f_i(x,y)=\frac{L}{2}x^\top Ay-\frac{LR_y}{2\sqrt{d}}e_1^\top x,\quad \forall i \in [n],
\end{align}
where $A$ is defined as equation (\ref{eq:matrixA}). The structure of $A$ results the following lemma.
\begin{lem}\label{lem:zero-chain2}
    For the function set (\ref{def:fcc3}), all $(x,y)\in \fF_{k}\times\fF_{k}$ and $k=0,\dots,d-1$, we have 
    \begin{align*} 
        \nabla f_i(x,y)\in \fF_{k+1}\times\fF_{k+1}.
    \end{align*}
\end{lem}

One can follow the similar method as the proof of Theorem \ref{thm:lowerboundVG1cc} to prove the following proposition and lemma for the proof of Lemma \ref{lem:lowerboundG2cc}.
\begin{prop}\label{prop:lowerccrounds3}
    For any $d\geq 3$, $f_i(x,y)$ and $f(x,y)$ defined as equation (\ref{def:fcc3}) satisfy
    \begin{enumerate}[leftmargin=2em]
        \item $f_i$ is $L$-smooth and convex-concave, function set $\{f_i\}_{i=1}^n$ has $\delta$ second-order similarity for any $\delta>0$. Thus, $f$ is convex-concave.
        \item For $1\leq k\leq d-1$, we have
        \begin{align*}
            \min_{(x,y)\in\fZ\cap\fF_k^2}{\rm Gap}(x,y)=
            \min_{x\in\fX\cap\fF_k}\max_{y\in\fY}f(x,y)-\max_{y\in\fY\cap\fF_k}\min_{x\in\fX}f(x,y)\geq \frac{L R_xR_y}{2\sqrt{d(k+1)}}.
        \end{align*}
    \end{enumerate}
\end{prop}

\begin{lem}\label{lem:coordinatebound2}
    Consider the minimax problem (\ref{prob:cc}) with the function  (\ref{def:fcc3}), \mbox{$R_x=R_y=D$}, \mbox{$d=\lfloor \delta D^2/(\sqrt{2}\varepsilon)\rfloor-1$} and $\varepsilon\leq \delta D^2/(4\sqrt{2})$. 
    We let $M=\lfloor (d-1)/2\rfloor \geq \Omega(L D^2/\varepsilon)$, then
    any point $(x,y) \in \fZ\cap\fF_M^2$ satisfies ${\rm Gap}(x,y)\geq \varepsilon$.
\end{lem}

Then we can prove Lemma \ref{lem:lowerboundG2cc} by combining Lemma \ref{lem:zero-chain2} and \ref{lem:coordinatebound2}.

\subsection{Proof of Theorem \ref{thm:lowerboundGcc}}\label{appendix:lower-bound-SCSC}
\begin{proof}
In the case of $\sqrt{n}\delta \geq \Omega(L)$, the problem with with functions (\ref{def:ficc2}) and (\ref{def:fcc2}) in Theorem \ref{thm:lowerboundVG1cc} implies the lower bound on the local gradient complexity of \begin{align*}
    \Omega\left(n+\frac{\sqrt{n}\delta D^2}{\varepsilon}\right)=\Omega\left(n+\frac{(\sqrt{n}\delta+L)D^2}{\varepsilon}\right).
\end{align*}
In the case of $L \geq \Omega(\sqrt{n}\delta) $, the problem with function (\ref{def:fcc3}) in Lemma \ref{lem:lowerboundG2cc} implies the lower bound on the local gradient complexity of $\Omega(LD^2/\varepsilon)=\Omega(n+(\sqrt{n}\delta+L)D^2/\varepsilon)$. 
Combining these two cases, we achieve the lower bound on the local gradient complexity of $\Omega(n+(\sqrt{n}\delta+L)D^2/\varepsilon)$.
\end{proof}

{\small\section{Lower Bounds in Strongly-Convex-Strongly-Concave Case}\label{appx:lowerscsc}}

We follow similar steps as in Appendix \ref{appx:lowercc} to prove Theorem \ref{thm:lowerboundGG2scsc}. Firstly we divide it into several detailed theorems and lemmas as below and prove them one by one.

\begin{thm}\label{thm:lowerboundVG1scsc}
    For any $\mu, \delta, L > 0$ with $L\geq \max\{\mu,\delta\}$ and $n\geq 2$, there exist $L$-smooth and convex-concave functions $f_1,\dots,f_n:\BR^{d_x}\times\BR^{d_y}$ with $\delta$-second-order similarity such that the function $f(x,y)=\frac{1}{n}\sum_{i=1}^n f_i(x,y)$ is $\mu$-strongly-convex-$\mu$-strongly-concave. In order to find a solution of problem (\ref{prob:main}) such that $\E[\|z-z^*\|^2]\leq\varepsilon$, any DFO algorithm needs at least $\Omega((n+\sqrt{n}\delta/\mu)\log(1/\varepsilon))$ communication complexity and $\Omega((n+\sqrt{n}\delta/\mu)\log(1/\varepsilon))$ local gradient calls.
\end{thm}

\begin{lem}\label{lem:lowerboundG2scsc}
    For any $\mu, \delta, L > 0$ with $L\geq \max\{\mu,\delta\}$ and $n\geq 2$, there exist $L$-smooth and convex-concave functions $f_1,\dots,f_n:\BR^{d_x}\times\BR^{d_y}$ with $\delta$-second-order similarity such that the function $f(x,y)=\frac{1}{n}\sum_{i=1}^n f_i(x,y)$ is $\mu$-strongly-convex-$\mu$-strongly-concave. In order to find a solution of problem (\ref{prob:main}) such that $\E[\|z-z^*\|^2]\leq\varepsilon$, any DFO algorithm needs at least $\Omega(L/\mu\log(1/\varepsilon))$ local gradient calls.
\end{lem}

\begin{thm}\label{thm:lowerboundGscsc}
    For any $\mu, \delta, L > 0$ with $L\geq \max\{\mu,\delta\}$ and $n\geq 2$, there exist $L$-smooth and convex-concave functions $f_1,\dots,f_n:\BR^{d_x}\times\BR^{d_y}$ with $\delta$-second-order similarity such that the function $f(x,y)=\frac{1}{n}\sum_{i=1}^n f_i(x,y)$ is $\mu$-strongly-convex-$\mu$-strongly-concave. In order to find a solution of problem (\ref{prob:main}) such that $\E[\|z-z^*\|^2]\leq\varepsilon$, any DFO algorithm needs at least \mbox{$\Omega((n+(\sqrt{n}\delta+L)/\mu)\log(1/\varepsilon))$} local gradient calls.
\end{thm}

\subsection{Proof of Theorem \ref{thm:lowerboundVG1scsc}}
We introduce the function set as in \cite{beznosikov2024similarity}, which is similar to equation (\ref{def:ficc2}) that
\begin{align}\label{eq121-badfi}\begin{split}
    f_i(x, y)  =\begin{dcases}
        \frac{\mu}{2}\|x\|^2-\frac{\mu}{2}\|y\|^2+\frac{\delta^2}{16 \mu} e_1^\top  y ,&\quad i=1,\\
        \frac{\delta \sqrt{n}}{4} x^\top \Bigg[\sum_{j \equiv(i-1) \bmod (n-1)} e_j a_j^\top\Bigg] y+\frac{\mu}{2}\|x\|^2-\frac{\mu}{2}\|y\|^2 ,&\quad  i>2.
    \end{dcases}
\end{split}\end{align}

Then corresponding global objective is
\begin{align}\label{eq120-badf}
    f(x, y)=\frac{\delta}{4 \sqrt{n}} x^\top  A y+\frac{\mu}{2}\|x\|^2-\frac{\mu}{2}\|y\|^2+\frac{\delta^2}{16 n \mu} e_1^\top  y,
\end{align}
where all the notation keeps the same as the former section. We point out that the global objective function (\ref{eq120-badf}) with local functions (\ref{eq121-badfi}) satisfies Assumptions \ref{asm:set}, \ref{asm:smooth}, \ref{asm:cc}, \ref{asm:scsc}, and \ref{asm:ss}, with constant $\mu$, $\delta$, and $L\geq \delta$. See Lemma C.8 of \citet{beznosikov2024similarity} for details.

We provide the following lemmas for the proof of Theorem \ref{thm:lowerboundVG1scsc}.

\begin{lem}\label{lem:lower-scsc-000}
 Consider the minimax problem (\ref{prob:cc}) with functions (\ref{eq121-badfi}) and (\ref{eq120-badf}) and run any DFO algorithm with $V$ communication complexity and $C$ local gradient calls. In expectation, only the first $M\leq \min\{2V/n,2C/n\}$ coordinates of the final output can be non-zero while the rest of the $d-M$ coordinates are strictly equal to zero.
\end{lem}
\begin{proof}
Follow the proof of Lemma \ref{lem:coordinate}.    
\end{proof}

%Then we can obtain the lower bound for the idea that the zero coordinates lead us to an error. The method for proof is similar to Lemmas C.6, C.7, and Theorem C.8 of  \citet{kovalev2022optimal}.
\begin{lem}[{\citet[Theorem C.10]{beznosikov2024similarity}}]\label{lem:lower-scsc-001}
 Let $\mu, \delta > 0$, $n\in\BN$, $M\in\BN$. There exists a centralized distributed saddle-point problem with functions (\ref{eq121-badfi}) and (\ref{eq120-badf}), in which $z^*\neq 0$ and $d \geq \max\{2\log_q(\alpha/(4\sqrt{2})),2M\}$ where $q=(2+\alpha-\sqrt{\alpha^2+4\alpha})/2\in(0,1)$ and $\alpha=16n\mu^2/\delta^2$. Then for any output $(\hat{x},\hat{y})$ of any DFO algorithm leaving $d-M$ coordinates zero, one can obtain the following estimate:
    \begin{align*}
        \|\hat{x}-x^*\|^2+\|\hat{y}-y^*\|^2=\Omega\left(\exp\left(-\frac{16}{1+\sqrt{\frac{\delta^2}{16\mu^2n}+1}}\cdot M\right)\|y_0-y^*\|^2\right).
    \end{align*}
\end{lem}

Then we can prove Theorem \ref{thm:lowerboundVG1scsc} by combining Lemma \ref{lem:lower-scsc-000} and \ref{lem:lower-scsc-001} and noting that to reach a solution of $\varepsilon$-accuracy requires
\begin{align*}
    \min\left\{\frac{2V}{n}, \frac{2C}{n}\right\}\geq M=\Omega\left(\left(1+\frac{\delta}{\sqrt{n}\mu}\right)\log\frac{1}{\varepsilon}\right).
\end{align*}

\subsection{Proof of Lemma \ref{lem:lowerboundG2scsc}}

We introduce the function set as 
\begin{align}\label{eq123-badf}
    f(x,y)=f_i(x,y)=\frac{L}{2}x^\top \tilde{A}y+\frac{\mu}{2}\|x\|^2-\frac{\mu}{2}\|y\|^2+\frac{L^2}{4\mu}e_1^\top y, \quad \forall i\in [n],
\end{align}
where
\begin{align*}
    \tilde{A} = \begin{pmatrix}
 &    &  & 1 \\
 &    & 1 &-1  \\
 &    \rotatebox[origin=c]{90}{$\ddots$}
  & \rotatebox[origin=c]{90}{$\ddots$} &  \\
1 & -1   &  & 
\end{pmatrix}.
\end{align*}

Note that since all the nodes have the same function, this function set (\ref{eq123-badf}) satisfies Assumptions \ref{asm:set}, \ref{asm:smooth}, \ref{asm:cc}, \ref{asm:scsc}, and \ref{asm:ss} for $L,\mu $ and any $\delta>0$. The structure of $\tilde{A}$ results the following lemma.
\begin{lem}\label{lem:zero-chain3}
    For the function set (\ref{eq123-badf}), all $(x,y)\in \fF_{k}\times\fF_{k}$ and $k=0,\dots,d-1$, we have 
    \begin{align*} 
        \nabla f_i(x,y)\in \fF_{k+1}\times\fF_{k+1}.
    \end{align*}
\end{lem}

\begin{lem}\label{lem:lowerboundG2modified}
Let $\mu,\delta,L>0$, $n\in\BN$, $C\in\BN$. There exists a centralized distributed saddle-point problem with functions (\ref{eq123-badf}), in which $z^*\neq 0$ and $d\geq \max\{2\log_q(\alpha/\sqrt{2}),4C\}$ where $q=1+2\alpha-2\sqrt{\alpha^2+\alpha}\in(0,1)$ and $\alpha=\mu^2/L^2$. Then for any DFO algorithm, the output $(x^C,y^C)$ after $C$ local gradient calls will satisfies
\begin{align}\label{eq:lowerzhang}
    \|y^C-y^*\|^2\geq q^C\cdot\frac{\|y^0-y^*\|^2}{16}.
\end{align}
\end{lem}
\begin{proof}
    Follow the proof of Theorem 3.5 by \citet{zhang2022lower}, in which we take $L_x=L_y=L_{xy}=L$ and $\mu_x=\mu_y=\mu$.
\end{proof}

Then we can prove Lemma \ref{lem:lowerboundG2scsc} by combining Lemma \ref{lem:zero-chain3} and \ref{lem:lowerboundG2modified} and noting that to reach a solution with $\varepsilon$-accuracy needs at least $\Omega(L/\mu\log(1/\varepsilon))$ local gradient calls from equation (\ref{eq:lowerzhang}).

\subsection{Proof of Theorem \ref{thm:lowerboundGscsc}}
\begin{proof}
In the case of $n\mu+\sqrt{n}\delta\geq \Omega(L)$, the problem with with functions (\ref{eq121-badfi}) and (\ref{eq120-badf}) in Theorem \ref{thm:lowerboundVG1scsc} implies the lower bound on the local gradient complexity of 
\begin{align*}
    \Omega\left(\left(n+\frac{\sqrt{n}\delta}{\mu}\right)\log\frac{1}{\varepsilon}\right)=\Omega\left(\left(n+\frac{\sqrt{n}\delta+L}{\mu}\right)\log\frac{1}{\varepsilon}\right).
\end{align*}
In the case of $L\geq \Omega (n\mu+\sqrt{n}\delta)$, the problem with function (\ref{eq123-badf}) in Lemma \ref{lem:lowerboundG2scsc} implies the lower bound on the local gradient complexity of $\Omega(L/\mu\log(1/\varepsilon))=\Omega((n+(\sqrt{n}\delta+L)/\mu)\log(1/\varepsilon))$. 
Combining these two cases, we can achieve the lower bound on the local gradient complexity of $\Omega((n+(\sqrt{n}\delta+L)/\mu)\log(1/\varepsilon))$.
\end{proof}

\section{Making the Gradient Small}\label{appx:gds}

In constrained case, we can also use gradient mapping to measure the sub-optimality of a solution $z$, that is
\begin{align*}
    \ffF_\tau(z)=\frac{z-\fP_\fZ (z-\tau F(z))}{\tau}.
\end{align*}
For the $L$-smooth convex-concave function $f$, we consider the constrained problem
\begin{align*}
    \min_{x\in\fX} \max_{y\in\fY} \hat f(x,y) := f(x,y) + \frac{\lambda}{2}\norm{x-x^0}^2 -  \frac{\lambda}{2}\norm{y-y^0}^2,
\end{align*}
where $\hat f$ is $\lambda$-strongly-convex-$\lambda$-strongly-concave and $\fZ$ is bounded by diameter $D$.
From the the result of Theorem \ref{thm:convergenceSVOGSscsc}, we have $ \E\big[\norm{z^K - \hat z^*}^2\big] \leq (1-\chi\lambda/\delta)^K\norm{z^0 - \hat z^*}^2$ for some constant $\chi\in(0,1)$, then we have
\begin{align*}
     \E\left[\|\ffF_\tau(z^K)\|\right]&=\E\left\|\frac{z^K-\fP_\fZ\left(z^K-\tau(\hat F(z^K) - \lambda(z^K-z^0))\right)}{\tau}\right\|\\
    &=\E\left\|\frac{\fP_\fZ(z^K)-\fP_\fZ\left(z^K-\tau(\hat F(z^K) - \lambda(z^K-z^0))\right)}{\tau}\right\|\\
    &\leq \E\left[\norm{\hat F(z^K) - \lambda(z^K-z^0)}\right]\\
     &\leq \E\left[\norm{\hat F(z^K)}\right]  + \lambda\E\left[\norm{z^K-z^0}\right] \\
&= \E\left[\norm{\hat F(z^K)-\hat F(\hat z^*)}\right] + \lambda\E\left[\norm{z^K-z^0}\right] \\
 &\leq L\E\left[\norm{z^K-\hat z^*}\right] + \lambda\E\left[\norm{z^K-z^0}\right],
\end{align*}
where we use the property of the projection that $\norm{\fP_\fZ(z_1)-\fP_\fZ(z_2)} \leq \norm{z_1-z_2}$ and the gradient mapping holds that $ \ffF_\tau(\hat z^*)=0$. 

Then we have
\begin{align*}
    \E\left[ \norm{\ffF(z^K)}^2\right]  &\leq  2L^2\E\left[ \norm{z^K-\hat z^*}^2\right] + 2\lambda^2\E\left[ \norm{z^K-z^0}^2\right] \\
 &\leq 2L^2(1-\chi\lambda/\delta)^K\norm{z^0 - \hat z^*}^2 + 2\lambda^2\E\left[ \norm{z^K-z^0}^2\right] \\
 &\leq (2L^2(1-\chi\lambda/\delta)^K + 2\lambda^2)D^2.
\end{align*}

Let $\lambda = \sqrt{\varepsilon/(4D^2)}$ and $K=\fO(\delta/\lambda\log(L/\varepsilon))$, then we have $\E\big[ \norm{\ffF(z^K)}^2\big] \leq \varepsilon$. 

Hence, the complexity of communication rounds is $K= \fO\left({\delta D}/{\sqrt{\varepsilon}}\log({L}/{\varepsilon})\right)$. 
We verify that the corresponding communication complexity is $\tilde\fO((n+\sqrt{n}\delta D/\sqrt{\varepsilon})\log(1/\varepsilon))$ and the local gradient complexity is $\tilde\fO((n+\sqrt{n}L D/\sqrt{\varepsilon})\log(1/\varepsilon))$ by following the discussion in Appendix \ref{appx:proofscsc}.

\section{Experimental Setting}\label{appx:para}
%The optimization is carried out within a centralized distributed setting with 500 nodes, one of which serves as the master node. 

We solve the sub-problem in SVOGS (Algorithm \ref{alg:SVOGS}), SMMDS \cite{beznosikov2021distributed}, EGS \cite{kovalev2022optimalgradient}, and TPAPP \cite{beznosikov2024similarity} by Extra-Gradient method of \citet{korpelevich1976extragradient}. 
We implement all of the methods by Python 3.9 with NumPy and run on a machine with AMD Ryzen(TM) 7 4800H 8 core with Radeon Graphics 2.90 GHz CPU with 16GB RAM.

\end{document}